\documentclass[11pt,a4paper,reqno]{amsart}
\pdfoutput=1

\usepackage{times}
\usepackage[margin=1.0in,tmargin=0.9in,bmargin=0.80in]{geometry}
\usepackage[dvipsnames]{xcolor}
\usepackage[english]{babel}

\definecolor{bred}{rgb}{0.8,0,0}

\usepackage{xltabular}
\usepackage{subfigure}
\usepackage{amsmath,amssymb,graphicx,epsfig}
\usepackage{mathtools}
\usepackage{enumerate,amsthm,epstopdf}
\usepackage[flushleft]{threeparttable}
\usepackage{multirow}
\usepackage{longtable}
\usepackage[utf8]{inputenc}
\usepackage[T1]{fontenc}
\usepackage{bbm}
\usepackage{xargs}
\usepackage{latexsym}
\usepackage{wasysym}
\usepackage{float}
\usepackage{hyperref}
\usepackage{url}
\usepackage{booktabs}
\usepackage{amsfonts}
\usepackage{nicefrac}
\usepackage{microtype}
\usepackage{cleveref}
\usepackage[textsize=footnotesize]{todonotes}
\hypersetup{colorlinks,
	linkcolor={blue},
	citecolor={bred},
	urlcolor={blue}
	}
\usepackage[numbers]{natbib}
\usepackage{indentfirst}

\newcommand{\e}{\mathbb{E}}
\newcommand{\II}{\mathbbm{1}}
\newcommand{\p}{\mathbb{P}}
\newcommand{\dd}{\mathrm{d}}
\newcommand{\la}{\langle}
\newcommand{\ra}{\rangle}
\newcommand{\bt}{\lfloor t \rfloor}
\newcommand{\bs}{\lfloor s \rfloor}
\newcommand{\br}{\lfloor r \rfloor}
\newcommand{\ut}{\lceil t \rceil}
\newcommand{\us}{\lceil s \rceil}

\theoremstyle{theorem}
\newtheorem{theorem}{\textit{\textbf}{Theorem}}
\newtheorem{proposition}{\textbf{Proposition}}
\newtheorem{lemma}{\textbf{Lemma}}
\newtheorem{corollary}{\textbf{Corollary}}
\newtheorem{remark}{\textbf{Remark}}
\newtheorem{definition}{\textbf{Definition}}
\newtheorem{assumption}{\textbf{Assumption}}

\allowdisplaybreaks

\title{Tamed Stochastic Gradient Hamiltonian Monte Carlo}
\author[Z. Wang]{Zhuoran Wang}
\author[Y. Zhang]{Ying Zhang}

\address{Financial Technology Thrust, Society Hub, The Hong Kong University of Science and Technology (Guangzhou), Guangzhou, China}
\email{zwang104@connect.hkust-gz.edu.cn}

\address{Financial Technology Thrust, Society Hub, The Hong Kong University of Science and Technology (Guangzhou), Guangzhou, China}
\email{yingzhang@hkust-gz.edu.cn}

\thanks{Financial support by the Guangzhou-HKUST(GZ) Joint Funding Program (No. 2025A03J3322) is gratefully acknowledged.}

\begin{document}
	
\maketitle

\begin{abstract}
	In this paper, we propose a novel tamed stochastic gradient Hamiltonian Monte Carlo (tSGHMC) algorithm for sampling and stochastic optimization problems with superlinearly growing stochastic gradients. Under a certain continuity in average condition and a strong convexity condition, we establish a non-asymptotic error bound in Wasserstein-2 distance for tSGHMC with the rate of convergence equal to $1/4$. Then, we derive an upper estimate for the associated expected excess risk, which provides a theoretical guarantee for the performance of tSGHMC. To illustrate the effectiveness of the proposed algorithm, we apply tSGHMC to practical examples, including a newsvendor problem and a Conditional Value-at-Risk  minimization problem, using synthetic and real-world datasets. Numerical results support our theoretical findings. Furthermore, we compare tSGHMC with its first-order counterpart, namely, the tamed unadjusted stochastic Langevin algorithm. Simulation results demonstrate that tSGHMC achieves lower root mean square error and expected excess risk across a range of tasks.
\end{abstract}

\section{Introduction}

We consider the stochastic optimization problem:
\begin{equation}\label{OPT}
	\text{minimize}\quad
	\mathbb{R}^d\owns\theta\mapsto
	u(\theta)\coloneq \e[U(\theta,X)],
\end{equation}
where  $U\colon\mathbb{R}^d\times\mathbb{R}^k\to\mathbb{R}$ is a measurable function, and $X$ is an $\mathbb{R}^k$-valued random variable with known probability law $\mathcal{P}$. Our objective is to construct an estimator $\hat{\theta}$ that minimizes the expected excess risk
\[\e[u(\hat{\theta})]-\inf_{\theta\in\mathbb{R}^d}u(\theta).\]
A common sampling-based approach for solving \eqref{OPT} is to consider the associated Gibbs measure $\pi_\beta(\theta)\propto\exp(-\beta u(\theta))$ with $\beta>0$. As $\beta$ increases, $\pi_\beta$ concentrates around the minimizers of $u$ \cite{hwang1980laplace}. Consequently, samples drawn from $\pi_\beta$ for sufficiently large $\beta$ can be used to construct approximate solutions to \eqref{OPT} \cite{dalalyan2017further,dalalyan2017theoretical}.

A standard approach to sampling from $\pi_\beta$ is to use Langevin Monte Carlo (LMC)  \cite{ermak1975computer,parisi1981correlation} or stochastic gradient Langevin dynamics (SGLD) \cite{welling2011bayesian}, as both algorithms can be viewed as Euler discretizations of the following overdamped Langevin stochastic differential equation (SDE):
\begin{equation}\label{eq:overdamped}
	\bar{L}_0=\theta_0,\quad
	\dd \bar{L}_t
	=-h(\bar{L}_t)\,\dd t+\sqrt{2\beta^{-1}}\,\dd W_t,
	\quad t\geq 0,
\end{equation}
where $\theta_0$ is an $\mathbb{R}^d$-valued random variable, $h\coloneq\nabla u$, and $(W_t)_{t\geq 0}$ is a $d$-dimensional Brownian motion, which, under suitable conditions, admits $\pi_\beta$ as its unique invariant distribution \cite{pavliotis2014stochastic}. The convergence properties of LMC and SGLD have been extensively studied in the literature \cite{barkhagen2021stochastic,chau2021stochastic,cheng2018convergence,cheng2018sharp,chewi2025analysis,dalalyan2017further,dalalyan2017theoretical,durmus2017nonasymptotic,durmus2019high,erdogdu2021convergence,mou2022improved,neufeld2024robust,raginsky2017non,sabanis2020fully,vempala2019rapid,welling2011bayesian,xu2018global,zhang2023nonasymptotic}, providing theoretical guarantees for their performance.

An alternative way of solving the sampling task is to use the kinetic Langevin Monte Carlo (KLMC) algorithm \cite{cheng2018underdamped}, or its variant, the stochastic gradient Hamiltonian Monte Carlo (SGHMC) algorithm \cite{chen2014stochastic}, both of which are developed based on the underdamped (or kinetic) Langevin SDE:
\begin{align}\label{eq.underdamped}
	\begin{split}
		\widetilde{V}_0&=\nu_0,\quad\dd \widetilde{V}_t=-[\gamma \widetilde{V}_t+h(\widetilde{Z}_t)]\,\dd t+\sqrt{2\gamma\beta^{-1}}\,\dd W_t,\\
		\widetilde{Z}_0&=\theta_0,\quad\dd \widetilde{Z}_t=\widetilde{V}_t\,\dd t,\quad t\geq 0,
	\end{split}
\end{align}
where $(\theta_0,\nu_0)$ is an $\mathbb{R}^{2d}$-valued random variable, $(\widetilde{Z}_t,\widetilde{V}_t)_{t\geq 0}$ are the position and velocity processes, respectively, and $\gamma>0$ is the friction coefficient. Under mild conditions, the underdamped Langevin SDE \eqref{eq.underdamped} admits an invariant measure given by
\begin{equation}
	\Pi_\beta(\theta,\nu)\propto\exp\left(-\beta\left(u(\theta)+\frac{|\nu|^2}{2}\right)\right),
\end{equation}
whose marginal distribution in the position variable $\theta$ is exactly $\pi_\beta$ \cite{pavliotis2014stochastic}. While non-asymptotic convergence results have been established for KLMC and SGHMC under various assumptions \cite{akyildiz2024nonasymptotic,altschuler2026shifted,chau2022stochastic,chen2014stochastic,cheng2018underdamped,dalalyan2020sampling,gao2022global,liang2024non,liu2020improved,yu2024langevin,zhang2026dimension}, they all require $h$ to be globally Lipschitz continuous. Such an assumption excludes key applications in modern machine learning, particularly those involving highly nonlinear models or neural networks. In addition, for superlinearly growing gradients, standard Euler discretizations of \eqref{eq.underdamped} are unstable, as their absolute moments may diverge to infinity in finite time \cite{hutzenthaler2011strong}. To address the instability issue, \cite{dai2025explicit,johnston2024kinetic,lytras2025contractive} introduced tamed variants of KLMC by applying the taming technique \cite{beyn2016stochastic,hutzenthaler2012tame,sotirios2013tame,wang2013tamed}, and derived convergence guarantees under the polynomial Lipschitz condition of $h$. Nevertheless, these algorithms utilize the exact gradient $h$ at each iteration, leading to substantial computational costs in large-scale applications. This thus motivates the development of a tamed KLMC (tKLMC) algorithm that incorporates the associated stochastic gradient.

\subsection{Contributions}

In this work, we propose a novel tamed SGHMC (tSGHMC) algorithm for sampling and stochastic optimization with superlinear and discontinuous stochastic gradients. The proposed algorithm is given precisely by
\begin{align}\label{eq.tSGHMC-intro}
	\begin{split}
		\nu_0^\lambda=\nu_0,\quad
		\nu_{n+1}^\lambda&=\nu_n^\lambda-\lambda[\gamma \nu_n^\lambda+H_\gamma(\theta_n^\lambda,X_{n+1})]
		+\sqrt{2\lambda\gamma\beta^{-1}}\,\xi_{n+1},\\
		\theta_0^\lambda=\theta_0,\quad
		\theta_{n+1}^\lambda&=\theta_n^\lambda+\lambda \nu_n^\lambda,\quad n\in\mathbb{N}_0,
	\end{split}
\end{align}
where $\lambda>0$ is the step size, $(X_n)_{n\in\mathbb{N}}$ are $\mathbb{R}^k$-valued independent and identically distributed (i.i.d.) random variables with probability law $\mathcal{P}$, $(\xi_n)_{n\in\mathbb{N}}$ are $\mathbb{R}^d$-valued i.i.d.\ standard Gaussian random variables, and $H_\gamma:\mathbb{R}^d\times\mathbb{R}^k\to\mathbb{R}^d$ is the tamed stochastic gradient given by
\begin{equation}\label{eq.tamed-gradient}
	H_\gamma(\theta,x)\coloneq m\theta+
	\frac{H(\theta,x)-m\theta}
	{\sqrt{1+\gamma^{-1}|\theta|^{4r}}},
\end{equation}
for all $\theta\in\mathbb{R}^d$, $ x\in\mathbb{R}^k$, with $m,r>0$ and $H$ denoting the stochastic gradient which satisfies $h(\theta)=\e[H(\theta,X_0)]$ for all $\theta\in\mathbb{R}^d$. We note that tSGHMC \eqref{eq.tSGHMC-intro}-\eqref{eq.tamed-gradient} incorporates a carefully designed taming coefficient $H_\gamma$, preventing the numerical instability caused by the superlinear $H$, while retaining the low computational cost of stochastic gradient updates. We refer to Remark~\ref{remark.tamed} for more details regarding the construction of $H_\gamma$. To the best of the authors' knowledge, this is the first such variant of SGHMC in the literature.

For tSGHMC, we provide a non-asymptotic convergence bound in Wasserstein-2 distance, which is derived from a novel result on the moment estimates of the proposed algorithm. Such a convergence result further yields an upper estimate for the expected excess risk, ensuring the applicability of tSGHMC on optimization problems in the form of \eqref{OPT}. Crucially, these results are obtained under relaxed conditions including a polynomial continuity in average condition (Assumption~\ref{assumption2.local Lip+growth}) and a data-dependent strong convexity condition (Assumption~\ref{assumption3.strong convexity}), consequently, tSGHMC can be used to solve a wide range of applications with superlinearly growing and discontinuous stochastic gradients, as well as those involving online data where a uniform bound on the data stream is unavailable.

Another key contribution of this work is a theoretical improvement over the pioneer study \cite{johnston2024kinetic}. Under a suitable choice of parameters, we can show that the convergence result for tSGHMC in Wasserstein-2 distance has a rate of convergence equal to $1/4$, improving upon the order $1/5$ rate of convergence established for tKLMC in  \cite{johnston2024kinetic}. This enhancement is achieved by exploiting key properties of auxiliary processes introduced in Section~\ref{subsection.auxiliary}. As a result, by choosing $\gamma=\mathcal{O}(\lambda^{-1/4})$, we derive a $\lambda\gamma=\mathcal{O}(\lambda^{3/4})$ error term in Proposition~\ref{proposition: W_2 - interpolation and auxiliary}  instead of $\sqrt{\lambda\gamma}=\mathcal{O}(\lambda^{2/5})$ (with $\gamma=\mathcal{O}(\lambda^{-1/5})$) in \cite[Proposition~8.10]{johnston2024kinetic}, leading to the improvement of $\sqrt{\lambda}\gamma=\mathcal{O}(\lambda^{1/4})$ in Theorem~\ref{Main Theorem} compared to $\sqrt{\lambda}\gamma^{3/2}=\mathcal{O}(\lambda^{1/5})$ in \cite[Theorem 7.1]{johnston2024kinetic}. We refer to Remark~\ref{remark.rate1} for further details regarding the interplay between $\lambda$ and $\gamma$.

Finally, we investigate the empirical performance of tSGHMC in various examples relevant in practice, including, e.g., posterior sampling, newsvendor problem, Conditional Value-at-Risk (CVaR) minimization, and nonlinear regression. Numerical results on synthetic and real-world datasets demonstrate that tSGHMC can effectively address the aforementioned sampling and optimization problems. To illustrate the superior performance of tSGHMC, we compare it with the tamed unadjusted stochastic Langevin algorithm (TUSLA) \cite{lim2024non,lovas2023taming}, which can be viewed as the first-order counterpart of tSGHMC. For most experiments, tSGHMC converges faster and achieves higher accuracy than TUSLA in terms of root mean square error (RMSE) and expected excess risk, while exhibiting comparable performance in the remaining cases.

\subsection*{Notations}

At the end of this section, we introduce some notation. The set of nonnegative integers is denoted by $\mathbb{N}_0\coloneq \{0,1,2,3,\cdots\}$, while $\mathbb{N}\coloneq\mathbb{N}_0\backslash\{0\}$.
For a positive real number $a$, denote its integer part by $\lfloor a\rfloor$ and define $\lceil a\rceil = \lfloor a\rfloor + 1$.
The closed ball centered at $a$ with radius $r$ is denoted by $\bar{B}(a,r)$.
The Euclidean inner product is denoted by $\la\cdot,\cdot\ra$, where $|\cdot|$ denotes the corresponding Euclidean norm (the dimension of the space may vary depending on the context).
For matrices, $\mathbb{R}^{d\times d}$ denotes the set of $d\times d$ real matrices. $\operatorname{Hess}(u)(\theta)$ denotes the Hessian matrix of the function $u(\theta)$ with $\theta\in\mathbb{R}^d$, $I_d$ denotes the identity matrix in $\mathbb{R}^{d\times d}$, and $\operatorname{diag}(a_1,\dots,a_d)$ denotes the diagonal matrix whose diagonal entries are $a_1,\dots,a_d$. 
For two symmetric matrices $A,B\in\mathbb{R}^{d\times d}$, we write $A\preceq B$, or equivalently $B\succeq A$, if $B-A$ is positive
semidefinite.
Let $(\Omega,\mathcal{F},\p)$ denote a probability space. The expectation of a random variable $Z$ is denoted by $\e[Z]$. For $1 \leq p < \infty$, $L^p$ denotes the space of real-valued random variables that are $p$-integrable. The indicator function of a set $A$ is denoted by $\II_A$. For a random variable $Z$ taking values in $\mathbb{R}^d$, its distribution on $\mathcal{B}(\mathbb{R}^d)$, the Borel $\sigma$-field of $\mathbb{R}^d$, is denoted by $\mathcal{L}(Z)$. For any integer $q\geq 1$, let $\mathcal{P}(\mathbb{R}^q)$ denote the set of probability measures on $\mathcal{B}(\mathbb{R}^q)$. For $\mu,\nu\in\mathcal{P}(\mathbb{R}^d)$, let $\mathcal{C}(\mu,\nu)$ denote the set of probability measures $\zeta$ defined on $\mathcal{B}(\mathbb{R}^{2d})$ with marginals $\mu$ and $\nu$. For two Borel probability measures $\mu$ and $\nu$ defined on $\mathbb{R}^d$ with finite $q$-th moments, the Wasserstein distance of order $q\geq 1$ is defined by
\[W_q(\mu,\nu)  \coloneq  \left(\inf_{\zeta\in\mathcal{C}(\mu,\nu)}\int_{\mathbb{R}^d}\int_{\mathbb{R}^d}|\theta-\theta'|^q\,\zeta(\dd\theta\dd\theta')\right)^{1/q}.\]

\section{Assumptions and Main Results}
\label{section.Assumptions and Main Results}

\subsection{Setting}

Fix $d,k\in\mathbb{N}$. Let $U\colon\mathbb{R}^d\times\mathbb{R}^k\to\mathbb{R}$ be a measurable function, and assume that $\e[|U(\theta,X)|]<\infty$ holds for every $\theta\in\mathbb{R}^d$, where $X$ is an $\mathbb{R}^k$-valued random variable with probability law $\mathcal{L}(X)$.
Let $u\colon\mathbb{R}^d\to\mathbb{R}_{\geq 0}$ be a nonnegative twice continuously differentiable function defined by $u(\theta)\coloneq\e[U(\theta,X)]$, and let $h\coloneq \nabla u$ denote its gradient. Let $\theta^*$ denote a minimizer of $u$, i.e., $h(\theta^*)=0$.
Fix $\beta>0$, for all $A\in\mathcal{B}(\mathbb{R}^d)$, define the probability measures
\begin{align}
	\pi_\beta(A)&\coloneq \frac{\int_A \exp(-\beta u(\theta))\,\dd\theta}{\int_{\mathbb{R}^d}\exp(-\beta u(\theta))\,\dd\theta},\label{eq.pi_beta}\\
	\Pi_\beta(A)&\coloneq \frac{\int_A\exp\left(-\beta\left(u(\theta)+|\nu|^2/2\right)\right)\,\dd\theta\dd\nu}{\int_{\mathbb{R}^d}\exp\left(-\beta\left(u(\theta)+|\nu|^2/2\right)\right)\,\dd\theta\dd\nu},\label{eq.Pi_beta}
\end{align}
where $\int_{\mathbb{R}^d}\exp(-\beta u(\theta))\,\dd\theta<\infty$. Let $(\mathcal{G}_n)_{n\in\mathbb{N}_0}$ be a filtration describing the evolution of available information, and set $\mathcal{G}_\infty\coloneq\sigma(\bigcup_{n\in\mathbb{N}_0}\mathcal{G}_n)$. Consider a $(\mathcal{G}_n)$-adapted sequence $(X_n)_{n\in\mathbb{N}_0}$ of i.i.d.\ $\mathbb{R}^k$-valued random variables with distribution $\mathcal{L}(X)$. Let $(\xi_n)_{n\in\mathbb{N}_0}$ denote a sequence of independent standard Gaussian vectors in $\mathbb{R}^d$.

The \textbf{tamed stochastic gradient Hamiltonian Monte Carlo} (tSGHMC) algorithm $(\theta_n^\lambda,\nu_n^\lambda)_{n\in\mathbb{N}_0}$ is defined by
\begin{align}\label{tSGHMC}
	\begin{split}
		\nu_0^\lambda=\nu_0,\quad
		\nu_{n+1}^\lambda&=\nu_n^\lambda-\lambda[\gamma \nu_n^\lambda+H_\gamma(\theta_n^\lambda,X_{n+1})]
		+\sqrt{2\lambda\gamma\beta^{-1}}\,\xi_{n+1},\\
		\theta_0^\lambda=\theta_0,\quad
		\theta_{n+1}^\lambda&=\theta_n^\lambda+\lambda \nu_n^\lambda,\quad n\in\mathbb{N}_0,
	\end{split}
\end{align}
where $\theta_0$, $\nu_0$ are $\mathbb{R}^d$-valued random variables, $\lambda>0$ denotes the step size, $\gamma>0$ is the friction parameter, $\beta>0$ is the inverse temperature, and where for any $\theta\in\mathbb{R}^d$ and $x\in\mathbb{R}^k$,
\begin{equation}\label{eq.tamedSG}
	H_\gamma(\theta,x)=m\theta+\frac{H(\theta,x)-m\theta}{\sqrt{1+\gamma^{-1}|\theta|^{4r}}}\end{equation}
with $m,r>0$, and $H\colon\mathbb{R}^d\times\mathbb{R}^k\to\mathbb{R}^d$ satisfying $h(\theta)=\e[H(\theta,X_0)]$ for all $\theta\in\mathbb{R}^d$. Throughout, we assume that the initial conditions $\theta_0,\nu_0$, the sigma-field $\mathcal{G}_\infty$, and the Gaussian sequence $(\xi_n)_{n\in\mathbb{N}_0}$ are mutually independent.

\begin{remark}\label{remark.tamed}
	We note that tSGHMC \eqref{tSGHMC}-\eqref{eq.tamedSG} can be obtained by applying the taming technique \cite{beyn2016stochastic,hutzenthaler2012tame,sotirios2013tame,wang2013tamed} to the SGHMC algorithm introduced in \cite{chen2014stochastic}. Inspired by \cite{johnston2024kinetic,lytras2025taming,neufeld2025non}, we choose the tamed coefficient $H_\gamma$ in the current form \eqref{eq.tamedSG}, as it satisfies a dissipativity condition, a linear growth condition (in $\theta$), and possesses a certain convergence property as presented in Lemmas~\ref{Lemma.dissipativity of Expectation of SG}, \ref{lemma.2nd bound of taming factor}, and~\ref{lemma.mse - H_lambda and H}, respectively. These properties are key in deriving moment estimates of tSGHMC and establishing its convergence result in Wasserstein-2 distance.
\end{remark}

\subsection{Assumptions}\label{Section. Assumptions}

Let $\vartheta\in[1,\infty)$, $2r\in[\vartheta,\infty)\cap\mathbb{N}$, and $\rho\in[1,\infty)$ be fixed. The assumptions are listed below.

\begin{assumption}\label{assumption1.initial condition}
	The initial condition $(\theta_0,\nu_0)$ satisfies $\e[|\theta_0|^{12r}]+\e[|\nu_0|^{12r}]<\infty$. The process $(X_n)_{n\in\mathbb{N}_0}$ has a finite $12\rho r$-th moment, i.e., $\e[|X_0|^{12\rho r}]<\infty$. Furthermore, we have that $h(\theta)=\e[H(\theta,X_0)]$, for all $\theta\in\mathbb{R}^d$.
\end{assumption}

In the following assumption, we impose a (joint) locally Lipschitz condition and a suitable growth condition on the function $F$. In addition, we assume that $G$ satisfies a ``continuity in average'' condition, which is weaker than locally Lipschitz continuity.

\begin{assumption}\label{assumption2.local Lip+growth}
	The function $H\colon \mathbb{R}^d\times\mathbb{R}^k\to\mathbb{R}^d$ takes the form of
	\begin{equation}\label{eq.expression of H}
		H(\theta,x)=F(\theta,x)+G(\theta,x),
	\end{equation}
	where
	\begin{enumerate}
		\item $F\colon \mathbb{R}^d\times\mathbb{R}^k\to\mathbb{R}^d$ satisfies that there exists a constant $L_F>0$ such that, for all $\theta,\theta'\in\mathbb{R}^d$, $x,x'\in\mathbb{R}^k$,
		\[|F(\theta,x)-F(\theta',x')|\leq L_F(1+|x|+|x'|)^{\rho-1}(1+|\theta|+|\theta'|)^{2r-1}(|\theta-\theta'|+|x-x'|).\]
		Furthermore, there exists a constant $K_F>0$ such that for all $\theta\in\mathbb{R}^d$, $x\in\mathbb{R}^k$,
		\[|F(\theta,x)|\leq K_F(1+|x|)^\rho(1+|\theta|^{2r}).\]
		\item There exists a constant $L_G>0$ such that, for all $\theta,\theta'\in\mathbb{R}^d$,
		\[\e[|G(\theta,X_0)-G(\theta',X_0)|]\leq L_G(1+|\theta|+|\theta'|)^{\vartheta-1}|\theta-\theta'|.\]
		In addition, there exists a constant $K_G>1$ such that for all $\theta\in\mathbb{R}^d$ and $x\in\mathbb{R}^k$,
		\[|G(\theta,x)|\leq K_G(1+|x|)^\rho(1+|\theta|)^\vartheta.\]
	\end{enumerate}
\end{assumption}

\begin{remark}
	We consider stochastic gradients $H$ defined in \eqref{eq.expression of H}, where $F$ is locally Lipschitz continuous and $G$ may be discontinuous. This structure is common in applications with nonsmooth stochastic gradients, such as quantile estimation, CVaR minimization, regularized optimization with ReLU neural networks, see, e.g., \cite{bardou2008computation,daouia19extreme,lim2024non,lim2025langevin}, and the newsvendor problem as presented in Section~\ref{subsection.newsvendor}.
\end{remark}

\begin{remark}\label{remark.local Lip. of h}
	By Assumptions~\ref{assumption1.initial condition} and \ref{assumption2.local Lip+growth}, we obtain that $\e[F(\theta,X_0)]$ and $\e[G(\theta,X_0)]$ are well-defined. Moreover, for all $\theta\in\mathbb{R}^d$ and $x\in\mathbb{R}^k$, 
	\begin{equation}\label{eq.upper bound of |H|}
		|H(\theta,x)|\leq K_H(1+|x|)^\rho(1+|\theta|^{2r}),
	\end{equation}
	where $K_H\coloneq 2^{2r-1}K_G+K_F$. Besides, we note that $h$ is locally Lipschitz continuous, i.e., there exists a constant $L_h>0$ such that for all $\theta,\theta'\in\mathbb{R}^d$,
	\[|h(\theta)-h(\theta')|\leq L_h(1+|\theta|+|\theta'|)^{2r-1}|\theta-\theta'|,\]
	where $L_h\coloneq L_G+L_F\e[(1+2|X_0|)^{\rho-1}]$.
\end{remark}

We consider the following variant of the strong convexity condition, which introduces the dependence on the data stream.

\begin{assumption}\label{assumption3.strong convexity}
	There exists a measurable symmetric matrix-valued function $A\colon \mathbb{R}^k\to\mathbb{R}^{d\times d}$ such that for all $(\theta,x)\in\mathbb{R}^d\times\mathbb{R}^k$,
	\begin{equation}\label{eq.positive semi definited A}
		\la\theta,A(x)\theta\ra\geq 0,
	\end{equation}
	and, for all $\theta,\theta'\in\mathbb{R}^d$ and $x\in\mathbb{R}^k$,
	\begin{equation}\label{eq.local strong convexity}
		\la\theta-\theta',H(\theta,x)-H(\theta',x)\ra\geq\la\theta-\theta',A(x)(\theta-\theta')\ra.
	\end{equation}
	Moreover, the smallest eigenvalue of the matrix $\e[A(X_0)]$ is strictly positive and is denoted by $m$.
\end{assumption}

\begin{remark}
	The local convexity condition \eqref{eq.positive semi definited A}-\eqref{eq.local strong convexity} in Assumption~\ref{assumption3.strong convexity} is originally proposed in \cite{barkhagen2021stochastic}. Such a condition relaxes the strong convexity condition, i.e., for all $\theta,\theta'\in\mathbb{R}^d$ and $x\in\mathbb{R}^k$,
	\[\la\theta-\theta',H(\theta,x)-H(\theta',x)\ra\geq \bar{a}|\theta-\theta'|^2\]
	for some $\bar a> 0$, by allowing non-uniform dependence on the data stream, which accommodates, e.g., applications with online data.
\end{remark}

The following result shows that Assumption~\ref{assumption3.strong convexity} implies a strong convexity and a dissipativity condition of $u$.

\begin{remark}\label{remark.dissipativity of h}
	By Assumptions~\ref{assumption1.initial condition} and \ref{assumption3.strong convexity}, $u$ is $m$-strongly convex, i.e., for all $\theta,\theta'\in\mathbb{R}^d$,
	\[\la \theta-\theta',h(\theta)-h(\theta')\ra\geq m|\theta-\theta'|^2.\]
	Hence, we obtain, for all $\theta\in\mathbb{R}^d$, that
	\[u(0)\geq u(\theta)-\la h(\theta),\theta\ra+\frac{m}{2}|\theta|^2,\]
	which indicates that
	\[\la h(\theta),\theta\ra\geq\frac{m}{2}|\theta|^2-u(0).\]
	Furthermore, we have that
	\[\frac{m}{2}|\theta^*|^2-u(0)\leq 0,\]
	implying $\theta^*\in\bar{B}(0,R_0)$ with $R_0\coloneq\sqrt{2u(0)/m}$.
\end{remark}

The following remark shows that $H$ defined in \eqref{eq.expression of H} satisfies a dissipativity condition.

\begin{remark}\label{remark.dissipativitiy of SG}
	Let Assumptions~\ref{assumption1.initial condition}-\ref{assumption3.strong convexity} hold. Then, for all $\theta\in\mathbb{R}^d$ and $x\in\mathbb{R}^k$,
	\[\la\theta,H(\theta,x)\ra\geq \la\theta,\widetilde{A}(x)\theta\ra-\widetilde{b}(x),\]
	where $\widetilde{A}(x)=A(x)-m I_d$ and $\widetilde{b}(x)=K_H^2(1+|x|)^{2\rho}/(4m)$.
\end{remark}

\begin{proof}
	Postponed to Appendix \hyperref[Proof.dissipative of SG]{A}.
\end{proof}

\subsection{Main Results}

Define
\begin{equation}\label{eq.friction restriction}
	\gamma_{\min}\coloneq\max\left\{\sqrt{K+\frac{m}{2}},\sqrt{\frac{2(K+2|h(0)|)}{3}},6\sqrt{6}K,1,14m,\sqrt{m},32C_m\right\},
\end{equation}
where $K$ is defined in Remark~\ref{remark.MY}, and $C_m$ is explicitly given in Table~\ref{tab.constants}. Besides, for all $q\in\mathbb{N}$, we define
\begin{equation}\label{eq.step-size restriction}
	\lambda_{q,\max,\gamma}\coloneq\min\left\{\frac{1}{4\gamma},\frac{3m}{16q(q-1)2^{q-2}\gamma}\right\},\quad\lambda_{\max,\gamma}\coloneq\lambda_{12r,\max,\gamma}.
\end{equation}

The following theorem establishes a non-asymptotic bound in Wasserstein-2 distance between the distribution of $\theta_n^\lambda$ defined in \eqref{tSGHMC} and the target distribution $\pi_\beta$.

\begin{theorem}\label{Main Theorem}
	Let Assumptions~\ref{assumption1.initial condition}-\ref{assumption3.strong convexity} hold. Then, for all $\epsilon>0$ and $\beta>0$, there exists $\dot{C}>0$ such that, for all $\gamma\geq\gamma_{\min}$, $\lambda\leq\lambda_{\max,\gamma}$, and $n\in\mathbb{N}_0$,
	\begin{equation}\label{eq.main estimates}
		W_2(\mathcal{L}(\theta_n^\lambda),\pi_\beta)\leq \dot{C}\left(\sqrt{\lambda}\gamma+\lambda\gamma^{5/2}+\gamma^{-1}+\exp\left(-\frac{\lambda m}{2\gamma}n\right)W_2(\mathcal{L}(\theta_0,\nu_0),\Pi_\beta)+\epsilon\right),
	\end{equation}
	where $\dot{C}=\mathcal{O}((d/\beta)^{3r+1/2})$ is independent of
	$n,\lambda,\gamma$ and $\epsilon$. Its explicit expression is provided in Table~\ref{tab.constants}.
\end{theorem}

\begin{remark}\label{remark.rate1}
	We note that to make the right-hand side of \eqref{eq.main estimates} fall below a given precision level, we cannot choose freely the value of $\epsilon$, instead, $\epsilon$ and $\gamma$ should be selected in a coupled way. As a concrete example, we set $\epsilon=\lambda^{1/4}$ and consider the case where $\gamma=\mathcal{O}(\epsilon^{-1})$. Then
	\[W_2(\mathcal{L}(\theta_n^\lambda),\pi_\beta)\leq 5\dot{C}\left(\lambda^{1/4}+\exp\left(-\frac{\lambda^{5/4}m}{2}n\right)W_2(\mathcal{L}(\theta_0,\nu_0),\Pi_\beta)\right),\]
	which can be made arbitrarily small by choosing sequentially $\lambda$ and $n$. Moreover, we note that with the above choice of parameters, tSGHMC converges in Wasserstein-2 distance with the rate equal to $1/4$, which improves that achieved for tKLMC in \cite{johnston2024kinetic}. Furthermore, for any $\varepsilon>0$, if we choose
	\begin{align*}
		\lambda&\leq\min\left\{(\varepsilon/(10\dot C))^4,\lambda_{\max,\gamma}\right\},\\
		n&\geq \frac{2}{m}\max\left\{(10\dot{C}/\varepsilon)^5,\lambda_{\max,\gamma}^{-5/4}\right\}\log\left(\frac{10\dot{C}W_2(\mathcal{L}(\theta_0,\nu_0),\Pi_\beta)}{\varepsilon}\right),
	\end{align*}
	then we obtain $W_2(\mathcal{L}(\theta_n^\lambda),\pi_\beta)\leq\varepsilon$.
\end{remark}

Let $(\theta_n^\lambda,\nu_n^\lambda)$ be generated by tSGHMC \eqref{tSGHMC}-\eqref{eq.tamedSG}, and set $\hat{\theta}=\theta_n^\lambda$. Then, by using Theorem~\ref{Main Theorem}, the following result shows that tSGHMC can be used to solve the optimization problem:
\[\text{minimize}\quad\mathbb{R}^d\owns\theta\mapsto u(\theta)\coloneq\e[U(\theta,X)],\]
by providing an upper bound for the expected excess risk.

\begin{theorem}\label{Theorem.Optimization}
	Let Assumptions~\ref{assumption1.initial condition}-\ref{assumption3.strong convexity} hold. Then, for all $\epsilon>0$ and $\beta>0$, there exist $C',C''>0$ such that, for all $\gamma\geq\gamma_{\min}$, $\lambda\leq\lambda_{\max,\gamma}$, and $n\in\mathbb{N}_0$, 
	\[\e [u(\theta_n^\lambda)]-u(\theta^*)\leq C'\left(\sqrt{\lambda}\gamma+\lambda\gamma^{5/2}+\gamma^{-1}+\exp\left(-\frac{\lambda m}{2\gamma}n\right)W_2(\mathcal{L}(\theta_0,\nu_0),\Pi_\beta)+\epsilon\right)+C''\frac{2d}{m\beta},\]
	where $C'=\mathcal{O}((d/\beta)^{4r+1/2})$ and $C''=\mathcal{O}((d/\beta)^{r+1/2})$ are independent of $n,\lambda$, $\gamma$ and $\epsilon$. Their explicit expressions are provided in Table~\ref{tab.constants}.
\end{theorem}

\begin{remark}\label{remark.precision-expected}
	By the same argument as in Remark~\ref{remark.rate1}, setting $\epsilon=\lambda^{1/4}$ and choosing $\gamma=\mathcal{O}(\epsilon^{-1})$, we obtain
	\[\e[u(\theta_n^\lambda)]-u(\theta^*)\leq 5C'\left(\lambda^{1/4}+\exp\left(-\frac{\lambda^{5/4}m}{2}n\right)W_2(\mathcal{L}(\theta_0,\nu_0),\Pi_\beta)\right)+C''\frac{2d}{m\beta}.\]
	Consequently, for any $\epsilon>0$, if we choose
	\[\beta\geq\max\left\{1,\frac{6\bar C''d}{m\varepsilon}\right\},\quad\lambda\leq\min\left\{(\varepsilon/(15C'))^4,\lambda_{\max,\gamma}\right\},\]
	where $\bar{C}''$ is independent of $\beta,n,\lambda,\gamma,\epsilon$ with its explicit expression given in Table~\ref{tab.constants}, and
	\[n\geq\frac{2}{m}\max\left\{(15C'/\varepsilon)^5,\lambda_{\max,\gamma}^{-5/4}\right\}\log\left(\frac{15C' W_2(\mathcal{L}(\theta_0,\nu_0),\Pi_\beta)}{\varepsilon}\right),\]
	then we have $\e[u(\theta_n^\lambda)]-u(\theta^*)\leq\varepsilon$.
\end{remark}

The proofs of the main results can be found in Section~\ref{Proof of the Main Theorems}.

\subsection{Related Works}

We classify existing results according to the smoothness assumptions imposed on the gradient. We begin with kinetic Langevin discretizations under global Lipschitz conditions and then consider the regime of superlinearly growing gradients.

To establish convergence results, many existing works on KLMC impose certain global Lipschitz conditions on $h$. Under a strong convexity condition of $u$, \cite[Theorem~1]{cheng2018underdamped}, \cite[Theorem~2]{dalalyan2020sampling} and \cite[Theorem~6]{zhang2023improved} showed that KLMC achieves an $\varepsilon$-precision level in Wasserstein-2 distance within $\widetilde{\mathcal O}(\sqrt d/\varepsilon)$ iterations. Under the conditions that $\pi_\beta$ satisfies a log-Sobolev inequality (LSI) and $\nabla^2 u$ is globally Lipschitz, \cite[Theorem~1]{ma2021there} obtained an $\varepsilon$-precision level guarantee in Wasserstein-2 distance with the same iteration complexity.

Non-asymptotic guarantees for SGHMC have also been developed under global Lipschitz conditions on $H$  \cite{akyildiz2024nonasymptotic,chau2022stochastic,gao2022global,liang2024non}. When $U$ satisfies additionally a dissipativity condition, \cite[Lemma~EC.3]{gao2022global} derived a Wasserstein-2 bound of order $\mathcal{O}((\delta^{1/4}+\lambda^{1/4})(n\lambda)^{3/2}\sqrt{\log(n\lambda)}+(n\lambda)\sqrt{\lambda})$, with $\delta\in[0,1)$ being the gradient noise level, yet this bound deteriorates with the number of iterations $n$. By controlling $\delta$, SGHMC achieves an $\varepsilon$-precision level after $\widetilde{\mathcal{O}}(\varepsilon^{-4})$ iterations. In the same setting, \cite[Theorem~2.8]{chau2022stochastic} improved the convergence result in \cite{gao2022global} by proving time-uniform Wasserstein-$p$ error bounds of order $\mathcal{O}(\delta^{1/(2p)}+\lambda^{1/(2p)})$ with $p=1,2$. This yields an $\varepsilon$-precision guarantee in Wasserstein-$p$ distances by choosing $\lambda+\delta=\mathcal{O}(\varepsilon^{2p})$, whose corresponding iteration complexity is $\widetilde{\mathcal{O}}(\varepsilon^{-2p})$. Such a result is further improved in \cite{akyildiz2024nonasymptotic} under local smoothness assumptions and a local dissipativity condition. 
\cite[Theorem~2.1]{akyildiz2024nonasymptotic} established a convergence result in Wasserstein-2 distance with a rate of $\mathcal{O}(\lambda^{1/4})$, which yields the iteration complexity of $\widetilde{\mathcal{O}}(\varepsilon^{-4})$. More recently, \cite{liang2024non} extended the theory to the case of discontinuous stochastic gradients under a continuity in average condition. While \cite[Theorem~3.1]{liang2024non} recovered the same Wasserstein-2 convergence result as in \cite{akyildiz2024nonasymptotic}, \cite[Theorem~3.2]{liang2024non} further proved a sharper Wasserstein-1 bound of order $\mathcal{O}(\lambda^{1/2})$, implying an $\varepsilon$-precision guarantee within $\widetilde{\mathcal{O}}(\varepsilon^{-2})$ iterations.

To move beyond global Lipschitz continuity, several results in the literature have introduced tamed or structure-preserving kinetic Langevin schemes for superlinearly growing gradients. In particular, \cite{johnston2024kinetic,lytras2025contractive} assume that $u$ is strongly convex and that $h$ is polynomially Lipschitz. Under these conditions, \cite{johnston2024kinetic} proposed the tKLMC algorithm, which is obtained by applying the taming technique to KLMC. \cite[Theorem~7.1]{johnston2024kinetic} showed that it requires $\widetilde{\mathcal{O}}(\varepsilon^{-5})$ iterations for tKLMC to reach an $\varepsilon$-precision level in Wasserstein-2 distance. Furthermore, \cite{lytras2025contractive} introduced two monotonicity preserving taming schemes of kinetic Langevin SDE, namely the tamed stochastic exponential (tSE) scheme and the tamed OBABO (tOBABO) scheme. \cite[Theorem~3 and 6]{lytras2025contractive} derived $\varepsilon$-precision level guarantees in Wasserstein-2 distance of $\widetilde{\mathcal{O}}(\varepsilon^{-5/3})$ and $\widetilde{\mathcal{O}}(\varepsilon^{-5})$ for tSE and tOBABO, respectively. Moreover, when $u$ is non-convex, \cite{dai2025explicit} proposed an explicit splitting scalar auxiliary variable (SSAV) scheme, which preserves a modified energy and achieves first-order strong convergence over finite time intervals. By further assuming a dissipativity condition and exponential ergodicity, \cite[Theorem~5.6]{dai2025explicit} established a long time weak error bound for SSAV on the sampling tasks, leading to the complexity $\widetilde{\mathcal{O}}(\varepsilon^{-1})$ at $\varepsilon$-precision level for expectations of suitable test functions.

In this work, we propose the tSGHMC algorithm \eqref{tSGHMC}-\eqref{eq.tamedSG} which utilizes the stochastic gradient, hence, it is more computationally efficient than the existing algorithms that use the exact gradient. Moreover, we establish non-asymptotic convergence results in the Wasserstein-2 distance and for the expected excess risk, providing theoretical guarantees for tSGHMC to solve sampling and optimization problems. We note that the rate of convergence in Wasserstein-2 distance for tSGHMC is of order $1/4$, which improves the rate of order $1/5$ for tKLMC in \cite{johnston2024kinetic}. As a consequence, by setting $\epsilon=\lambda^{1/4}$ and by taking $\gamma=\mathcal{O}(\epsilon^{-1})$, Theorem~\ref{Main Theorem} yields an iteration complexity of $\widetilde{\mathcal{O}}(\varepsilon^{-5})$, see Remark~\ref{remark.rate1}. These results are obtained under relaxed conditions, i.e., Assumptions~\ref{assumption2.local Lip+growth} and \ref{assumption3.strong convexity}, which can accommodate a broad range of real-world applications as presented in Section~\ref{section.Applications}.

{\small\begin{table*}[t]
	\centering
	\caption{Comparison of iteration complexities between Theorem~\ref{Main Theorem} and existing results for achieving an $\varepsilon$-precision level in Wasserstein-2 distance.*}
	\label{tab:comparison}
	\renewcommand{\arraystretch}{1.35}
	\setlength{\tabcolsep}{4pt}
	\begin{threeparttable}
		\begin{tabular}{cccccc}
			\toprule
			Work & Algorithm & Complexity & Smoothness & Curvature & Others \\
			\midrule
			\cite{cheng2018underdamped,dalalyan2020sampling,zhang2023improved}
			& \multirow{2}{*}{KLMC} & $\widetilde{\mathcal{O}}(\sqrt{d}/\varepsilon)$
			& G-L $h$ & S-C & - \\
			\cite{ma2021there}
			&  & $\widetilde{\mathcal{O}}(\sqrt{d}/\varepsilon)$
			& G-L $h$, $\nabla^2 u$ & - & LSI \\
			\hline
			\cite{chau2022stochastic,gao2022global}
			& \multirow{3}{*}{SGHMC} & $\widetilde{\mathcal{O}}(\varepsilon^{-4})$
			& G-L $H$ & - & Dissip. \\
			\cite{akyildiz2024nonasymptotic}
			&  & $\widetilde{\mathcal{O}}(\varepsilon^{-4})$
			& D-G-L $H$ & - & D-Dissip. \\
			\cite{liang2024non}
			&  & $\widetilde{\mathcal{O}}(\varepsilon^{-4})$
			& A-C $H$ & - & D-Dissip. \\
			\hline
			\cite{johnston2024kinetic}
			& tKLMC & $\widetilde{\mathcal{O}}(\varepsilon^{-6})$
			& P-L $h$ & S-C & - \\
			\hline
			\multirow{2}{*}{\cite{lytras2025contractive}}
			& tSE & $\widetilde{\mathcal{O}}(\varepsilon^{-5/3})$
			& P-L $h$ & S-C & - \\
			& tOBABO & $\widetilde{\mathcal{O}}(\varepsilon^{-5})$
			& P-L $h$ & S-C & - \\
			\hline
			\cite{dai2025explicit}$^\star$
			& SSAV & $\widetilde{\mathcal{O}}(\varepsilon^{-1})$
			& P-L $h$ & - & Dissip., EE \\
			\hline
			Theorem~\ref{Main Theorem}
			& tSGHMC & $\widetilde{\mathcal{O}}(\varepsilon^{-5})$
			& P-A-C $H$ & D-S-C & - \\
			\bottomrule
		\end{tabular}
		\begin{tablenotes}[flushleft]
			\footnotesize
			\item * G-L: global Lipschitzness; D-G-L: data-dependent global
			Lipschitzness; A-C: continuity in average; P-L: polynomial
			Lipschitzness; P-A-C: polynomial continuity in average; S-C: strong
			convexity; D-S-C: data-dependent strong convexity; Dissip.: dissipativity;
			D-Dissip.: data-dependent dissipativity; EE: exponential ergodicity.
			\item $\star$ The iteration complexity of SSAV is established for the weak error.
		\end{tablenotes}
	\end{threeparttable}
\end{table*}}

\section{Applications}
\label{section.Applications}

In this section, we apply the tSGHMC algorithm \eqref{tSGHMC}-\eqref{eq.tamedSG} to vrious sampling and optimization examples, including a posterior sampling problem in Section~\ref{subsection.sampling}, an artificial optimization example in Section~\ref{subsection.artificial}, a newsvendor problem in Section~\ref{subsection.newsvendor}, and two real-world applications in Section~\ref{subsection.real-world}. We show numerically that tSGHMC can solve the aforementioned problems, which supports our theoretical findings. Furthermore, we compare the performance of tSGHMC with that of the TUSLA algorithm \cite{lim2024non,lovas2023taming}. Numerical results illustrate that tSGHMC outperforms TUSLA in terms of training accuracy in most cases, while, in other cases, tSGHMC achieves comparable performance to that of TUSLA. The Python code for the experiments can be found at \url{https://github.com/ZhuoranWangCn/tSGHMC}.

\subsection{Posterior Sampling for Penalized Logistic Regression}
\label{subsection.sampling}

Let $d\in \mathbb{N}$ be fixed. We consider the posterior sampling problem associated with a penalized logistic regression model discussed in \cite{yu2024langevin}. More precisely, we consider the problem of sampling from $\pi_\beta(\theta)\propto\exp(-\beta u(\theta))$ with 
\begin{align}\label{eq.sampling-obj}
	u(\theta)\coloneq
	\frac{1}{\mathsf{N}}\sum_{i=1}^\mathsf{N} U(\theta,z_i)=\frac{1}{\mathsf{N}}\sum_{i=1}^\mathsf{N}\log\left(1+\exp(-y_i\la x_i,\theta\ra)\right)+\frac{\eta}{2}|\theta|^2,
\end{align}
where $\theta\in\mathbb{R}^d$, $\eta>0$, $\mathsf{N}$ is the sample size, and $\{z_i=(x_i,y_i)\}_{i=1}^\mathsf{N}$ denotes the observed dataset consisting of feature vectors $x_i\in\mathbb{R}^d$ and binary labels $y_i\in\{-1,1\}$.

We aim to solve the problem using tSGHMC \eqref{tSGHMC}-\eqref{eq.tamedSG}. To this end, define $H\colon\mathbb{R}^d\times(\mathbb{R}^d\times\{-1,1\})^\mathsf{K}\to\mathbb{R}^d$ by $H\coloneq F+G$, where, for all $(\theta,\bar z)\coloneq(\theta,\bar x_1,\bar y_1,\dots,\bar x_\mathsf{K},\bar y_\mathsf{K})\in\mathbb{R}^d\times(\mathbb{R}^d\times\{-1,1\})^\mathsf{K}$, we set 
\begin{equation}\label{eq:Sampling-SG}
	F(\theta,\bar z)\coloneq\eta\theta,\quad G(\theta,\bar z)\coloneq-\frac{1}{\mathsf{K}}\sum_{i=1}^\mathsf{K}\frac{\bar y_i \bar x_i}{1+\exp(\bar y_i\la \bar x_i,\theta\ra)}.
\end{equation}
Therefore, the stochastic gradient $H$ used in tSGHMC is given by $H(\theta,U_{{\mathbf{z}}})= F(\theta,U_{{\mathbf{z}}})+G(\theta,U_{{\mathbf{z}}})$, where
\begin{equation}\label{eq:Sampling-SGt}
	F(\theta,U_{{\mathbf{z}}})= \eta\theta,\quad G(\theta,U_{{\mathbf{z}}})= -\frac{1}{\mathsf{K}}\sum_{l=1}^\mathsf{K}\frac{y_{I_l}x_{I_l}}{1+\exp(y_{I_l}\la x_{I_l},\theta\ra)}.
\end{equation}
Here, $U_{\mathbf{z}}\coloneq\{(x_{I_1},y_{I_1}),\dots,(x_{I_\mathsf{K}},y_{I_\mathsf{K}})\}$ denotes the mini-batch, where the indices $I_1,\dots,I_\mathsf{K}$ are sampled independently and uniformly from $\{1,\ldots,\mathsf{N}\}$, with $\mathsf{K}\ll \mathsf{N}$.

\begin{proposition}\label{prop.Sampling}
	The function $H$ defined in \eqref{eq:Sampling-SG}-\eqref{eq:Sampling-SGt} satisfies Assumptions~\ref{assumption2.local Lip+growth} and \ref{assumption3.strong convexity}. Moreover, for all $\theta\in\mathbb{R}^d$, $\e[H(\theta,U_{\mathbf{z}})]=h(\theta)$.
\end{proposition}

\begin{proof}
	Postponed to Appendix \hyperref[proof.Sampling]{A.2}.
\end{proof}

\begin{figure}
	\centering
	\includegraphics[width=1.0\linewidth]{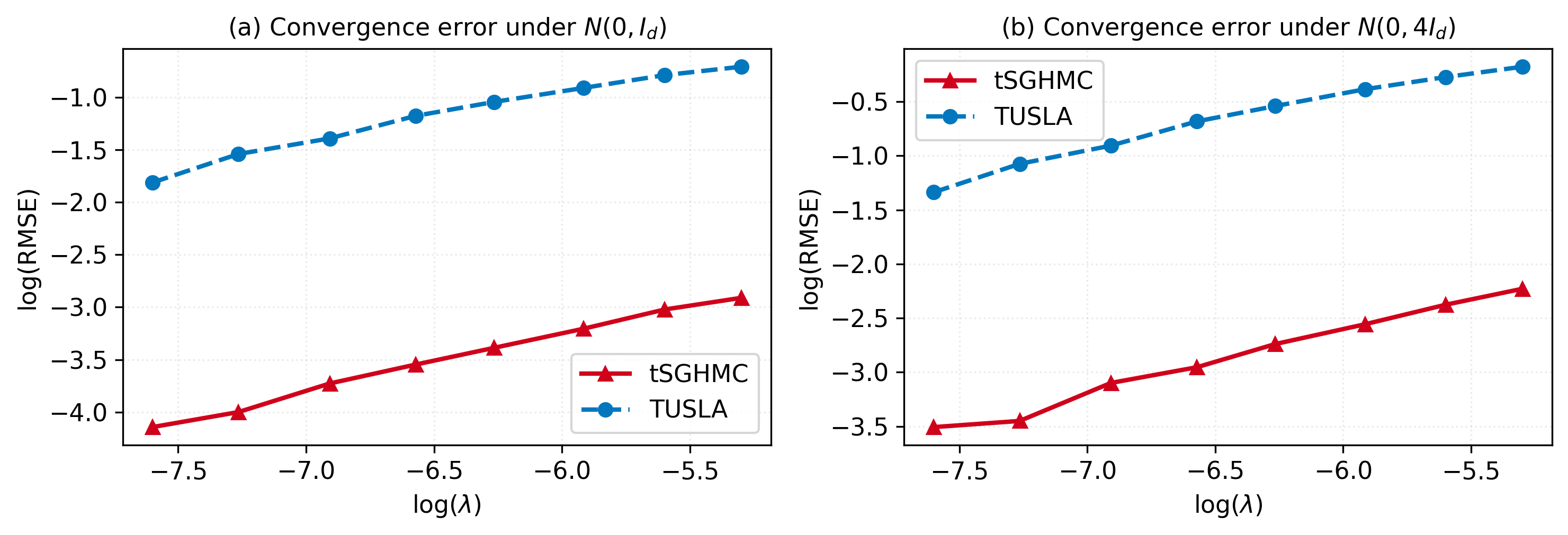}
	\caption{RMSE obtained using tSGHMC and TUSLA for penalized logistic regression with feature vectors generated from $\mathcal{N}(0,I_d)$ in panel (a) and $\mathcal{N}(0,4I_d)$ in panel (b).}
	\label{fig.example1sampling}
\end{figure}

In the numerical experiment, we set $d=3$, $\mathsf{N}=100$, $\eta=10^{-3}$, $T=n\lambda=5$, $\mathsf{K}=1$, $\beta=1$, $r=0.5$, $\gamma=1$ and $m=10^{-3}$. We generate the true parameter vector $\theta^{\text{true}}\in\mathbb{R}^d$ from a standard Gaussian distribution and keep it fixed throughout each experiment. To obtain the dataset $\{(x_i,y_i)\}_{i=1}^\mathsf{N}$, we first draw the feature vectors $\{x_i\}_{i=1}^\mathsf{N}$ from a Gaussian distribution $\mathcal{N}(0,\Sigma)$. Then, for each $x_i$, we sample $B_i\sim\mathrm{Bernoulli}(p_i)$ with parameter
\begin{align*}
	p_i\coloneq\p(B_i=1\mid x_i)=\frac{1}{1+\exp(-\la x_i,\theta^{\text{true}}\ra)},
\end{align*}
and obtain $y_i=2B_i-1$. We consider two experimental settings by taking $\Sigma=I_d$ and $\Sigma=4I_d$, respectively, and generate the corresponding labels. We run tSGHMC 100 times, and for each independent run, tSGHMC is initialized with $\theta_0\sim\mathcal{N}(0,I_d)$ and $\nu_0=0$. By Proposition~\ref{prop.Sampling} and the setup above, Assumptions~\ref{assumption1.initial condition}-\ref{assumption3.strong convexity} hold. Therefore, Theorem~\ref{Main Theorem} can be used to provide a theoretical guarantee for the performance of tSGHMC. 

Furthermore, to evaluate the numerical performance of tSGHMC, we compute the root mean square error (RMSE) defined by
\begin{equation}\label{eq.RMSE-self}
	\operatorname{RMSE}\coloneq \left(\frac{1}{R}\sum_{i=1}^{R}\left|\theta_{\lfloor T/\lambda\rfloor}^{(i)}-\theta_{\lfloor T/\lambda_{\text{ref}}\rfloor}^{(i)}\right|^2\right)^{1/2},
\end{equation}
where $\theta_{\lfloor T/\lambda\rfloor}^{(i)}$ and $\theta_{\lfloor T/\lambda_{\text{ref}}\rfloor}^{(i)}$ denote the samples generated using tSGHMC with step sizes $\lambda$ and $\lambda_{\text{ref}}$, respectively. We note that, by setting $\lambda_{\text{ref}}$ to be a small value, $\theta_{\lfloor T/\lambda_{\text{ref}}\rfloor}^{(i)}$ can be viewed as the exact solution of the corresponding SDE in \eqref{eq.KLD}. We set $R=100$, $\lambda_{\text{ref}}=10^{-4}$, and $\lambda\in\{5\times 10^{-4}, 7\times 10^{-4}, 10^{-3}, 1.4\times 10^{-3}, 1.9\times 10^{-3}, 2.7\times 10^{-3}, 3.7\times 10^{-3}, 5\times 10^{-3}\}$. Figure~\ref{fig.example1sampling} depicts the log-log plot of RMSE against the step sizes $\lambda$ for tSGHMC when $x_i$'s are realizations of $\mathcal{N}(0,I_d)$ and $\mathcal{N}(0,4I_d)$. These results demonstrate the effectiveness of tSGHMC for the sampling from $\pi_\beta$. Moreover, we compare the performance of tSGHMC with that of TUSLA. Numerical results show that tSGHMC consistently achieves a smaller RMSE than TUSLA across all step sizes, indicating its superior accuracy.

\subsection{Artificial Example}
\label{subsection.artificial}

Let $d\in\mathbb{N}$ be fixed. Consider the following optimization problem:
\begin{equation}\label{eq.Artificial-OPT}
	\text{minimize}\quad\mathbb{R}^d \owns\theta\mapsto u(\theta)  \coloneq  \e[U(\theta,X)],
\end{equation}
where $X\sim \mathcal{N}(0,I_d)$ and the loss function $U\colon \mathbb{R}^d\times\mathbb{R}^d\to\mathbb{R}$ is defined by
\begin{equation}\label{eq.Artificial-loss}
	U(\theta,x)  \coloneq  \frac{1}{2}\la\theta,\widetilde{A}\theta\ra+\frac{\eta}{4}\la\theta,\widetilde{B}\theta\ra^2+\la x,\theta\ra.
\end{equation}
Here, $\eta>0$, $\widetilde{A}=\operatorname{diag}(\mu_1,\dots,\mu_d)$ with $\mu_1=1$ and $\mu_2=\dots=\mu_d=\bar{\varepsilon}$, $0<\bar{\varepsilon}<1$, and $\widetilde{B}=vv^\top$ is a matrix generated from a vector $v\in\mathbb{R}^d$. The minimizer of $u$ is $\theta^*=0$.

We solve the aforementioned problem \eqref{eq.Artificial-OPT}-\eqref{eq.Artificial-loss} using tSGHMC \eqref{tSGHMC}-\eqref{eq.tamedSG}.
For all $(\theta,x)\in\mathbb{R}^d\times\mathbb{R}^d$, define the stochastic gradient $H\colon\mathbb{R}^d\times\mathbb{R}^d\to\mathbb{R}^d$ by $H(\theta,x)\coloneq F(\theta,x)+G(\theta,x)$, where
\begin{equation}\label{eq.Artificial-SG}
	F(\theta,x)  \coloneq  \widetilde{A}\theta+x,\quad G(\theta,x)  \coloneq  \eta\la\theta,\widetilde{B}\theta\ra \widetilde{B}\theta.
\end{equation}
Thus, at iteration $n$, the stochastic gradient used in tSGHMC is $H(\theta,X_{n+1})=F(\theta,X_{n+1})+G(\theta,X_{n+1})$ for each $\theta\in\mathbb{R}^d$ with
\begin{equation}\label{eq.Artificial-SGt}
	F(\theta,X_{n+1})=\widetilde{A}\theta+X_{n+1},\quad G(\theta,X_{n+1})=\eta\la\theta,\widetilde{B}\theta\ra \widetilde{B}\theta,
\end{equation}
where $(X_n)_{n\in\mathbb{N}_0}$ is a sequence of i.i.d.\ standard Gaussian vectors.

\begin{proposition}\label{prop.Artificial}
	The stochastic gradient $H$ defined in \eqref{eq.Artificial-SG}-\eqref{eq.Artificial-SGt} satisfies Assumptions~\ref{assumption2.local Lip+growth} and \ref{assumption3.strong convexity}. Moreover, for all $\theta\in\mathbb{R}^d$, $\e[H(\theta,X_0)]=h(\theta)$.
\end{proposition}

\begin{proof}
	Postponed to Appendix \hyperref[proof.Artificial]{A.2}.
\end{proof}

\begin{figure}
	\centering
	\includegraphics[width=1.0\linewidth]{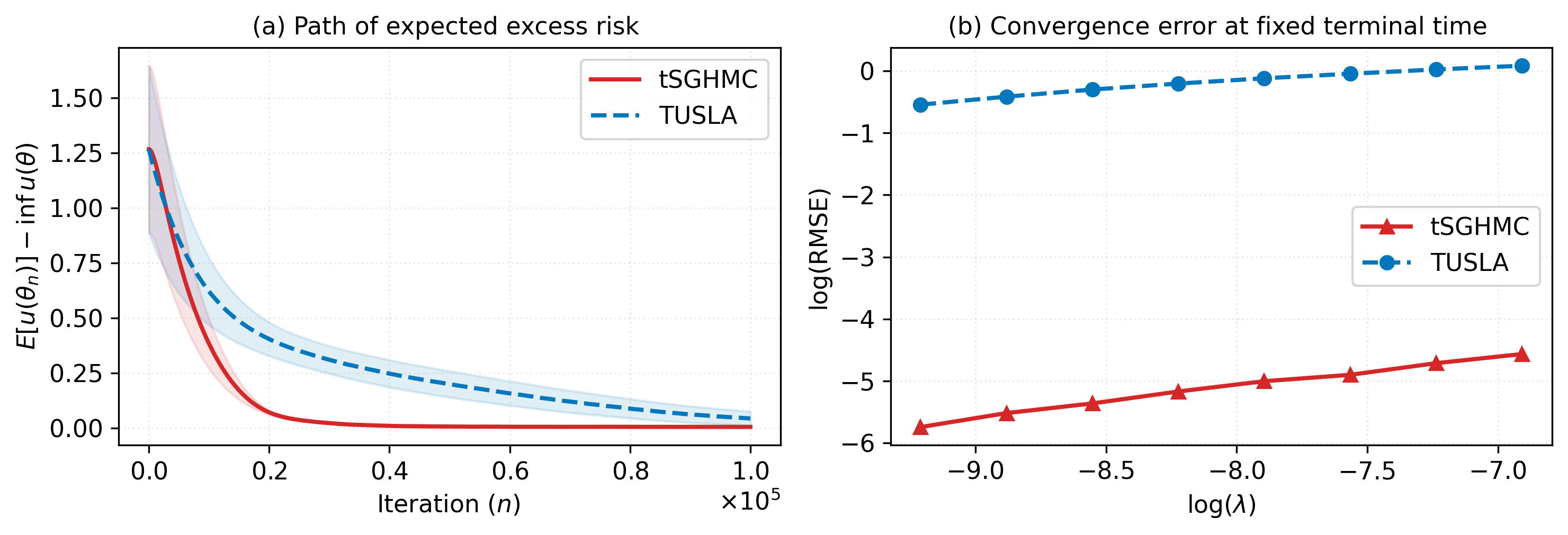}
	\caption{Convergence comparison between tSGHMC and TUSLA for the artificial example \eqref{eq.Artificial-OPT}-\eqref{eq.Artificial-loss}. (a) Path of expected excess risk versus the number of iterations $n$. (b) Log-log plot of the RMSE at $T=10$ for different step sizes $\lambda$.}
	\label{fig.example1artificial}
\end{figure}

In the numerical simulation, we set $d=20$, $\bar{\varepsilon}=0.05$ and $\eta=0.1$. We sample a standard Gaussian vector $\bar{v}$ once at the beginning of the experiment, normalize it to $\bar{v}/|\bar{v}|$, and then keep it fixed throughout the simulation. Let $\widetilde{B}=vv^\top$ with $v=\bar{v}/|\bar{v}|$. We set $\beta=10^{10}$, $r=1.5$, $\gamma=1$ and $m=0.05$, then execute 30 independent runs of tSGHMC, initializing each with $\theta_0\sim\mathcal{N}(0,I_d)$ and $\nu_0=0$. By Proposition~\ref{prop.Artificial} and the preceding setup,
Assumptions~\ref{assumption1.initial condition}-\ref{assumption3.strong convexity}
are satisfied. Hence, Theorems~\ref{Main Theorem} and~\ref{Theorem.Optimization}
apply and provide theoretical guarantees for the performance of tSGHMC.

In the first experiment, we set $\lambda=10^{-3}$ and compute the expected excess risk associated with problem \eqref{eq.Artificial-OPT}-\eqref{eq.Artificial-loss}. Figure~\ref{fig.example1artificial}(a) illustrates the path of expected excess risk over $10^5$ iterations obtained using tSGHMC, with the shaded region representing the sample mean $\pm 1.96$ standard errors at each iteration. We observe that the path converges to zero, which shows that tSGHMC can be used to solve the optimization problem. Moreover, we conduct the same experiment using TUSLA and compare the performance of both algorithms. Numerical results show that tSGHMC converges much faster than TUSLA, in other words, it requires fewer iterations for tSGHMC to fall below a prescribed precision level compared to TUSLA.

In the second experiment, we compute the RMSE defined in \eqref{eq.RMSE-self} using both tSGHMC and TUSLA with the parameters $R=30$, $T=10$, $\lambda_{\text{ref}}=10^{-5}$, and $\lambda\in\{10^{-4}, 1.39\times 10^{-4}, 1.93\times 10^{-4}, 2.68\times 10^{-4}, 3.73\times 10^{-4}, 5.18\times 10^{-4}, 7.20\times 10^{-4}, 10^{-3}\}$. The simulation results are presented in Figure ~\ref{fig.example1artificial}(b), which show the superior performance of tSGHMC in terms of RMSE.

\subsection{Newsvendor Problem}
\label{subsection.newsvendor}

In this section, we investigate a newsvendor problem with nonlinear ordering cost function \cite{arrow1951optimal,halman2012approximating,lau1988newsboy,lau2007designing,petruzzi1999pricing}. Specifically, we consider the following regularized optimization problem:
\begin{equation}\label{eq.Newsvendor-OPT}
	\text{minimize}\quad\mathbb{R}^d\owns\theta\mapsto u(\theta)  \coloneq  \e[U(\theta,X)]+\frac{\kappa}{4}|\theta|^4,
\end{equation}
where $\kappa>0$, $\theta\in\mathbb{R}^d$ denotes the inventory decision vector for $d$ products, and $X\sim\mathcal{N}(\mu,\Sigma_X)$ is an $\mathbb{R}^d$-valued random demand vector. The loss function $U\colon \mathbb{R}^d\times\mathbb{R}^d\to\mathbb{R}$ is defined by
\begin{equation}\label{eq.Newsvendor-Loss}
	U(\theta,x)\coloneq\frac{1}{2}\la\theta, \bar{A}\theta\ra+\sum_{i=1}^d \Big(h_i\max\{\theta_i-x_i,0\}+ s_i\max\{x_i-\theta_i,0\}\Big),
\end{equation}
where $\bar{A}\in\mathbb{R}^{d\times d}$ is a symmetric positive definite matrix, and $h_i,s_i>0$ for all $i=1,\dots,d$. The first term on the right-hand side of \eqref{eq.Newsvendor-Loss} represents the ordering cost, while the second and third terms correspond to the holding and shortage costs of the $i$-th product, respectively.

We solve the newsvendor problem \eqref{eq.Newsvendor-OPT}-\eqref{eq.Newsvendor-Loss} using tSGHMC \eqref{tSGHMC}-\eqref{eq.tamedSG}. For all $(\theta,x)\in\mathbb{R}^d\times\mathbb{R}^d$, define $H\colon\mathbb{R}^d\times\mathbb{R}^d\to\mathbb{R}^d$ by $H(\theta,x)\coloneq F(\theta,x)+G(\theta,x)$ with
\begin{equation}\label{eq.Newsvendor-SG}
	F(\theta,x)\coloneq \bar{A}\theta+\kappa|\theta|^2\theta,\quad G(\theta,x)\coloneq\sum_{i=1}^d\left(h_i\II_{\{\theta_i>x_i\}}-s_i\II_{\{x_i>\theta_i\}}\right)e_i,
\end{equation}
where $e_i$ is the unit vector with the $i$-th coordinate equal to $1$. Thus, at iteration $n$, tSGHMC uses the stochastic gradient $H(\theta,X_{n+1})=F(\theta,X_{n+1})+G(\theta,X_{n+1})$ with
\begin{equation}\label{eq.Newsvendor-SGt}
	F(\theta,X_{n+1})=\bar{A}\theta+\kappa|\theta|^2\theta,\quad G(\theta,X_{n+1})=\sum_{i=1}^d\left(h_i\II_{\{\theta_i>X_{n+1}^i\}}-s_i\II_{\{X_{n+1}^i>\theta_i\}}\right)e_i,
\end{equation}
where $(X_n)_{n\in\mathbb{N}_0}$ is a sequence of i.i.d.\ Gaussian vectors with law $\mathcal{N}(\mu,\Sigma_X)$, and $X_{n+1}^i$ denotes the $i$-th coordinate of $X_{n+1}$.

\begin{proposition}\label{prop.Newsvendor}
	The stochastic gradient $H$ defined in \eqref{eq.Newsvendor-SG}-\eqref{eq.Newsvendor-SGt} satisfies Assumptions~\ref{assumption2.local Lip+growth} and \ref{assumption3.strong convexity}. Moreover, for all $\theta\in\mathbb{R}^d$, $\e[H(\theta,X_0)]=h(\theta)$.
\end{proposition}

\begin{proof}
	Postponed to Appendix \hyperref[proof.Newsvendor]{A.2}.
\end{proof}

\begin{figure}
	\centering
	\includegraphics[width=1.0\linewidth]{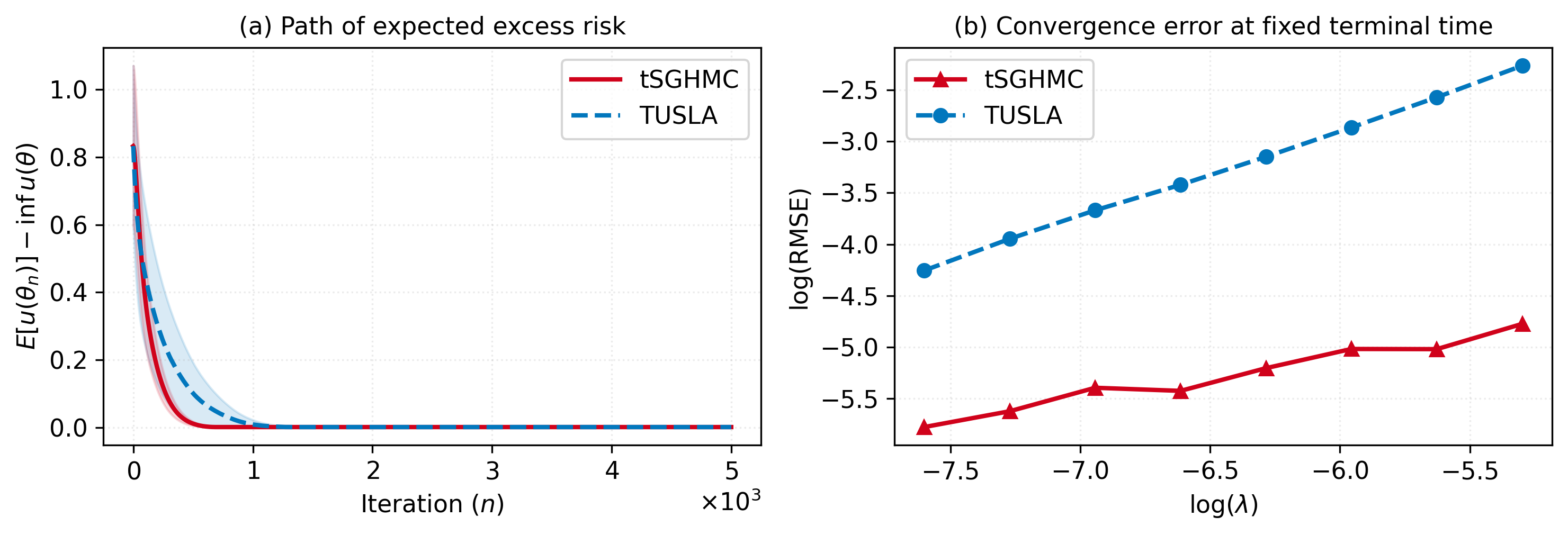}
	\caption{Numerical performance of tSGHMC and TUSLA for the newsvendor problem \eqref{eq.Newsvendor-OPT}-\eqref{eq.Newsvendor-Loss}. (a) Trajectory of expected excess risk versus the number of iterations $n$. (b) Log-log plot of the RMSE at $T=20$ for different step sizes $\lambda$.}
	\label{fig.example2newsvendor}
\end{figure}

To conduct numerical experiments, we set $d=5$, $\kappa=0.02$, $h_i=0.05$, and $s_i=0.1$ for all $i=1,\dots, d$. Moreover, we define $\bar{A}=\eta I_d+\bar\Sigma$, where $\eta=10^{-3}$ and  $\bar\Sigma=\operatorname{diag}(1,1/d^2,\dots,1/d^2)$. We also set $\mu=(-1,-0.5,0,0.5,1)$ and $\Sigma_X=\bar\Sigma$. We run tSGHMC 30 times with $\beta=10^{10}$, $r=1.5$, $\gamma=1$ and $m=0.041$. For each independent run, tSGHMC is initialized from $\theta_0\sim\mathcal{N}(0,I_d)$ and $\nu_0=0$. Given this setup and Proposition~\ref{prop.Newsvendor}, Assumptions~\ref{assumption1.initial condition}-\ref{assumption3.strong convexity} hold. Thus, Theorems~\ref{Main Theorem} and \ref{Theorem.Optimization} theoretically guarantee that tSGHMC solves the newsvendor problem defined by \eqref{eq.Newsvendor-OPT}-\eqref{eq.Newsvendor-Loss}.

Figure~\ref{fig.example2newsvendor}(a) depicts the trajectory of the expected excess risk over 5,000 iterations obtained using tSGHMC with $\lambda=0.05$, where the shaded band represents the sample mean $\pm 1.96$ standard errors at each iteration. We note that the reference solution $\theta^*$ to this problem is computed numerically using the Adam algorithm \cite{kingma2014adam}. The trajectory of the expected excess risk converges to zero, indicating that tSGHMC is effective for solving this optimization problem. We further compare the performance of tSGHMC with that of TUSLA. Numerical results show that tSGHMC converges faster than TUSLA while achieving comparable accuracy.

Next, we compute the RMSE defined in \eqref{eq.RMSE-self} using tSGHMC and TUSLA, where we set $R=30$, $T=20$, $\lambda_{\text{ref}}=10^{-4}$, and $\lambda\in\{5\times 10^{-4},6.95\times 10^{-4},9.65\times 10^{-4},1.34\times 10^{-3}, 1.86\times 10^{-3},2.59\times 10^{-3},3.60\times 10^{-4},5\times 10^{-3}\}$. As shown in Figure~\ref{fig.example2newsvendor}(b), tSGHMC outperforms TUSLA and achieves a higher accuracy measured using RMSE.

\subsection{Real-World Applications}
\label{subsection.real-world}

In this section, we further evaluate the performance of tSGHMC on CVaR minimization and nonlinear regression problems constructed using real-world datasets.

\subsubsection{CVaR Minimization}
\label{subsection.VaR}

In the insurance industry, Value-at-Risk (VaR) and Conditional Value-at-Risk (CVaR) are standard metrics for quantifying extreme tail risks, particularly those arising from large but infrequent claims. These risk measures are crucial for pricing high-excess coverage, allocating risk capital, and designing reinsurance contracts to ensure that extreme payouts remain at an acceptably low frequency. The empirical analysis uses the Danish reinsurance claim dataset \cite{mcneil1997estimating}, available from the CASdatasets package \cite{dutang2026}.\footnote{For details on the Danish reinsurance claim dataset, see \url{https://dutangc.github.io/CASdatasets/reference/danish.html}. General installation instructions for CASdatasets are available at \url{https://dutangc.github.io/CASdatasets/index.html}.} Comprising 2,167 observations, the data points are treated as i.i.d.\ in accordance with the assumptions of tSGHMC in \eqref{tSGHMC}-\eqref{eq.tamedSG}. More precisely, the dataset contains large fire insurance claims with three coverage-specific loss components: building damage, damage to furniture and personal property (i.e., contents), and loss of profits, representing distinct sources of insured losses. In this example, we consider CVaR minimization problems based on these claims. We first investigate a univariate CVaR computation problem, where we use the aggregate loss over these three components, denoted as ``Total'', as the input data. Then, we extend the setting to a risk allocation problem, in which the three coverage-specific loss components are treated as distinct risk exposures and their allocation weights are optimized to account for their different contributions to tail risk.

\paragraph{\textbf{Univariate CVaR Computation}.} Consider the following regularized optimization problem \cite{bardou2008computation,sabanis2020fully}:
\begin{equation}\label{eq.VaR-OPT}
	\text{minimize}\quad\mathbb{R}\owns\theta\mapsto u(\theta)  \coloneq \e[U(\theta,X)]+\frac{\eta}{2}|\theta|^2,
\end{equation}
where $\eta>0$ is the regularization parameter, $X$ is the input random variable, and the loss function $U\colon \mathbb{R}\times\mathbb{R}\to\mathbb{R}$ is given by
\begin{equation}\label{eq.VaR-loss}
	U(\theta,x)\coloneq \theta+\frac{1}{1-\bar{q}}\max\{x-\theta,0\},
\end{equation}
with $0<\bar{q}<1$ denoting the confidence level. By \cite[Proposition 2.1]{bardou2008computation}, $\operatorname{VaR}_{q}(X)=\arg\min_\theta u(\theta)$ and $\operatorname{CVaR}_{q}(X)=\min_\theta u(\theta)$.

We solve the problem \eqref{eq.VaR-OPT}-\eqref{eq.VaR-loss} using tSGHMC, and we define the stochastic gradient $H\colon \mathbb{R}\times\mathbb{R}\to\mathbb{R}$ by $H(\theta,x)\coloneq F(\theta,x)+G(\theta,x)$, where
\begin{equation}\label{eq.stochastic gradient of VaR}
	F(\theta,x)  \coloneq  1+\eta\theta,\quad G(\theta,x)  \coloneq -\frac{1}{1-\bar{q}}\II_{\{x>\theta\}}.
\end{equation}

\paragraph{\textbf{Risk Allocation Problem}.} According to \cite{sabanis2020fully}, the optimization problem can be formulated as follows:
\begin{equation}\label{eq.Risk-OPT}
	\text{minimize}\quad\mathbb{R}^{d+1}\owns\bar\theta\mapsto \bar{u}(\bar{\theta})\coloneq\e[\bar{U}(\bar\theta,X)]+\frac{\eta}{2(r+1)}|\bar{\theta}|^{2(r+1)},
\end{equation}
where $\eta>0$, $r\in\mathbb{N}$, $X$ is an $\mathbb{R}^d$-valued input random variable, and $\bar{U}\colon \mathbb{R}^{d+1}\times\mathbb{R}^d\to\mathbb{R}$ is defined as
\begin{equation}\label{eq.Risk-loss}
	\bar{U}(\bar\theta,x)\coloneq\theta+\frac{1}{1-\bar{q}}\max\left\{\sum_{i=1}^d g_i(\omega)x_i-\theta,0\right\},
\end{equation}
with $\bar{\theta}\coloneq(\theta,\omega)=(\theta,\omega_1,\dots,\omega_d)$, $0<\bar{q}<1$, and $g_i(\omega)\coloneq e^{w_i}/(\sum_{k=1}^d e^{w_k})$ for $i=1,\dots,d$.

To solve the problem \eqref{eq.Risk-OPT}-\eqref{eq.Risk-loss} using tSGHMC, we denote by $\bar H\colon \mathbb{R}^{d+1}\times\mathbb{R}^d\to\mathbb{R}^{d+1}$ the stochastic gradient defined by $\bar H(\bar\theta,x)\coloneq \bar F(\bar\theta,x)+\bar G(\bar\theta,x)$ with
\begin{align*}
	\bar F(\bar{\theta},x)&\coloneq(\bar F_\theta(\bar{\theta},x),\bar F_{\omega_1}(\bar{\theta},x),\dots,\bar F_{\omega_d}(\bar{\theta},x)),\\
	\bar G(\bar{\theta},x)&\coloneq(\bar G_\theta(\bar{\theta},x),\bar G_{\omega_1}(\bar{\theta},x),\dots,\bar G_{\omega_d}(\bar{\theta},x)),
\end{align*}
where $\bar F_\theta$, $\bar F_{\omega_j}$, $\bar G_\theta$ and $\bar G_{\omega_j}$ for all $j=1,\dots,d$ are given by
\begin{align*}
	\bar F_\theta(\bar{\theta},x)&\coloneq\eta|\bar{\theta}|^{2r}\theta, \quad\ \ \bar G_\theta(\bar{\theta},x)\coloneq 1-\frac{1}{1-\bar{q}}\II_{\{\sum_{i=1}^d g_i(\omega)x_i> \theta\}},\\
	\bar F_{\omega_j}(\bar{\theta},x)&\coloneq \eta |\bar{\theta}|^{2r}\omega_j,\quad \bar G_{\omega_j}(\bar{\theta},x)\coloneq\frac{1}{1-\bar{q}}\bar{g}_{\omega_j}(\omega,x)\II_{\{\sum_{i=1}^d g_i(\omega)x_i> \theta\}}.
\end{align*}
Here, for all $j=1,\dots,d$, $\bar{g}_{\omega_j}(\omega,x)\coloneq\sum_{i=1}^d\frac{\partial g_i(\omega)}{\partial \omega_j}x_i$ with  $\frac{\partial g_j(\omega)}{\partial \omega_j}=e^{\omega_j}(\sum_{k\neq j}e^{\omega_k})/(\sum_{k=1}^d e^{\omega_k})^2$ and $\frac{\partial g_i(\omega)}{\partial \omega_j}=-e^{\omega_i}e^{\omega_j}/(\sum_{k=1}^d e^{\omega_k})^2$ for $i\neq j$.

Next, we specify the stochastic gradient $H^\#(\theta,X_{n+1})$ for the $n$-th tSGHMC update:
\[H^\#(\theta,X_{n+1})\coloneq
\begin{cases}
	H(\theta,X_{n+1}), & \text{for the univariate CVaR computation},\\
	\bar H(\theta,X_{n+1}), & \text{for the risk allocation problem}.
\end{cases}\]
Here, $X_{n+1}\coloneq x_{I_{n+1}}$, where the index $I_{n+1}$ is drawn independently and uniformly from $\{1,\dots,2167\}$ for each $n\in\mathbb{N}_0$. In the univariate case, $\{x_i\}_{i=1}^{2167}$ are the observed values of the ``Total'' aggregate loss. In the risk allocation case, $\{x_i\}_{i=1}^{2167}$ are three-dimensional vectors, each containing the building, contents, and profits components.

\paragraph{\textbf{Numerical Experiments}}

We set $\bar q=0.95$ and $\eta=10^{-3}$, then conduct the experiments on the Danish reinsurance claim dataset using tSGHMC to solve the aforementioned optimization problems. The reference solution $\theta^*$ is obtained numerically using Adam \cite{kingma2014adam}.

\begin{figure}
	\centering
	\includegraphics[width=1.0\linewidth]{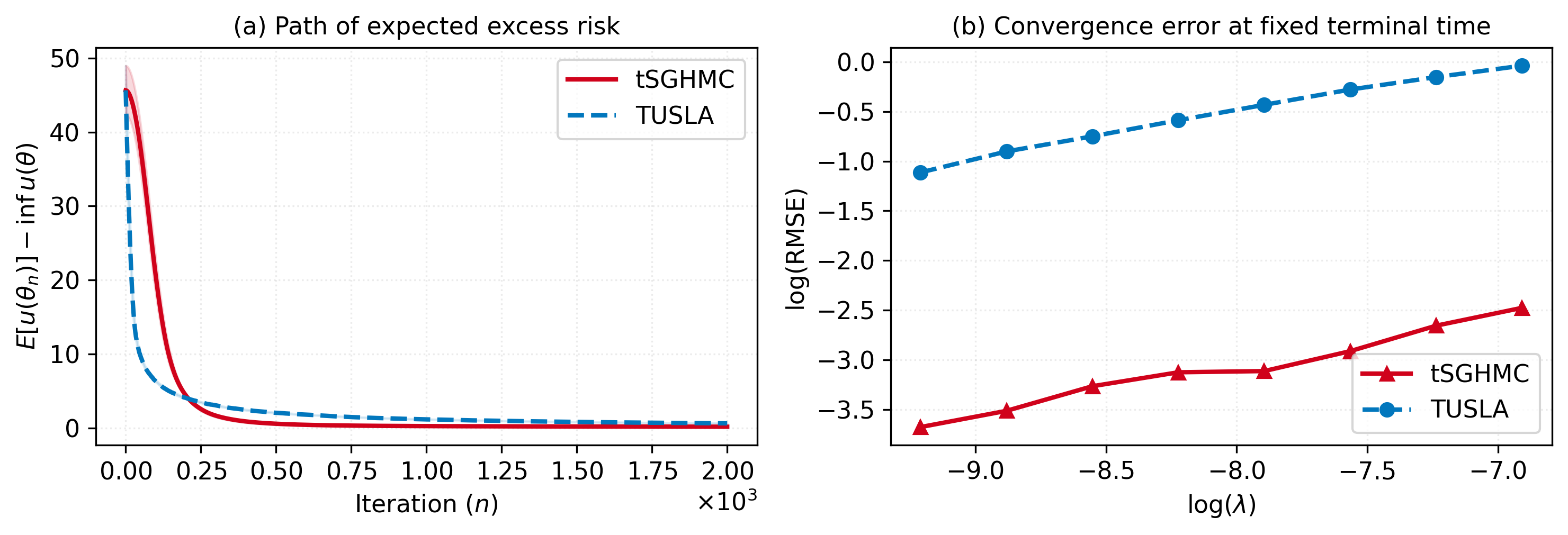}
	\caption{Comparison of tSGHMC and TUSLA on the univariate CVaR computation problem \eqref{eq.VaR-OPT}-\eqref{eq.VaR-loss}. (a) Path of expected excess risk versus the number of iterations $n$. (b) Log-log plot of the RMSE at $T=10$ across various step sizes $\lambda$.}
	\label{fig.example3var1}
\end{figure}

\begin{figure}
	\centering
	\includegraphics[width=1.0\linewidth]{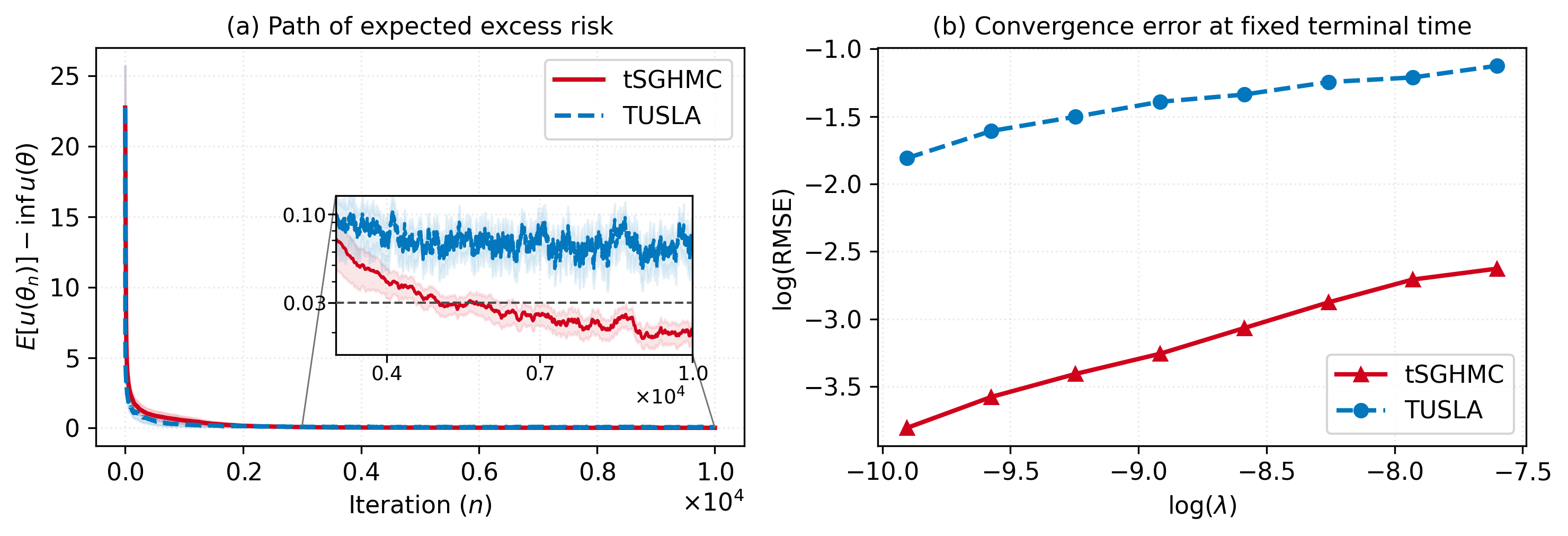}
	\caption{Numerical comparison of tSGHMC and TUSLA on the risk allocation problem \eqref{eq.Risk-OPT}-\eqref{eq.Risk-loss}. (a) Expected excess risk as a function of the iteration number $n$. (b) Log-log plot of the RMSE at $T=10$ across various step sizes $\lambda$.}
	\label{fig.example3var3}
\end{figure}

\begin{enumerate}
	\item For the univariate CVaR computation, we set $\beta=10^{10}$, $r=0.5$, $\gamma=1$, and $m=10^{-3}$. We then run 100 independent trials, initializing each with $\theta_0\sim\mathcal{N}(0,I_d)$ and $\nu_0=0$.\\
	Figure~\ref{fig.example3var1}(a) shows the expected excess risk trajectory over
	2,000 iterations with $\lambda=5\times 10^{-3}$, where the shaded region
	represents the sample mean $\pm 1.96$ standard errors at each iteration. The
	trajectory generated by tSGHMC decreases toward zero, demonstrating its effectiveness for
	the optimization problem. Compared with TUSLA, tSGHMC attains a smaller expected
	excess risk, indicating higher accuracy. In addition,
	Figure~\ref{fig.example3var1}(b) presents the log-log plot of the RMSE \eqref{eq.RMSE-self} against different step sizes $\lambda\in\{10^{-4},1.39\times 10^{-4},1.93\times 10^{-4},2.68\times 10^{-4},3.73\times 10^{-4}, 5.18\times 10^{-4}, 7.20\times 10^{-4},10^{-3}\}$, using tSGHMC and TUSLA with $R=100$, $T=10$, and $\lambda_{\text{ref}}=10^{-5}$. The numerical results show that tSGHMC yields a smaller RMSE than TUSLA.
	\item For the risk allocation problem, we set the input dimension $d=3$. We run tSGHMC 100 times with $\beta=10^{10}$, $r=1$, $\gamma=3$, and $m=10^{-3}$. For each independent run, the initial conditions are $\theta_0\sim\mathcal{N}(0,I_{d+1})$ and $\nu_0=0$.\\
	Figure~\ref{fig.example3var3}(a) reports the expected excess risk over $10^4$ iterations for tSGHMC with $\lambda=0.05$, with shaded bands representing the sample mean $\pm 1.96$ standard errors at each iteration. The risk converges to zero, showing that tSGHMC
	can be used to solve the risk allocation problem. In comparison with TUSLA, tSGHMC exhibits higher accuracy: it reaches the precision level of 0.03, whereas TUSLA does not. Moreover, Figure~\ref{fig.example3var3}(b) shows the RMSE in \eqref{eq.RMSE-self} on a log-log scale for step sizes $\lambda\in\{5\times 10^{-5},6.95\times 10^{-5},9.65\times 10^{-5},1.34\times 10^{-4},1.86\times 10^{-4},2.59\times 10^{-4},3.60\times 10^{-4},5\times 10^{-4}\}$, with $R=100$, $T=10$, and $\lambda_{\text{ref}}=10^{-5}$. Across all step sizes, tSGHMC consistently achieves smaller RMSE values than TUSLA.
\end{enumerate}

\subsubsection{Nonlinear Regression}

In this section, we consider the nonlinear regression problem \cite{liang2024non,lim2024non} defined as follows:
\begin{equation}\label{eq.Reg-OPT}
	\text{minimize}\quad \mathbb{R}^d \ni \theta \mapsto 
	u(\theta) \coloneq 
	\e[\ell(Y,\mathfrak{N}(\theta,Z))]
	+\frac{\eta}{2(r+1)}|\theta|^{2(r+1)},
\end{equation}
where $\ell\colon \mathbb{R}^{m_2}\times\mathbb{R}^{m_2}\to\mathbb{R}$ denotes the loss function with $m_2\in\mathbb{N}$, $\theta\in\mathbb{R}^d$ is the parameter of the problem, $Z$ is the $\mathbb{R}^{m_1}$-valued input random variable with $m_1\in\mathbb{N}$, $Y$ is the $\mathbb{R}^{m_2}$-valued target random variable, and $\mathfrak{N}\colon \mathbb{R}^d\times\mathbb{R}^{m_1}\to\mathbb{R}^{m_2}$ denotes the neural network model. In this experiment, we use the squared loss $\ell(u,v)=|u-v|^2$ for $u,v\in\mathbb{R}^{m_2}$ and consider a one-hidden-layer feedforward neural network $\mathfrak{N}\colon \mathbb{R}^d\times\mathbb{R}^{m_1}\to\mathbb{R}^{m_2}$ defined by
\begin{equation}
	\mathfrak{N}(\theta,z)\coloneq W_2\sigma(W_1 z+b_1)+b_2,
\end{equation}
where $\theta\coloneq([W_1],[W_2],b_1,b_2)\in\mathbb{R}^d$, with $d=d_1(m_1+m_2+1)+m_2$, $W_1\in\mathbb{R}^{d_1\times m_1}$, $b_1\in\mathbb{R}^{d_1}$, $W_2\in\mathbb{R}^{m_2\times d_1}$, $b_2\in\mathbb{R}^{m_2}$, $\sigma$ denotes the ReLU activation function, defined componentwise by $\sigma(u)\coloneq\max\{0,u\}$.

\begin{figure}
	\centering
	\includegraphics[width=1.0\linewidth]{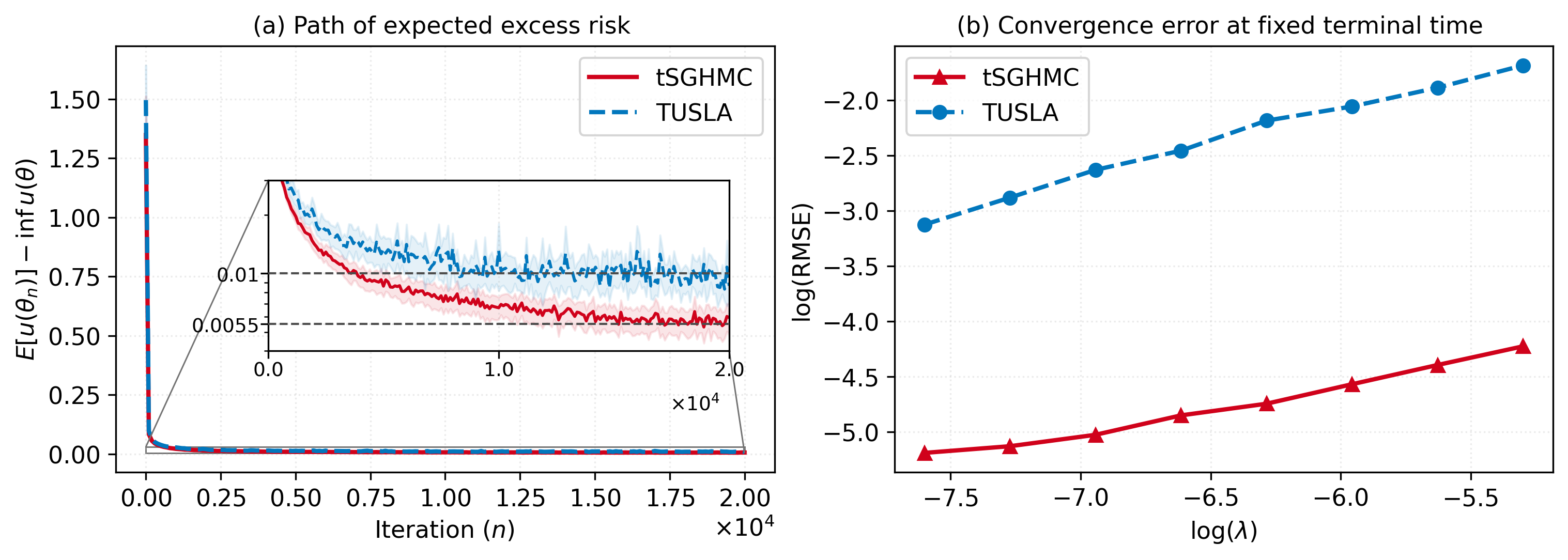}
	\caption{Convergence comparison between tSGHMC and TUSLA for the nonlinear regression application \eqref{eq.Reg-OPT}. (a) Path of expected excess risk versus the number of iterations $n$. (b) Log-log plot of the RMSE at $T=10$ across different step sizes $\lambda$.}
	\label{fig.example4nn}
\end{figure}

We test the performance of tSGHMC on the Concrete Compressive Strength dataset \cite{yeh1998modeling} from the UCI Machine Learning Repository,\footnote{Further details on the concrete compressive strength dataset are available at \url{https://archive.ics.uci.edu/dataset/165/concrete+compressive+strength}; additional information about the UCI Machine Learning Repository can be found at \url{https://archive.ics.uci.edu}.} which contains 1,030 samples. Each sample has 8 input features and 1 target variable, so that $m_1=8$ and $m_2=1$. Both the input variables and the target variable are standardized before training. We set $d_1=20$ (which yields $d=201$), $\eta=10^{-3}$, $\beta=10^{10}$, $r=0.5$, $\gamma=0.5$ and $m=10^{-10}$. We use mini-batch stochastic gradient with a mini-batch size of 256. We run tSGHMC 30 times, for each independent run, initialize $\theta_0=([W_1^0],[W_2^0],b_1^0,b_2^0)$ using Xavier initialization \cite{glorot2010understanding}, i.e., $[W_1^0]_{i,j}\sim\mathcal{N}(0,2/(m_1+d_1))$, $[W_2^0]_{i,j}\sim\mathcal{N}(0,2/(d_1+m_2))$, with $b_1^0=0$, $b_2^0=0$, and $\nu_0=0$.

For the first experiment, we compute the expected excess risk over $2\times 10^4$ iterations with $\lambda=0.5$. The reference solution $\theta^*$ is computed numerically using the Adam algorithm \cite{kingma2014adam}.  As shown in Figure~\ref{fig.example4nn}(a), the tSGHMC trajectory converges to zero, with the shaded area denoting the sample mean $\pm 1.96$ standard errors, validating the effectiveness of the algorithm. Compared with TUSLA, tSGHMC converges faster and attains a smaller excess risk. Specifically, it falls below the precision level of 0.01 using fewer iterations than TUSLA and reaches a stricter threshold of $5.5\times 10^{-3}$ which cannot be achieved by TUSLA.

Finally, we compare the RMSE defined in \eqref{eq.RMSE-self} for both algorithms. We set $R = 30$, $T = 10$, $\lambda_{\text{ref}} = 10^{-4}$, and $\lambda\in\{5\times 10^{-4}, 6.95\times 10^{-4}, 9.65\times 10^{-4}, 1.34\times 10^{-3}, 1.86\times 10^{-3}, 2.59\times 10^{-3}, 3.60\times 10^{-3}, 5\times 10^{-3}\}$. As illustrated in Figure~\ref{fig.example4nn}(b), tSGHMC consistently yields smaller RMSE values than TUSLA across the entire range of step sizes.

\section{Proof of the Main Results}
\label{section.Main Results}

In this section, we provide proofs for the main results. We begin with a proof overview in Section~\ref{subsection:Proof Overview}. Then, in Section~\ref{subsection.MY}, we introduce the Moreau-Yosida regularization as an appropriate approximation of $u$ possessing certain regularity properties and review several of its properties. Next, in Section~\ref{subsection.auxiliary}, we define several auxiliary processes that are essential to the convergence analysis. We derive preliminary estimates in Section~\ref{subsection.estimates}, including moment bounds for both the algorithm and the associated auxiliary processes, which provide the basis for the convergence analysis. The section concludes in Section~\ref{Proof of the Main Theorems} with the proofs of the main theorems.

\subsection{Proof Overview}
\label{subsection:Proof Overview}

Our goal is to establish an error bound for $W_2(\mathcal{L}(\theta_n^\lambda),\pi_\beta)$. Standard strategy for the convergence analysis involves the following decomposition:
\[W_2(\mathcal{L}(\theta_n^\lambda),\pi_\beta)\leq W_2(\mathcal{L}(\theta_n^\lambda),\mathcal{L}(\widetilde{Z}_t^\lambda))+W_2(\mathcal{L}(\widetilde{Z}_t^\lambda),\pi_\beta),\]
where $(\widetilde{Z}_t^\lambda)_{t\geq 0}=(\widetilde{Z}_{\lambda t})_{t\geq 0}$ is the time-changed process of the underdamped Langevin SDE \eqref{eq.KLD}. The second term on the right-hand side of the above inequality relates to certain contraction property of the underdamped Langevin SDE. However, existing contraction results typically require the potential gradient $h$ to be globally Lipschitz, see, e.g., \cite{dalalyan2020sampling,eberle2019couplings}, which cannot be used in our setting where $h$ is polynomially Lipschitz continuous, see Remark~\ref{remark.local Lip. of h}. To address this issue, we adopt the approach developed in \cite{johnston2024kinetic} and introduce the Moreau-Yosida regularization $u_{MY,\epsilon}$ of $u$, which provides a smooth approximation of the original potential. Crucially, $u_{MY,\epsilon}$ possesses a globally Lipschitz gradient and preserves the strong convexity property of $u$. Moreover, the associated Gibbs measure $\pi_\beta^\epsilon$, defined in Definition~\ref{def.pi}, remains close to the original target measure $\pi_\beta$.

Within this framework, we establish the main result by utilizing the decomposition:
\begin{equation}\label{eq.main decomposition}
	W_2(\mathcal{L}(\theta_n^\lambda),\pi_\beta)
	\leq
	W_2(\mathcal{L}(\theta_n^\lambda),\pi_\beta^\epsilon)
	+
	W_2(\pi_\beta^\epsilon,\pi_\beta).
\end{equation}
The first term on the right-hand side of \eqref{eq.main decomposition} captures the convergence of tSGHMC towards $\pi_\beta^\epsilon$. Bounding this term requires introducing several auxiliary processes, which we detail in Section~\ref{subsection.auxiliary}. The second term on the right-hand side of \eqref{eq.main decomposition} quantifies the approximation error induced by the Moreau-Yosida regularization and is analyzed in Section~\ref{subsection.MY}.

\subsection{Introduction to Moreau-Yosida Regularization}
\label{subsection.MY}

In this section, we introduce the Moreau-Yosida regularization, see, e.g., \cite{beck2017first,durmus2018efficient,johnston2024kinetic,planiden2019proximal,rockafellar1998variational} as a key tool for approximating a strongly convex objective function $u$ whose gradient is not globally Lipschitz.

We first present several classical results for Moreau-Yosida regularization. Given a proper, closed, and convex function $g\colon \mathbb{R}^d\to(-\infty,\infty]$ and a smoothing parameter $\epsilon>0$, the $\epsilon$-Moreau-Yosida regularization $g^\epsilon\colon \mathbb{R}^d\to(-\infty,\infty]$ of $g$ is given by
\[g^\epsilon(\theta)\coloneq \inf_{\phi\in\mathbb{R}^d}\left\{g(\phi)+\frac{1}{2\epsilon}|\theta-\phi|^2\right\}.\]

We recall the following properties of $g^\epsilon$. For detailed derivations of results in Lemma \ref{lemma.MY}(1)-(6), one may refer to \cite[Example~1.46]{rockafellar1998variational}, \cite[Theorem~6.55 and 6.60]{beck2017first}, \cite[Equation (13) and Lemma~5.2]{johnston2024kinetic}, and \cite[Theorem~3.12]{planiden2019proximal}, respectively.

\begin{lemma}\label{lemma.MY}
	Let $g\colon \mathbb{R}^d\to(-\infty,\infty]$ be a proper, closed, and convex function. Then:
	\begin{enumerate}
		\item $\arg\min g^\epsilon=\arg\min g$ and $\min g^\epsilon=\min g$.
		\item $g^\epsilon\colon \mathbb{R}^d\to\mathbb{R}$ is real-valued and convex.
		\item $g^\epsilon\colon \mathbb{R}^d\to\mathbb{R}$ is continuously differentiable with an $\epsilon^{-1}$-Lipschitz gradient, i.e., for all $\theta,\theta'\in\mathbb{R}^d$,
		\[|\nabla g^\epsilon(\theta)-\nabla g^\epsilon(\theta')|\leq\frac{1}{\epsilon}|\theta-\theta'|.\]
		Furthermore, its gradient is given by
		\[\nabla g^\epsilon(\theta)=\frac{1}{\epsilon}\left(\theta-\operatorname{prox}_g^\epsilon(\theta)\right),\]
		where the proximal operator is defined as
		\[\operatorname{prox}_g^\epsilon(\theta)\coloneq\arg\min_{\phi\in\mathbb{R}^d}\left\{g(\phi)+\frac{1}{2\epsilon}|\theta-\phi|^2\right\}.\]
		\item If $g$ is continuously differentiable, we have, for all $\theta\in\mathbb{R}^d$, that
		\[\nabla g^\epsilon(\theta)=\nabla g(\operatorname{prox}_g^\epsilon(\theta)).\]
		\item If $g$ is twice continuously differentiable, we have that $g^\epsilon$ is also twice continuously differentiable.
		\item If $g$ is $m$-strongly convex and $\epsilon<m^{-1}$, we have that $g^\epsilon$ is an $(m/2)$-strongly convex function.
	\end{enumerate}
\end{lemma}

We now define the Moreau-Yosida regularization of $u$, denoted by $u_{MY,\epsilon}\colon \mathbb{R}^d\to[0,\infty]$, as
\[u_{MY,\epsilon}(\theta)\coloneq \inf_{\phi\in\mathbb{R}^d}\left\{u(\phi)+\frac{1}{2\epsilon}|\theta-\phi|^2\right\}.\]
Let $h_{MY,\epsilon}\coloneq\nabla u_{MY,\epsilon}$ denote its gradient. By Lemma~\ref{lemma.MY}, we can show that $u_{MY,\epsilon}$ is gradient Lipschitz, $(m/2)$-strongly convex, and satisfies a dissipativity condition.

\begin{remark}\label{remark.MY}
	Let $\epsilon>0$. By Lemma~\ref{lemma.MY} and Remark~\ref{remark.dissipativity of h}, $u_{MY,\epsilon}$ is $(m/2)$-strongly convex provided that $\epsilon<m^{-1}$. Moreover, $h_{MY,\epsilon}$ is $K$-Lipschitz continuous, where $K\equiv K(\epsilon)\leq \epsilon^{-1}$. Hence, we have, for all $\theta\in\mathbb{R}^d$, that
	\[\frac{m}{2}I_d\preceq \operatorname{Hess}(u_{MY,\epsilon})(\theta)\preceq KI_d.\]
	In addition, Lemma~\ref{lemma.MY} and \cite[Definition 1.22]{rockafellar1998variational} imply that for all $\theta\in\mathbb{R}^d$,
	\begin{equation}\label{eq.uMY le u}
		0\leq u(\theta^*)=u_{MY,\epsilon}(\theta^*)\leq u_{MY,\epsilon}(\theta)\leq u(\theta).
	\end{equation}
	Finally, combining the strong convexity of $u_{MY,\epsilon}$ with \eqref{eq.uMY le u} yields
	\[\la h_{MY,\epsilon}(\theta),\theta\ra\geq\frac{m}{4}|\theta|^2-u_{MY,\epsilon}(0)\geq\frac{m}{4}|\theta|^2-u(0).\]
\end{remark}

In the following result, we show that $h_{MY,\epsilon}$ can be upper bounded by $h$.

\begin{lemma}\label{lemma.hMY leq h}
	Let Assumptions~\ref{assumption1.initial condition}-\ref{assumption3.strong convexity} hold, and let $\epsilon>0$. Then, for all $\theta\in\mathbb{R}^d$,
	\[|h_{MY,\epsilon}(\theta)|\leq|h(\theta)|.\]
\end{lemma}

\begin{proof}
	Postponed to Appendix \hyperref[proof.hMY leq h]{C.1}.
\end{proof}

We provide below an error estimate between $h$ and $h_{MY,\epsilon}$.

\begin{lemma}\label{lemma.error of MY regularization}
	Let Assumptions~\ref{assumption1.initial condition}-\ref{assumption3.strong convexity} hold, and let $\epsilon>0$. Then, for all $\theta\in\mathbb{R}^d$,
	\[|h(\theta)-h_{MY,\epsilon}(\theta)|\leq 2^{2r-1}L_h^2(1+|\theta|+R_0)^{4r-1}\epsilon,\]
	where $L_h$ is defined in Remark~\ref{remark.local Lip. of h}, and $R_0$ is defined in Remark~\ref{remark.dissipativity of h}.
\end{lemma}

\begin{proof}
	Postponed to Appendix \hyperref[proof.error of MY regularization]{C.1}.
\end{proof}

Next, we introduce probability measures associated with $u_{MY,\epsilon}$, which are essential in establishing convergence results for tSGHMC \eqref{tSGHMC}-\eqref{eq.tamedSG}.

\begin{definition}\label{def.pi}
	Let $0<\epsilon<m^{-1}$. We define
	\[\pi_\beta^\epsilon(\theta)\coloneq \frac{\exp(-\beta u_{MY,\epsilon}(\theta))}{\int_{\mathbb{R}^d}\exp(-\beta u_{MY,\epsilon}(\theta))\,\dd\theta},\quad \Pi_\beta^\epsilon(\theta,\nu)\coloneq \frac{\exp(-\beta ( u_{MY,\epsilon}(\theta) + |\nu|^2/2 ))}{\int_{\mathbb{R}^d}\exp(-\beta ( u_{MY,\epsilon}(\theta) + |\nu|^2/2 ))\,\dd\theta\dd\nu}.\]
\end{definition}

Recalling the definitions of $\pi_\beta$ and $\Pi_\beta$ in \eqref{eq.pi_beta} and \eqref{eq.Pi_beta}, respectively, we establish a convergence result for $\Pi_\beta$ and $\Pi_\beta^\epsilon$ in Wasserstein-2 distance, which is established through $W_2(\pi_\beta,\pi_\beta^\epsilon)$. To make sense of this quantity, we first introduce moment estimates for $\pi_\beta$ and $\pi_\beta^\epsilon$.

\begin{lemma}\label{lemma.moment bounds of pi_beta}
	Let Assumptions~\ref{assumption1.initial condition}-\ref{assumption3.strong convexity} hold, let $q\in\mathbb{N}$, and let $Y$ and $\bar Y$ be random variables with $\mathcal{L}(Y)=\pi_\beta$ and $\mathcal{L}(\bar Y)=\pi_\beta^\epsilon$, respectively. Then, we have
	\begin{align*}
		\e[|Y|^{2q}]&\leq C_{\pi_\beta,2q}\coloneq\left(\frac{4u(0)}{m}+\frac{4(d+2(q-1))}{\beta m}\right)^q<\infty,\\
		\e[|\bar Y|^{2q}]&\leq \left(\frac{8u(0)}{m}+\frac{8(d+2(q-1))}{\beta m}\right)^q<\infty,
	\end{align*}
	where $C_{\pi_\beta,2q}=\mathcal{O}((d/\beta)^q)$.
\end{lemma}

\begin{proof}
	Postponed to Appendix \hyperref[proof.moment bounds of pi_beta]{C.1}.
\end{proof}

\begin{lemma}
	\label{lemma.contraction of pi_beta}
	Let Assumptions~\ref{assumption1.initial condition}-\ref{assumption3.strong convexity} hold, and let $\epsilon>0$. Then, there exists a constant $\bar{c}>0$ such that
	\[W_2(\Pi_\beta,\Pi_\beta^\epsilon)\leq W_2(\pi_\beta ,\pi_\beta^\epsilon)\leq \bar{c}\epsilon,\]
	where $\bar{c}=\mathcal{O}((d/\beta)^{2r-1/2})$ with its explicit expression given in Table~\ref{tab.constants}.
\end{lemma}

\begin{proof}
	Postponed to Appendix \hyperref[proof.contraction of pi_beta]{C.1}.
\end{proof}

\subsection{Auxiliary Processes}
\label{subsection.auxiliary}

To provide a convergence upper bound for $W_2(\mathcal{L}(\theta_n^\lambda),\pi_\beta)$, we introduce several auxiliary processes. We consider the underdamped Langevin SDE $(\widetilde{Z}_t,\widetilde{V}_t)_{t\geq 0}$ given by
\begin{align}\label{eq.KLD}
	\begin{split}
		\dd \widetilde{V}_t&=-[\gamma \widetilde{V}_t+h(\widetilde{Z}_t)]\,\dd t+\sqrt{2\gamma\beta^{-1}}\,\dd W_t,\\
		\dd \widetilde{Z}_t&=\widetilde{V}_t\,\dd t,
	\end{split}
\end{align}
with $(\widetilde{Z}_0,\widetilde{V}_0)=(\theta_0,\nu_0)$ being $\mathbb{R}^{2d}$-valued random variable, where $(W_t)_{t\geq 0}$ is the standard $d$-dimensional Brownian motion on $(\Omega,\mathcal{F},\p)$. Denote by $(\mathcal{F}_t)_{t\geq 0}$ the completed natural filtration of $(W_t)_{t\geq 0}$, which is assumed to be independent of $\mathcal{G}_\infty\vee\sigma(\theta_0,\nu_0)$. 

We next consider the underdamped Langevin SDE $(r_t,R_t)_{t\geq 0}$ related to the Moreau-Yosida regularization given by
\begin{align}\label{eq.MY kinetic Langevin}
	\begin{split}
		\dd R_t&=-[\gamma R_t+h_{MY,\epsilon}(r_t)]\,\dd t+\sqrt{2\gamma\beta^{-1}}\,\dd W_t,\\
		\dd r_t&=R_t\,\dd t,
	\end{split}
\end{align}
with $(r_0,R_0)=(\theta_0,\nu_0)$. 

For each $\lambda>0$, we consider the scaled process $(r_t^\lambda,R_t^\lambda)\coloneq (r_{\lambda t},R_{\lambda t})$ with $(r_t,R_t)_{t\geq 0}$ defined in \eqref{eq.MY kinetic Langevin}, which is given by
\begin{align}\label{scaled process}
	\begin{split}
		\dd R_t^\lambda&=-\lambda[\gamma R_t^\lambda+h_{MY,\epsilon}(r_t^\lambda)]\,\dd t+\sqrt{2\lambda\gamma\beta^{-1}}\,\dd W_t^\lambda,\\
		\dd r_t^\lambda&=\lambda R_t^\lambda\,\dd t,
	\end{split}
\end{align}
with $(r_t^\lambda,R_t^\lambda)=(\theta_0,\nu_0)$, where $W_t^\lambda\coloneq \lambda^{-1/2}W_{\lambda t}$. We denote the filtration of $(W_t^\lambda)_{t\geq 0}$ by $(\mathcal{F}_t^\lambda)_{t\geq 0}\coloneq (\mathcal{F}_{\lambda t})_{t\geq 0}$, and we note that $(\mathcal{F}_t^\lambda)_{t\geq 0}$ is independent of $\mathcal{G}_\infty\vee\sigma(\theta_0,\nu_0)$. 

Next, we define the continuous-time interpolation of tSGHMC \eqref{tSGHMC}-\eqref{eq.tamedSG}, denoted by $(Z_t^\lambda,V_t^\lambda)_{t\geq 0}$, as
\begin{align}\label{Continuous-time interpolation}
	\begin{split}
		\dd V_t^\lambda&=-\lambda[\gamma V_{\lfloor t\rfloor}^\lambda+H_\gamma(Z_{\lfloor t\rfloor}^\lambda,X_{\lceil t\rceil})]\,\dd t+\sqrt{2\lambda\gamma\beta^{-1}}\,\dd W_t^\lambda,\\
		\dd Z_t^\lambda&=\lambda V_{\lfloor t\rfloor}^\lambda\,\dd t,
	\end{split}
\end{align}
with $(Z_t^\lambda,V_t^\lambda)=(\theta_0,\nu_0)$. The process \eqref{Continuous-time interpolation} mimics the recursion \eqref{tSGHMC}-\eqref{eq.tamedSG} at grid points in the sense that $\mathcal{L}(\theta_n^\lambda,\nu_n^\lambda)=\mathcal{L}(Z_n^\lambda,V_n^\lambda)$ for all $n\in\mathbb{N}_0$. 

Furthermore, we define the auxiliary process $(\widehat{\zeta}_t^{s,u,v,\lambda},\widehat{V}_t^{s,u,v,\lambda})_{t\geq s}$ as
\begin{align*}
	\begin{split}
		\dd \widehat{V}_t^{s,u,v,\lambda}&=-\lambda[\gamma \widehat{V}_t^{s,u,v,\lambda}+h_{MY,\epsilon}(\widehat{\zeta}_t^{s,u,v,\lambda})]\,\dd t+\sqrt{2\lambda\gamma\beta^{-1}}\,\dd W_t^\lambda,\\
		\dd \widehat{\zeta}_t^{s,u,v,\lambda}&=\lambda \widehat{V}_t^{s,u,v,\lambda}\,\dd t,
	\end{split}
\end{align*}
with $(\widehat{\zeta}_s^{s,u,v,\lambda},\widehat{V}_s^{s,u,v,\lambda})=(u,v)$.

\begin{definition}\label{Changed Initial Processes}
	For any $0<\lambda\leq\lambda_{\max,\gamma}$ and $n\in\mathbb{N}_0$, set $T\coloneq \lfloor 1/\lambda\rfloor$ and define
	\begin{equation}\label{Auxiliary process with certain initial}
		\zeta_t^{\lambda,n}=\widehat{\zeta}_t^{nT,Z_{nT}^\lambda,V_{nT}^\lambda,\lambda},\quad V_t^{\lambda,n}=\widehat{V}_t^{nT,Z_{nT}^\lambda,V_{nT}^\lambda,\lambda}.
	\end{equation}
\end{definition}

We note that, by Definition~\ref{Changed Initial Processes}, the process $(\zeta_t^{\lambda,n},V_t^{\lambda,n})_{t\geq nT}$ is the underdamped Langevin process \eqref{scaled process} starting at time $nT$ with $(Z_{nT}^\lambda,V_{nT}^\lambda)$.

\subsection{Preliminary Estimates}\label{subsection.estimates}

In this section, we provide moment estimates for tSGHMC \eqref{tSGHMC}-\eqref{eq.tamedSG}, as well as the auxiliary process introduced in Section~\ref{subsection.auxiliary}.

Recall $\gamma_{\min}$ and $\lambda_{\max,\gamma}$ defined in \eqref{eq.friction restriction} and \eqref{eq.step-size restriction}, respectively. We provide below moment estimates for tSGHMC.

\begin{lemma}\label{lemma.2 moment of algorithm}
	Let Assumptions~\ref{assumption1.initial condition}-\ref{assumption3.strong convexity} hold. For $\gamma\geq\gamma_{\min}$ and $\lambda\leq\lambda_{\max,\gamma}$, we obtain that
 	\[\sup_{n\in\mathbb{N}_0}\e[|\theta_n^\lambda|^2]\leq \bar{C}_2,\quad\sup_{n\in\mathbb{N}_0}\e[|\nu_n^\lambda|^2]\leq \bar{B}_2,\]
	where $\bar{C}_2=\mathcal{O}(d/\beta)$ and $\bar{B}_2=\mathcal{O}(d/\beta)$ with their explicit expressions given in Table~\ref{tab.constants}.
\end{lemma}

\begin{proof}
	Postponed to Appendix \hyperref[Proof.second moment bounds]{C.2}.
\end{proof}

\begin{lemma}\label{lemma.2q moment of algorithm}
	Let Assumptions~\ref{assumption1.initial condition}-\ref{assumption3.strong convexity} hold. For all $q\in[2,\infty)\cap\mathbb{N}$, $\gamma\geq\gamma_{\min}$ and $\lambda\leq\lambda_{\max,\gamma}$, we have that
	\[\sup_{n\in\mathbb{N}_0}\e[|\theta_n^\lambda|^{2q}]\leq\bar{C}_{2q},\]
	where $\bar{C}_{2q}=\mathcal{O}((d/\beta)^q)$ with the explicit expression given in Table~\ref{tab.constants}.
\end{lemma}

\begin{proof}
	Postponed to Appendix \hyperref[Proof.Even Order Moment bounds]{C.2}.
\end{proof}

Next we present estimates for the one-step errors of $(Z_t^\lambda,V_t^\lambda)_{t\geq 0}$ defined in \eqref{Continuous-time interpolation}.

\begin{lemma}\label{lemma.mse of interpolation}
	Let Assumptions~\ref{assumption1.initial condition}-\ref{assumption3.strong convexity} hold. For all $t\geq 0$, $\gamma\geq\gamma_{\min}$ and $\lambda\leq\lambda_{\max,\gamma}$, we obtain that
	\[\e[|Z_t^\lambda-Z_{\bt}^\lambda|^2]\leq \lambda \bar{B}_2,\quad\e[|V_t^\lambda-V_{\bt}^\lambda|^2]\leq\lambda\gamma C_{1,v},\]
	where $C_{1,v}=\mathcal{O}(d/\beta)$ with the explicit expression given in Table~\ref{tab.constants}.
\end{lemma}

\begin{proof}
	Postponed to Appendix \hyperref[Proof.MSE of interpolation]{C.2}.
\end{proof}

The following lemmas state that the underdamped Langevin SDEs \eqref{eq.MY kinetic Langevin} and \eqref{Auxiliary process with certain initial} have uniform in time second moments.

\begin{lemma}\label{lemma.2nd bound MY kinetic}
	Let Assumptions~\ref{assumption1.initial condition}-\ref{assumption3.strong convexity} hold, and let $\gamma\geq\gamma_{\min}$. Then, we obtain that
	\[\sup_{t\geq 0}\e[|r_t|^2]\leq C_r,\quad\sup_{t\geq 0}\e[|R_t|^2]\leq C_R,\]
	where $C_r=\mathcal{O}(d/\beta)$ and $C_R=\mathcal{O}(d/\beta)$ with their explicit expressions given in Table~\ref{tab.constants}.
\end{lemma}

\begin{proof}
	Postponed to Appendix \hyperref[proof.2nd bound MY kinetic]{C.2}.
\end{proof}

\begin{lemma}\label{lemma.2nd bound of auxiliary process}
	Let Assumptions~\ref{assumption1.initial condition}-\ref{assumption3.strong convexity} hold, and let $\gamma\geq\gamma_{\min}$ and $\lambda\leq\lambda_{\max,\gamma}$. Then, we obtain that
	\[\sup_{n\in\mathbb{N}_0}\sup_{t\geq nT}\e[|\zeta_t^{\lambda,n}|^2]\leq C_\zeta^\#,\quad \sup_{n\in\mathbb{N}_0}\sup_{t\geq nT}\e[|V_t^{\lambda,n}|^2]\leq C_V^\#,\]
	where $C_\zeta^\#=\mathcal{O}(d/\beta)$ and $C_V^\#=\mathcal{O}(d/\beta)$ with their explicit expressions given in Table~\ref{tab.constants}.
\end{lemma}

\begin{proof}
	Postponed to Appendix \hyperref[proof.2nd bound of auxiliary process]{C.2}.
\end{proof}

\subsection{Proof of the Main Theorems}
\label{Proof of the Main Theorems}

We now outline the main steps and intermediate results used to establish Theorem~\ref{Main Theorem}. To derive a non-asymptotic bound in the Wasserstein-2 distance between $\mathcal{L}(\theta_n^\lambda)$ and the target distribution $\pi_\beta$, we begin by considering the following splitting using $(Z_t^\lambda,V_t^\lambda)_{t\geq 0}$ defined in \eqref{Continuous-time interpolation}: For any $n\in\mathbb{N}_0$ and $t\in(nT,(n+1)T]$, we have
\begin{equation}\label{Interpolation splitting}
	W_2(\mathcal{L}(Z_t^\lambda),\pi_\beta)\leq W_2(\mathcal{L}(Z_t^\lambda),\pi_\beta^\epsilon)+W_2(\pi_\beta^\epsilon,\pi_\beta).
\end{equation}
A convergence upper bound for of the second term on the right-hand side of \eqref{Interpolation splitting} has been established in Lemma~\ref{lemma.contraction of pi_beta} using Moreau-Yosida regularization described in Section~\ref{subsection.MY}. To upper bound the first term on the right-hand side of \eqref{Interpolation splitting}, we further split it as follows: For all $n\in\mathbb{N}_0$ and $t\in(nT,(n+1)T]$, we have
\begin{equation}\label{eq.auxiliary splitting}
	W_2(\mathcal{L}(Z_t^\lambda),\pi_\beta^\epsilon)\leq W_2(\mathcal{L}(Z_t^\lambda),\mathcal{L}(\zeta_t^{\lambda,n}))+W_2(\mathcal{L}(\zeta_t^{\lambda,n}),\mathcal{L}(r_t^\lambda))+W_2(\mathcal{L}(r_t^\lambda),\pi_\beta^\epsilon),
\end{equation}
where $(\zeta_t^{\lambda,n})_{t\geq 0}$ and $(r_t^\lambda)_{t\geq 0}$ are processes introduced in \eqref{Auxiliary process with certain initial} and \eqref{scaled process}, respectively.

At this stage, our task reduces to obtain an upper bound for each of the terms on the right-hand side of \eqref{eq.auxiliary splitting}. To this end, we provide below an upper bound for $W_2(\mathcal{L}(Z_t^\lambda),\mathcal{L}(\zeta_t^{\lambda,n}))$, which can be viewed as a convergence result for tSGHMC to the associated underdamped Langevin SDE.

\begin{proposition}\label{proposition: W_2 - interpolation and auxiliary}
	Let Assumptions~\ref{assumption1.initial condition}-\ref{assumption3.strong convexity} hold. Then, we have, for all $n\in\mathbb{N}_0$, $t\in(nT,(n+1)T]$, $\gamma\geq\gamma_{\min}$ and $\lambda\leq\lambda_{\max,\gamma}$, that
	\begin{equation}
		W_2(\mathcal{L}(Z_t^\lambda,V_t^\lambda),\mathcal{L}(\zeta_t^{\lambda,n},V_t^{\lambda,n}))\leq C_0\left(\lambda\gamma+\sqrt{\frac{S_{\lambda,\gamma}}{\gamma}}\right),
	\end{equation}
	where $C_0=\mathcal{O}((d/\beta)^{1/2})$, and $S_{\lambda,\gamma}=C_1\lambda^2\gamma^4+C_2\lambda\gamma+C_3\gamma^{-3}+C_4\gamma^{-1}\epsilon^2$ with $C_1=\mathcal{O}((d/\beta)^{2r})$, $C_2=\mathcal{O}((d/\beta)^{2r})$, $C_3=\mathcal{O}((d/\beta)^{6r})$, and $C_4=\mathcal{O}((d/\beta)^{4r-1})$. The explicit expressions are given in Table~\ref{tab.constants}.
\end{proposition}

\begin{proof}
	Postponed to Appendix \hyperref[Proof.W_2 - interpolation and auxiliary]{C.3}.
\end{proof}

To obtain upper estimates for the remaining two terms on the right-hand side of \eqref{Interpolation splitting}, we note that they both relate to the contraction property of the underdamped Langevin SDE \eqref{eq.MY kinetic Langevin}. Hence, we provide below a contraction result of SDE \eqref{eq.MY kinetic Langevin} adapted from \cite[Theorem 1]{dalalyan2020sampling}.

\begin{theorem}\label{theorem:contraction result of dalalyan}
	Let Assumptions~\ref{assumption1.initial condition}-\ref{assumption3.strong convexity} hold. Let $(r_t,R_t)$ and $(r_t',R_t')$ be two solutions of the underdamped Langevin SDE \eqref{eq.MY kinetic Langevin} with initial conditions $(r_0,R_0)$ and $(r_0',R_0')$, respectively. Then, we obtain, for all $t\geq 0$ and $\gamma\geq\gamma_{\min}$, that 
	\[W_2(\mathcal{L}(r_t),\mathcal{L}(r_t'))\leq 4\exp\left(-\frac{m}{2\gamma}t\right)W_2(\mathcal{L}(r_0,R_0),\mathcal{L}(r_0',R_0')).\]
\end{theorem}

\begin{proof}
	Postponed to Appendix \hyperref[proof.contraction result of dalalyan]{C.3}.
\end{proof}

By using Theorem \ref{theorem:contraction result of dalalyan}, we are able to derive upper bounds for the last two terms on the right-hand side of \eqref{Interpolation splitting}. These results are provided below.

\begin{proposition}\label{theorem:W_2 - scaled and auxiliary}
	Let Assumptions~\ref{assumption1.initial condition}-\ref{assumption3.strong convexity} hold. Then, we have, for all $n\in\mathbb{N}_0$, $t\in(nT,(n+1)T]$, $\gamma\geq\gamma_{\min}$ and $\lambda\leq\lambda_{\max,\gamma}$, that
	\[W_2(\mathcal{L}(\zeta_t^{\lambda,n}),\mathcal{L}(r_t^\lambda))\leq\frac{16C_0e}{m}\left(\lambda\gamma^2+\sqrt{\gamma S_{\lambda,\gamma}}\right),\]
	where $C_0$ is given in Proposition~\ref{proposition: W_2 - interpolation and auxiliary}. 
\end{proposition}

\begin{proof}
	Postponed to Appendix \hyperref[Proof.W_2 - scaled and auxiliary]{C.3}.
\end{proof}

\begin{corollary}\label{coro.r_t pi}
	Let Assumptions~\ref{assumption1.initial condition}-\ref{assumption3.strong convexity} hold. Let $(r_t^\lambda,R_t^\lambda)$ be a solution of the underdamped Langevin SDE \eqref{scaled process} with initial condition $(\theta_0,\nu_0)$. Then, we obtain, for all $t\geq 0$ and $\gamma\geq\gamma_{\min}$, that 
	\[W_2(\mathcal{L}(r_t^\lambda),\pi_\beta^\epsilon)\leq 4\exp\left(-\frac{\lambda m}{2\gamma}t\right)W_2(\mathcal{L}(\theta_0,\nu_0),\Pi_\beta^\epsilon).\]
\end{corollary}

\begin{proof}
	Postponed to Appendix \hyperref[proof.r_t pi]{C.3}.
\end{proof}

We are now ready to present the proof of Theorem~\ref{Main Theorem}.

\phantomsection
\begin{proof}[\textbf{Proof of Theorem~\ref{Main Theorem}}]\label{Proof.Main theorem}
	By the triangle inequality for the Wasserstein-2 distance with Propositions~\ref{proposition: W_2 - interpolation and auxiliary} and \ref{theorem:W_2 - scaled and auxiliary}, Corollary~\ref{coro.r_t pi}, and Lemma~\ref{lemma.contraction of pi_beta}, we have, for all $n\in\mathbb{N}_0$, $t\in(nT,(n+1)T]$, $\gamma\geq\gamma_{\min}$ and $\lambda\leq\lambda_{\max,\gamma}$, that
	\begin{align*}
		W_2(\mathcal{L}(Z_t^\lambda),\pi_\beta)&\leq W_2(\mathcal{L}(Z_t^\lambda),\mathcal{L}(\zeta_t^{\lambda,n}))+W_2(\mathcal{L}(\zeta_t^{\lambda,n}),\mathcal{L}(r_t^\lambda))+ W_2(\mathcal{L}(r_t^\lambda),\pi_\beta^\epsilon)+W_2(\pi_\beta^\epsilon,\pi_\beta)\\
		&\leq C_0\left(\lambda\gamma+\sqrt{\frac{S_{\lambda,\gamma}}{\gamma}}\right)+\frac{16C_0e}{m}\left(\lambda\gamma^2+\sqrt{\gamma S_{\lambda,\gamma}}\right)\\
		&\quad+4\exp{\left(-\frac{\lambda m}{2\gamma} t\right)}W_2(\mathcal{L}(\theta_0,\nu_0),\Pi_\beta^\epsilon)+W_2(\pi_\beta^\epsilon,\pi_\beta)\\
		&\leq 2\max\left\{C_0,\frac{16C_0e}{m}\right\}\left(\lambda\gamma^2+\sqrt{\gamma S_{\lambda\gamma}}\right)\\
		&\quad+4\exp{\left(-\frac{\lambda m}{2\gamma} t\right)}W_2(\mathcal{L}(\theta_0,\nu_0),\Pi_\beta)+5W_2(\pi_\beta^\epsilon,\pi_\beta)\\
		&\leq 2\max\left\{C_0,\frac{16C_0e}{m}\right\}\left(\lambda\gamma^2+\sqrt{\gamma S_{\lambda\gamma}}\right)+4\exp{\left(-\frac{\lambda m}{2\gamma} t\right)}W_2(\mathcal{L}(\theta_0,\nu_0),\Pi_\beta)+5\bar{c}\epsilon\\
		&\leq \dot{C}\left(\sqrt{\lambda}\gamma+\lambda\gamma^{5/2}+\gamma^{-1}+\exp\left(-\frac{\lambda m}{2\gamma} t\right)W_2(\mathcal{L}(\theta_0,\nu_0),\Pi_\beta)+\epsilon\right),
	\end{align*}
	where $\dot{C}=\mathcal{O}((d/\beta)^{3r+1/2})$, with its explicit expression given in Table~\ref{tab.constants}. Since $\lambda T\geq 1/2$, it follows, for any $n\in\mathbb{N}_0$, that
	\begin{align*}
		W_2(\mathcal{L}(Z_{(n+1)T}^\lambda),\pi_\beta)&\leq\dot{C}\left(\sqrt{\lambda}\gamma+\lambda\gamma^{5/2}+\gamma^{-1}+\exp\left(-\frac{\lambda m}{2\gamma} (n+1)T\right)W_2(\mathcal{L}(\theta_0,\nu_0),\Pi_\beta)+\epsilon\right).
	\end{align*}
	Then, we replace $(n+1)T$ on both sides of the above inequality with $n+1$ to obtain
	\[W_2(\mathcal{L}(Z_{n+1}^\lambda),\pi_\beta)\leq\dot{C}\left(\sqrt{\lambda}\gamma+\lambda\gamma^{5/2}+\gamma^{-1}+\exp\left(-\frac{\lambda m}{2\gamma} (n+1)\right)W_2(\mathcal{L}(\theta_0,\nu_0),\Pi_\beta)+\epsilon\right),\]
	which yields, by using $\mathcal{L}(V_n^\lambda,Z_n^\lambda)=\mathcal{L}(\nu_n^\lambda,\theta_n^\lambda)$, that
	\[W_2(\mathcal{L}(\theta_n^\lambda),\pi_\beta)\leq\dot{C}\left(\sqrt{\lambda}\gamma+\lambda\gamma^{5/2}+\gamma^{-1}+\exp\left(-\frac{\lambda m}{2\gamma}n\right)W_2(\mathcal{L}(\theta_0,\nu_0),\Pi_\beta)+\epsilon\right).\]
	This completes the proof.
\end{proof}

To derive a non-asymptotic error bound for the expected excess risk, we consider the following decomposition:
\begin{equation}
	\e[u(\theta_n^\lambda)]-u(\theta^*)=\left(\e[u(\theta_n^\lambda)]-\e[u(\theta_\infty)]\right)+\left(\e[u(\theta_\infty)]-u(\theta^*)\right),
\end{equation}
where $\theta_\infty$ is distributed according to $\pi_\beta$. The first term on the right-hand side characterizes the sampling error of tSGHMC \eqref{tSGHMC}-\eqref{eq.tamedSG}, and can be controlled through the corresponding Wasserstein-2 distance. The second term reflects the concentration properties of $\pi_\beta$, which becomes small when the inverse temperature parameter $\beta>0$ is sufficiently large.

\begin{lemma}\label{Lemma.sampling error}
	Let Assumptions~\ref{assumption1.initial condition}-\ref{assumption3.strong convexity} hold. Then, there exists $C'>0$ such that, for all $\gamma\geq\gamma_{\min}$, $\lambda\leq\lambda_{\max,\gamma}$, and $n\in\mathbb{N}_0$,
	\[\e[u(\theta_n^\lambda)]-\e[u(\theta_\infty)]\leq C'\left(\sqrt{\lambda}\gamma+\lambda\gamma^{5/2}+\gamma^{-1}+\exp\left(-\frac{\lambda m}{4\gamma}n\right)W_2(\mathcal{L}(\theta_0,\nu_0),\Pi_\beta)+\epsilon\right),\]
	where $C'=\mathcal{O}((d/\beta)^{4r+1/2})$ is given explicitly in Table~\ref{tab.constants}.
\end{lemma}

\begin{proof}
	Postponed to Appendix \hyperref[Proof.sampling error]{C.3}.
\end{proof}

\begin{lemma}\label{Lemma.concentration property}
	Let Assumptions~\ref{assumption1.initial condition}-\ref{assumption3.strong convexity} hold. Then, there exists $C''>0$ such that, for all $\gamma\geq\gamma_{\min}$ and $\lambda\leq\lambda_{\max,\gamma}$,
	\[\e[u(\theta_\infty)]-u(\theta^*)\leq C''\e[|\theta_\infty-\theta^*|^2],\]
	where $C''=\mathcal{O}((d/\beta)^{r+1/2})$ is given explicitly in Table~\ref{tab.constants}.
\end{lemma}

\begin{proof}
	Postponed to Appendix \hyperref[Proof.concentration property]{C.3}.
\end{proof}

\begin{lemma}[\cite{johnston2024kinetic}, Lemma 10.3]\label{lemma.second part of opt.}
	Let Assumptions~\ref{assumption1.initial condition}-\ref{assumption3.strong convexity} hold. Then
	\[\e[|\theta_\infty-\theta^*|^2]\leq\frac{2d}{m\beta}.\]
\end{lemma}

Finally, we present the proof of Theorem~\ref{Theorem.Optimization}.

\phantomsection
\begin{proof}[\textbf{Proof of Theorem~\ref{Theorem.Optimization}}]\label{Proof.Optimization Problem}
	Combining Lemmas~\ref{Lemma.sampling error}, \ref{Lemma.concentration property} and \ref{lemma.second part of opt.} yields the result.
\end{proof}

\newpage
\appendix

\section{Proofs of Results in Sections~\ref{section.Assumptions and Main Results}}
\label{Proofs of section.Applications}

\phantomsection
\begin{proof}[\textbf{Proof of Remark~\ref{remark.dissipativitiy of SG}}]
	\label{Proof.dissipative of SG}
	By Assumption~\ref{assumption3.strong convexity}, for all $\theta\in\mathbb{R}^d$ and $x\in\mathbb{R}^k$, it holds that
	\[\la \theta,H(\theta,x)-H(0,x)\ra\geq \la \theta,A(x)\theta\ra.\]
	Consequently, we obtain that
	\begin{align*}
		\la\theta,H(\theta,x)\ra&\geq \la\theta,A(x)\theta\ra+\la\theta,H(0,x)\ra\\
		&\geq \la\theta,A(x)\theta\ra-|\theta|\,|H(0,x)|\\
		&\geq \la\theta,A(x)\theta\ra-m|\theta|^2-K_H^2(1+|x|)^{2\rho}/(4m)\\
		&\geq \la\theta,\widetilde{A}(x)\theta\ra-\widetilde{b}(x)
	\end{align*}
	with $\widetilde{A}(x)=A(x)-m I_d$ and $\widetilde{b}(x)=K_H^2(1+|x|)^{2\rho}/(4m)$, where the third inequality follows from \eqref{eq.upper bound of |H|} together with the elementary inequality $ab<m a^2+b^2/(4m)$ for all $a,b>0$.
\end{proof}

\section{Proofs of Results in Section~\ref{section.Applications}}

\phantomsection
\begin{proof}[\textbf{Proof of Proposition~\ref{prop.Sampling}}]
	\label{proof.Sampling}
	By \eqref{eq.sampling-obj}, $u$ is nonnegative and twice continuously differentiable. Moreover, its gradient is given by
	\[h(\theta)=-\frac{1}{\mathsf{N}}\sum_{i=1}^\mathsf{N}\frac{y_ix_i}{1+\exp(y_i\la x_i,\theta\ra)}+\eta\theta.\]
	Since $I_l\sim \mathrm{Uniform}(\{1,\ldots,\mathsf{N}\})$ for $l=1,\ldots,\mathsf{K}$, we have $\p(I_l=i)=1/\mathsf{N}$ for all $i=1,\ldots,\mathsf{N}$. Consequently, for all $\theta\in\mathbb{R}^d$,
	\begin{align*}
		\e[H(\theta,U_{\mathbf{z}})]&=\eta\theta-\frac{1}{\mathsf{K}}\sum_{l=1}^\mathsf{K}\e\left[\frac{y_{I_l}x_{I_l}}{1+\exp(y_{I_l}\la x_{I_l},\theta\ra)}\right]\\
		&=\eta\theta-\frac{1}{\mathsf{K}}\sum_{l=1}^\mathsf{K}\left(\frac{1}{\mathsf{N}}\sum_{i=1}^\mathsf{N}\frac{y_ix_i}{1+\exp(y_i\la x_i,\theta\ra)}\right)=h(\theta).
	\end{align*}
	
	To show Assumption~\ref{assumption2.local Lip+growth} holds, we have, for all $\theta,\theta'\in\mathbb{R}^d$ and $\bar z,\bar z'\in(\mathbb{R}^d\times\{-1,1\})^\mathsf{K}$, that
	\[|F(\theta,\bar z)-F(\theta',\bar z')|=\eta|\theta-\theta'|,\quad|F(\theta,\bar z)|=\eta|\theta|,\]
	which implies that Assumption~\ref{assumption2.local Lip+growth}(1) is satisfied with $\rho=1$, $r=1/2$, and $L_F=K_F=\eta$. Furthermore, let $s(t)\coloneq 1/(1+e^t)$ for $t\in\mathbb{R}$, we have that
	\[G(\theta,\bar z)=-\frac{1}{\mathsf{K}}\sum_{i=1}^\mathsf{K} \bar y_i\bar x_i s(\bar y_i\la \bar x_i,\theta\ra).\]
	Moreover, $s'(t)=-e^t/(1+e^t)^2$, so that $|s'(t)|\leq 1/4$. Let $M_x\coloneq\max_{1\leq i\leq \mathsf{N}}|x_i|$. Applying Cauchy-Schwarz inequality, together with $|s'|\leq 1/4$, $|y_i|=1$, and $|x_i|\leq M_x$ for all $i=1,\dots,\mathsf{N}$, yields
	\begin{align*}
		\e[|G(\theta,U_{\mathbf{z}})-G(\theta',U_{\mathbf{z}})|]&\leq\frac{1}{\mathsf{K}}\sum_{l=1}^\mathsf{K}\e\left[|y_{I_l}x_{I_l}|\,|s(y_{I_l}\la x_{I_l},\theta\ra)-s(y_{I_l}\la x_{I_l},\theta'\ra)|\right]\\
		&\leq \frac{1}{4\mathsf{K}}\sum_{l=1}^\mathsf{K}\e\left[|x_{I_l}|\,|y_{I_l}\la x_{I_l},\theta-\theta'\ra|\right]\\
		&\leq \frac{1}{4\mathsf{K}}\sum_{l=1}^\mathsf{K}\e\left[|x_{I_l}|^2\right]|\theta-\theta'|\\
		&= \frac{1}{4\mathsf{N}}\sum_{i=1}^\mathsf{N}|x_i|^2|\theta-\theta'|\\
		&\leq \frac{1}{4}M_x^2|\theta-\theta'|.
	\end{align*}
	Furthermore, by using $|s|\leq 1$, $|\bar x_i|\leq |\bar z|$ and $|\bar y_i|=1$ for all $i=1,\dots,\mathsf{K}$, we obtain that
	\[|G(\theta,\bar z)|\leq \frac{1}{\mathsf{K}}\sum_{i=1}^\mathsf{K} |\bar y_i \bar x_i|\,|s(\bar y_i\la \bar x_i,\theta\ra)|\leq \frac{1}{\mathsf{K}}\sum_{i=1}^\mathsf{K}|\bar x_i|\leq |\bar z|.\]
	Thus, Assumption~\ref{assumption2.local Lip+growth}(2) is satisfied with $\vartheta=1$, $\rho=1$, $L_G=M_x^2/4$ and $K_G=1$. 
	
	It remains to verify Assumption~\ref{assumption3.strong convexity}. For all $\theta\in\mathbb{R}^d$ and $\bar z\in(\mathbb{R}^d\times\{-1,1\})^\mathsf{K}$, the Jacobian matrix of $H$ with respect to $\theta$ is given by
	\begin{equation}\label{eq.Jacobian of H}
		\nabla_\theta H(\theta,\bar z)=\eta I_d-\frac{1}{\mathsf{K}}\sum_{i=1}^\mathsf{K} (\bar y_i\bar x_i) s'(\bar y_i\la \bar x_i,\theta\ra)(\bar y_i\bar x_i)^\top.
	\end{equation}
	Since $-s'(t)=e^t/(1+e^t)^2\geq 0$, $|\bar y_i|=1$, and $\bar x_i\bar x_i^\top$ is positive semidefinite for all $i=1,\dots,\mathsf{K}$, \eqref{eq.Jacobian of H} can be rewritten as
	\[\nabla_\theta H(\theta,\bar z)=\eta I_d+\frac{1}{\mathsf{K}}\sum_{i=1}^\mathsf{K} \frac{\exp(\bar y_i\la \bar x_i,\theta\ra)}{(1+\exp(\bar y_i\la \bar x_i,\theta\ra))^2}\bar x_i\bar x_i^\top\succeq\eta I_d,\]
	which implies, for all $\theta,\theta'\in\mathbb{R}^d$ and $\bar{z}\in(\mathbb{R}^d\times\{-1,1\})^\mathsf{K}$, that
	\begin{align*}
		\la\theta-\theta',H(\theta,\bar z)-H(\theta',\bar z)\ra&=\int_0^1\la\theta-\theta',\nabla_\theta H(\theta'+t(\theta-\theta'),\bar z)(\theta-\theta')\ra\,\dd t\\
		&\geq \int_0^1\eta|\theta-\theta'|^2\,\dd t=\eta|\theta-\theta'|^2.
	\end{align*}
	Thus, Assumption~\ref{assumption3.strong convexity} is satisfied by taking $A(\bar z)=\eta I_d$, with $m=\eta$.
\end{proof}

\phantomsection

\begin{proof}[\textbf{Proof of Proposition~\ref{prop.Artificial}}]
	\label{proof.Artificial}
	Recall $H$ defined in \eqref{eq.Artificial-SG}, $(X_n)_{n\in\mathbb{N}_0}$ consists of i.i.d.\ standard Gaussian vectors, $\widetilde{A}$ is positive definite and $\widetilde{B}$ is positive semidefinite. Therefore, we obtain, for any $\theta\in\mathbb{R}^d$, that
	\[u(\theta)=\frac{1}{2}\la \theta,\widetilde{A}\theta\ra+\frac{\eta}{4}\la\theta,\widetilde{B}\theta\ra^2\geq 0,\]
	and
	\begin{align*}
		\e[H(\theta,X_0)]&=\e[\widetilde{A}\theta+X_0+\eta\la\theta,\widetilde{B}\theta\ra\widetilde{B}\theta]=\widetilde{A}\theta+\eta\la\theta,\widetilde{B}\theta\ra\widetilde{B}\theta=h(\theta).
	\end{align*}
	Moreover, $u$ is twice continuously differentiable, and $\e[|X_0|^{12\rho r}]<\infty$. 
	
	We next verify Assumption~\ref{assumption2.local Lip+growth}. From the definition of $\widetilde{A}$ and $0<\bar{\varepsilon}<1$, we have, for all $\theta,\theta'\in\mathbb{R}^d$ and $x,x'\in\mathbb{R}^d$, that
	\[|F(\theta,x)-F(\theta',x')|=|\widetilde{A}(\theta-\theta')+(x-x')|\leq |\theta-\theta'|+|x-x'|,\]
	and
	\[|F(\theta,x)|=|\widetilde{A}\theta+x|\leq |\theta|+|x|\leq(1+|x|)(1+|\theta|).\]
	Thus, Assumption~\ref{assumption2.local Lip+growth}(1) is satisfied with $\rho=1$, $r\geq 1$ and $L_F=K_F=1$. Moreover, recall that $\widetilde{B}=vv^\top$ for some fixed $v\in\mathbb{R}^d$, we have, for any $\theta\in\mathbb{R}^d$, that
	\[\la\theta,\widetilde{B}\theta\ra=\la\theta, vv^\top\theta\ra=\la v,\theta\ra^2\quad\text{and}\quad \widetilde{B}\theta=vv^\top\theta=\la v,\theta\ra v,\]
	which implies that
	\begin{equation}\label{eq.bar G}
		G(\theta,x)=\eta\la v,\theta\ra^3 v.
	\end{equation}
	By using Cauchy-Schwarz inequality and \eqref{eq.bar G}, we obtain, for all $\theta,\theta'\in\mathbb{R}^d$, that
	\begin{align*}
		\e[|G(\theta,X_0)-G(\theta',X_0)|]&\leq\eta|\la v,\theta\ra^3-\la v,\theta'\ra^3|\,|v|\\
		&\leq\eta|\la v,\theta-\theta'\ra|\,\left(|\la v,\theta\ra|^2+|\la v,\theta\ra|\,|\la v,\theta'\ra|+|\la v,\theta'\ra|^2\right)|v|\\
		&\leq \eta|v|^2|\theta-\theta'|(|v|^2|\theta|^2+|v|^2|\theta|\,|\theta'|+|v|^2|\theta'|^2)\\
		&\leq \eta|v|^4(1+|\theta|+|\theta'|)^2|\theta-\theta'|,
	\end{align*}
	and
	\[|G(\theta,x)|=|\eta\la v,\theta\ra^3 v|\leq \eta|v|^4|\theta|^3\leq \eta|v|^4(1+|\theta|)^3.\]
	Hence, Assumption~\ref{assumption2.local Lip+growth}(2) holds with $L_G=K_G=\eta|v|^4$, $\rho=1$ and $\vartheta=3$. Since $2r\in[\vartheta,\infty)\cap\mathbb{N}$, we take $r=3/2$. 
	
	It remains to verify Assumption~\ref{assumption3.strong convexity}. For any $\theta,\theta'\in\mathbb{R}^d$ and $x\in\mathbb{R}^d$,
	\begin{equation}\label{eq.Artificial-SG-difference}
		H(\theta,x)-H(\theta',x)=\widetilde{A}(\theta-\theta')+\eta[\la v,\theta\ra^3-\la v,\theta'\ra^3]\,v.
	\end{equation}
	The definition of $\widetilde{A}$ and $0<\bar{\varepsilon}<1$ imply
	\begin{equation}\label{eq.Artificial-convex-1}
		\la\theta-\theta',\widetilde{A}(\theta-\theta')\ra\geq\bar{\varepsilon}|\theta-\theta'|^2.
	\end{equation}
	In addition, 
	\begin{align}
		\left\la\theta-\theta',\eta(\la v,\theta\ra^3-\la v,\theta'\ra^3)v\right\ra&=\eta(\la v,\theta\ra^3-\la v,\theta'\ra^3)\la\theta-\theta',v\ra\notag\\
		&=\eta(\la v,\theta\ra^3-\la v,\theta'\ra^3)(\la v,\theta\ra-\la v,\theta'\ra)\notag\\
		&=\eta(\la v,\theta\ra-\la v,\theta'\ra)^2[\la v,\theta\ra^2+\la v,\theta\ra\la v,\theta'\ra+\la v,\theta'\ra^2]\notag\\
		&\geq 0,\label{eq.Artificial-convex-2}
	\end{align}
	where the inequality follows from $a^2+ab+b^2=(a+b/2)^2+3b^2/4\geq 0$ for any $a,b\in\mathbb{R}$. Substituting \eqref{eq.Artificial-convex-1} and \eqref{eq.Artificial-convex-2} into \eqref{eq.Artificial-SG-difference} yields
	\[\la\theta-\theta',H(\theta,x)-H(\theta',x)\ra\geq \bar{\varepsilon}|\theta-\theta'|^2.\]
	Therefore, Assumption~\ref{assumption3.strong convexity} holds with $A(x)=\bar{\varepsilon} I_d$ and $m=\bar{\varepsilon}$.
\end{proof}

\phantomsection

\begin{proof}[\textbf{Proof of Proposition~\ref{prop.Newsvendor}}]
	\label{proof.Newsvendor}
	We recall that $\bar{A}$ is positive definite, $\kappa>0$, $h_i,s_i>0$ for all $i=1,\ldots,d$, and $X\coloneq(X^1,\dots,X^d)\sim\mathcal{N}(\mu,\Sigma_X)$. Then, we have, for any $\theta\in\mathbb{R}^d$, that
	\begin{align}\label{eq.Newsvendor-u}
		\begin{split}
			u(\theta)&=\frac{1}{2}\la \theta,\bar A\theta\ra+\sum_{i=1}^d\left(h_i\e[\max\{\theta_i-X^i,0\}]+s_i\e[\max\{X^i-\theta_i,0\}]\right)+\frac{\kappa}{4}|\theta|^4\\
			&=\frac{1}{2}\la \theta,\bar A\theta\ra+\sum_{i=1}^d\left(h_i\int_{-\infty}^{\theta_i}(\theta_i-t)f_i(t)\,\dd t+s_i\int_{\theta_i}^\infty(t-\theta_i)f_i(t)\,\dd t\right)+\frac{\kappa}{4}|\theta|^4\geq 0,
		\end{split}
	\end{align}
	where we denote by $f_i$ the marginal density of $X_i$ for $i=1,\dots,d$. By differentiating \eqref{eq.Newsvendor-u} and using $X_0\sim\mathcal{N}(\mu,\Sigma_X)$, we obtain
	\begin{align*}
		h(\theta)&=\bar A\theta+\sum_{i=1}^d\left(h_i\int_{-\infty}^{\theta_i}f_i(t)\,\dd t-s_i\int_{\theta_i}^\infty f_i(t)\,\dd t\right)e_i+\kappa|\theta|^2\theta\\
		&=\bar A\theta+\sum_{i=1}^d\left(h_i\p(X^i<\theta_i)-s_i\p(X^i>\theta_i)\right)e_i+\kappa|\theta|^2\theta.
	\end{align*} 
	Furthermore, recall $H$ in \eqref{eq.Newsvendor-SG}, it holds that, for all $\theta\in\mathbb{R}^d$, 
	\begin{align*}
		\e[H(\theta,X_0)]&=\bar{A}\theta+\sum_{i=1}^d\left(h_i\e[\II_{\{\theta_i>X_0^i\}}]-s_i\e[\II_{\{X_0^i>\theta_i\}}]\right)e_i+\kappa|\theta|^2\theta\\
		&=\bar{A}\theta+\sum_{i=1}^d\left(h_i\p(X_0^i<\theta_i)-s_i\p(X_0^i>\theta_i)\right)e_i+\kappa|\theta|^2\theta=h(\theta).
	\end{align*}
	Moreover, $u$ is twice continuously differentiable, and $\e[|X_0|^{12\rho r}]<\infty$.
	
	We next verify Assumption~\ref{assumption2.local Lip+growth}. Let $\bar c_d$ be a uniform upper bound for $f_i$, $i=1,\ldots,d$. For any $\theta,\theta'\in\mathbb{R}^d$, and $x,x'\in\mathbb{R}^d$, we have 
	\begin{align*}
		|F(\theta,x)-F(\theta',x')|&\leq\|\bar A\|_{\mathrm{op}}|\theta-\theta'|+\kappa||\theta|^2\theta-|\theta'|^2\theta'|\\
		&\leq \|\bar A\|_{\mathrm{op}}|\theta-\theta'|+\kappa\left(|\theta|^2|\theta-\theta'|+||\theta|^2-|\theta'|^2|\,|\theta'|\right)\\
		&\leq \|\bar A\|_{\mathrm{op}}|\theta-\theta'|+\kappa\left(|\theta|^2|\theta-\theta'|+|\theta-\theta'|(|\theta|+|\theta'|)|\theta'|\right)\\
		&\leq (\|\bar A\|_{\mathrm{op}}+\kappa)(1+|\theta|+|\theta'|)^2|\theta-\theta'|,
	\end{align*}
	where $\|\bar A\|_{\mathrm{op}}\coloneq \sup_{|\theta|=1}|\bar A\theta|$ is the operator norm of $\bar A$. Moreover, using $(a+b)^3\leq 4(a^3+b^3)$ for all $a,b\geq 0$, we obtain that
	\begin{align*}
		|F(\theta,x)|&\leq \|\bar A\|_{\mathrm{op}}|\theta|+\kappa|\theta|^3\leq 4(\|\bar A\|_{\mathrm{op}}+\kappa)(1+|\theta|^3).
	\end{align*}
	Thus, Assumption~\ref{assumption2.local Lip+growth}(1) holds with $\rho=1$, $r=3/2$ and $L_F=\|\bar A\|_{\mathrm{op}}+\kappa$ and $K_F=4(\|\bar A\|_{\mathrm{op}}+\kappa)$. Moreover, for any $\theta,\theta'\in\mathbb{R}^d$ and $x\in\mathbb{R}^d$, it holds that
	\begin{align*}
		|G(\theta,x)-G(\theta',x)|&\leq\sum_{i=1}^d h_i|\II_{\{\theta_i>x_i\}}-\II_{\{\theta_i'>x_i\}}|+\sum_{i=1}^d s_i|\II_{\{x_i>\theta_i\}}-\II_{\{x_i>\theta_i'\}}|\\
		&\leq \sum_{i=1}^d (h_i+s_i)\II_{\{\min\{\theta_i,\theta_i'\}\leq x_i\leq\max\{\theta_i,\theta_i'\}\}}.
	\end{align*}
	Taking expectations, using Cauchy-Schwarz inequality, and $\sup_{1\leq i\leq d}|f_i|\leq \bar{c}_d$ for some $\bar{c}_d>0$, we obtain
	\begin{align*}
		\e[|G(\theta,X_0)-G(\theta',X_0)|]&\leq\sum_{i=1}^d (h_i+s_i)\int_{\min\{\theta_i,\theta_i'\}}^{\max\{\theta_i,\theta_i'\}}f_i(t)\,\dd t\\
		&\leq \bar{c}_d\sum_{i=1}^d (h_i+s_i)|\theta_i-\theta_i'|\\
		&\leq \bar{c_d}(|h|+|s|)\,|\theta-\theta'|,
	\end{align*}
	where $h\coloneq (h_1,\dots,h_d)$ and $s\coloneq(s_1,\dots,s_d)$. Furthermore, we have that
	\begin{align*}
		|G(\theta,x)|&\leq \left|\sum_{i=1}^d \left(h_i\II_{\{\theta_i>x_i\}}-s_i\II_{\{x_i>\theta_i\}}\right)e_i\right|\\
		&\leq \left|\sum_{i=1}^d h_i\II_{\{\theta_i>x_i\}}e_i\right|+\left|\sum_{i=1}^ds_i\II_{\{x_i>\theta_i\}}e_i\right|\\
		&=\left(\sum_{i=1}^d h_i^2\II_{\{\theta_i>x_i\}}^2\right)^{1/2}+\left(\sum_{i=1}^d s_i^2\II_{\{x_i>\theta_i\}}^2\right)^{1/2}\\
		&\leq |h|+|s|,
	\end{align*}
	which implies that Assumption~\ref{assumption2.local Lip+growth}(2) is satisfied with $\rho=\vartheta=1$, $L_G=\bar{c}_d(|h|+|s|)$ and $K_G=|h|+|s|$. 
	
	Finally, we verify Assumption~\ref{assumption3.strong convexity}. The mapping $\theta\mapsto|\theta|^4$ is convex and differentiable, which implies that $\la\theta-\theta',|\theta|^2\theta-|\theta'|^2\theta'\ra\geq 0$. In addition, for each $i=1,\dots,d$,
	\[(\theta_i-\theta_i')(\II_{\{\theta_i>x_i\}}-\II_{\{\theta_i'>x_i\}})\geq 0,\]
	and
	\[-(\theta_i-\theta_i')(\II_{\{x_i>\theta_i\}}-\II_{\{x_i>\theta_i'\}})\geq 0.\]
	Consequently, we have
	\begin{align*}
		\la	\theta-\theta',H(\theta,x)-H(\theta',x)\ra&=\la\theta-\theta',\bar{A}(\theta-\theta')\ra+\kappa\la\theta-\theta',|\theta|^2\theta-|\theta'|^2\theta'\ra\\
		&\quad+\sum_{i=1}^d h_i (\theta_i-\theta_i')(\II_{\{\theta_i>x_i\}}-\II_{\{\theta_i'>x_i\}})\\
		&\quad-\sum_{i=1}^d s_i(\theta_i-\theta_i')(\II_{\{x_i>\theta_i\}}-\II_{\{x_i>\theta_i'\}})\\
		&\geq \la\theta-\theta',\bar{A}(\theta-\theta')\ra.
	\end{align*}
	Thus, Assumption~\ref{assumption3.strong convexity} holds with $A(x)=\bar{A}$.
\end{proof}

\section{Proofs of Results in Section~\ref{section.Main Results}}
\label{Proofs of Moment Bounds}

\subsection{Proofs of Lemmas in Section~\ref{subsection.MY}}

\phantomsection
\begin{proof}[\textbf{Proof of Lemma~\ref{lemma.hMY leq h}}]
	\label{proof.hMY leq h}
	By Lemma \ref{lemma.MY}, we have, for all $\theta\in\mathbb{R}^d$, that 
	\begin{equation}\label{eq.hprox}
		h(\operatorname{prox}_u^\epsilon(\theta))=h_{MY,\epsilon}(\theta)=\frac{1}{\epsilon}\left(\theta-\operatorname{prox}_u^\epsilon(\theta)\right).
	\end{equation}
	If $\theta=\operatorname{prox}_u^\epsilon(\theta)$ for some $\theta\in\mathbb{R}^d$, then the desired result holds. Otherwise, assuming $\theta\neq\operatorname{prox}_u^\epsilon(\theta)$, we apply Remark~\ref{remark.dissipativity of h}, \eqref{eq.hprox} and Cauchy-Schwarz inequality to obtain
	\begin{align*}
		|h(\theta)|\,|\theta-\operatorname{prox}_u^\epsilon(\theta)|&\geq\la h(\theta),\theta-\operatorname{prox}_u^\epsilon(\theta)\ra\\
		&=\la h(\theta)-h(\operatorname{prox}_u^\epsilon(\theta)),\theta-\operatorname{prox}_u^\epsilon(\theta)\ra+\la h(\operatorname{prox}_u^\epsilon(\theta)),\theta-\operatorname{prox}_u^\epsilon(\theta)\ra\\
		&\geq m|\theta-\operatorname{prox}_u^\epsilon(\theta)|^2+\frac{1}{\epsilon}|\theta-\operatorname{prox}_u^\epsilon(\theta)|^2\\
		&\geq \frac{1}{\epsilon}|\theta-\operatorname{prox}_u^\epsilon(\theta)|^2.
	\end{align*}
	Combining this lower bound with \eqref{eq.hprox} yields
	\[|h_{MY,\epsilon}(\theta)|= \frac{1}{\epsilon}|\theta-\operatorname{prox}_u^\epsilon(\theta)|\leq |h(\theta)|,\]
	which completes the proof.
\end{proof}

\phantomsection
\begin{proof}[\textbf{Proof of Lemma~\ref{lemma.error of MY regularization}}]
	\label{proof.error of MY regularization}
	By Remark~\ref{remark.local Lip. of h}, Lemma~\ref{lemma.hMY leq h}, \eqref{eq.hprox}, and $h(\theta^*)=0$, we have, for all $\theta\in\mathbb{R}^d$, that
	\begin{align}
		|h(\theta)-h_{MY,\epsilon}(\theta)|&= |h(\theta)-h(\operatorname{prox}_u^\epsilon(\theta))|\notag\\
		&\leq L_h(1+|\theta|+|\operatorname{prox}_u^\epsilon(\theta)|)^{2r-1}|\theta-\operatorname{prox}_u^\epsilon(\theta)|\notag\\
		&\leq L_h(1+|\theta|+|\operatorname{prox}_u^\epsilon(\theta)|)^{2r-1}|h_{MY,\epsilon}(\theta)|\epsilon\notag\\
		&\leq L_h(1+|\theta|+|\operatorname{prox}_u^\epsilon(\theta)|)^{2r-1}|h(\theta)|\epsilon\notag\\
		&=
		L_h(1+|\theta|+
		|\operatorname{prox}_u^\epsilon(\theta)|)^{2r-1}
		|h(\theta)-h(\theta^\ast)|\,\epsilon \notag\\
		&\leq L_h^2(1+|\theta|+|\operatorname{prox}_u^\epsilon(\theta)|)^{2r-1}(1+|\theta|+|\theta^*|)^{2r-1}|\theta-\theta^*|\epsilon\notag\\
		&\leq L_h^2(1+|\theta|+|\operatorname{prox}_u^\epsilon(\theta)|)^{2r-1}(1+|\theta|+|\theta^*|)^{2r}\epsilon.\label{eq.h-MY}
	\end{align}
	According to \cite[Theorem 6.42]{beck2017first}, the proximal operator $\operatorname{prox}_u^\epsilon$ is $1$-Lipschitz continuous. Additionally, Lemma~\ref{lemma.MY} and definition of $\operatorname{prox}_u^\epsilon$ imply that $\operatorname{prox}_u^\epsilon(\theta^*)=\theta^*$. Combining these results with Remark~\ref{remark.dissipativity of h}, we obtain
	\begin{align}
		|\operatorname{prox}_u^\epsilon(\theta)|&\leq|\operatorname{prox}_u^\epsilon(\theta^*)|+|\operatorname{prox}_u^\epsilon(\theta)-\operatorname{prox}_u^\epsilon(\theta^*)|\notag\\
		&\leq|\theta^*|+|\theta-\theta^*|\leq |\theta|+2|\theta^*|\leq |\theta|+2R_0.\label{eq.bound-prox}
	\end{align}
	Therefore, by substituting \eqref{eq.bound-prox} back into \eqref{eq.h-MY}, we obtain
	\begin{align*}
		|h(\theta)-h_{MY,\epsilon}(\theta)|&\leq L_h^2(1+2|\theta|+2R_0)^{2r-1}(1+|\theta|+R_0)^{2r}\epsilon\\
		&\leq L_h^2\big(2(1+|\theta|+R_0)\big)^{2r-1}(1+|\theta|+R_0)^{2r}\epsilon\\
		&\leq 2^{2r-1}L_h^2(1+|\theta|+R_0)^{4r-1}\epsilon,
	\end{align*}
	which completes the proof.
\end{proof}

\phantomsection
\begin{proof}[\textbf{Proof of Lemma \ref{lemma.moment bounds of pi_beta}}]
	\label{proof.moment bounds of pi_beta}
	Let $(x_t)_{t\geq 0}$ be the overdamped Langevin SDE defined by 
	\[\dd x_t=-h(x_t)\,\dd t+\sqrt{2\beta^{-1}}\,\dd B_t,\]
	where $x_0=\theta_0$ and $(B_t)_{t\geq 0}$ is a $d$-dimensional Brownian motion. By Assumption~\ref{assumption1.initial condition}, we note that $\theta_0$ has $2q$-th finite moment. Moreover, by Remark~\ref{remark.dissipativity of h}, $u$ satisfies a dissipativity condition. Hence, we can apply \cite[Lemma A.1]{lim2024non} with $Z_t^\lambda\curvearrowleft x_t$, $\lambda=1$, $p\curvearrowleft q$, $a_h\curvearrowleft m/2$, $b_h\curvearrowleft u(0)$ to obtain, for all $t\geq 0$, that
	\begin{equation}\label{eq.2q-bound overdamped}
		\e[|x_t|^{2q}]\leq \exp\left(-\frac{mq}{2}t\right)\e[|\theta_0|^{2q}]+\left(\frac{4u(0)}{m}+\frac{4(d+2(q-1))}{\beta m}\right)^q.
	\end{equation}
	Since $\mathcal{L}(x_t)$ converges to $\pi_\beta$ in Wasserstein-2 distance \cite[Proposition 1]{durmus2019high}, it also converges weakly to $\pi_\beta=\mathcal{L}(Y)$ \cite[Theorem 6.9]{villani2009optimal}. For any $R>0$, define $f_R(\theta)\coloneq \min\{|\theta|^{2q},R\}$, $\theta\in\mathbb{R}^d$. We have, for any $R>0$, that $f_R$ is nonnegative, bounded and continuous. Thus, by the definition of weak convergence,
	\[\e[f_R(Y)]=\lim_{t\to\infty}\e[f_R(x_t)].\]
	Moreover, since $f_R(\theta)\leq |\theta|^{2q}$ for all $\theta\in\mathbb{R}^d$, we have, for any $R>0$, that
	\begin{equation}\label{eq.f_R}
		\e[f_R(Y)]=\lim_{t\to\infty}\e[f_R(x_t)]\leq\liminf_{t\to\infty}\e[|x_t|^{2q}].
	\end{equation}
	Since $f_R(\theta)\uparrow |\theta|^{2q}$ as $R\to\infty$, by monotone convergence theorem, \eqref{eq.2q-bound overdamped} and \eqref{eq.f_R}, we obtain
	\[\e[|Y|^{2q}]=\lim_{R\to\infty}\e[f_R(Y)]\leq\lim_{R\to\infty}\liminf_{t\to\infty}\e[|x_t|^{2q}]\leq\left(\frac{4u(0)}{m}+\frac{4(d+2(q-1))}{\beta m}\right)^q \eqcolon C_{\pi_\beta,2q},\]
	where $C_{\pi_\beta,2q}=\mathcal{O}((d/\beta)^q)$. Similarly, let $(v_t)_{t\geq 0}$ be the overdamped Langevin SDE defined by
	\[\dd v_t=-h_{MY,\epsilon}(v_t)\,\dd t+\sqrt{2\beta^{-1}}\,\dd B_t,\]
	where $v_0=\theta_0$. By Assumption~\ref{assumption1.initial condition} and Remark~\ref{remark.MY}, we apply \cite[Lemma A.1]{lim2024non} with $Z_t^\lambda\curvearrowleft v_t$, $h\curvearrowleft h_{MY,\epsilon}$, $\lambda=1$, $p\curvearrowleft q$, $a_h\curvearrowleft m/4$, and $b_h\curvearrowleft u(0)$ to obtain, for all $t\geq 0$, that
	\[\e[|v_t|^{2q}]\leq\exp\left(-\frac{mq}{4}t\right)\e[|\theta_0|^{2q}]+\left(\frac{8u(0)}{m}+\frac{8(d+2(q-1))}{\beta m}\right)^q.\]
	By \cite[Proposition 1]{durmus2019high}, $\mathcal{L}(v_t)$ converges to $\pi_\beta^\epsilon$ in Wasserstein-2 distance. Thus, using the same argument gives
	\[\e[|\bar Y|^{2q}]\leq \left(\frac{8u(0)}{m}+\frac{8(d+2(q-1))}{\beta m}\right)^q.\]
	This completes the proof.
\end{proof}

\phantomsection
\begin{proof}[\textbf{Proof of Lemma~\ref{lemma.contraction of pi_beta}}]
	\label{proof.contraction of pi_beta}
	Let $(\widetilde{x}_t)_{t\geq 0}$ be the overdamped Langevin SDE defined by
	\begin{equation}\label{eq.overdamped x}
		\dd \widetilde{x}_t=-h(\widetilde{x}_t)\,\dd t+\sqrt{2\beta^{-1}}\,\dd B_t
	\end{equation}
	and let $(\widetilde{v}_t)_{t\geq 0}$ be defined by
	\[\dd \widetilde{v}_t=-h_{MY,\epsilon}(\widetilde{v}_t)\,\dd t+\sqrt{2\beta^{-1}}\,\dd B_t,\]
	where both processes are driven by the same Brownian motion $(B_t)_{t\geq 0}$ and have the same initial condition $\widetilde{x}_0=\widetilde{v}_0=\widetilde{\theta}_0$ satisfying $\mathcal{L}(\widetilde{\theta}_0)=\pi_\beta$. By Itô's formula, we obtain, for all $t\geq 0$, that
	\[\e[|\widetilde{x}_t-\widetilde{v}_t|^2]=-2\int_0^t\e[\la \widetilde{x}_s-\widetilde{v}_s,h(\widetilde{x}_s)-h_{MY,\epsilon}(\widetilde{v}_s)\ra]\,\dd s.\]
	Differentiating with respect to $t$, and then using Young's inequality together with Remark~\ref{remark.MY}, yields
	\begin{align*}\label{eq.time-derivative x-y}
		\begin{split}
			\frac{\dd}{\dd t}\e[|\widetilde{x}_t-\widetilde{v}_t|^2]&=-2\e[\la \widetilde{x}_t-\widetilde{v}_t,h(\widetilde{x}_t)-h_{MY,\epsilon}(\widetilde{v}_t)\ra]\\
			&=-2\e[\la \widetilde{x}_t-\widetilde{v}_t,h(\widetilde{x}_t)-h_{MY,\epsilon}(\widetilde{x}_t)\ra]-2\e[\la \widetilde{x}_t-\widetilde{v}_t,h_{MY,\epsilon}(\widetilde{x}_t)-h_{MY,\epsilon}(\widetilde{v}_t)\ra]\\
			&\leq \frac{m}{2}\e[|\widetilde{x}_t-\widetilde{v}_t|^2]+\frac{2}{m}\e[|h(\widetilde{x}_t)-h_{MY,\epsilon}(\widetilde{x}_t)|^2]-m\e[|\widetilde{x}_t-\widetilde{v}_t|^2]\\
			&\leq -\frac{m}{2}\e[|\widetilde{x}_t-\widetilde{v}_t|^2]+\frac{2}{m}\e[|h(\widetilde{x}_t)-h_{MY,\epsilon}(\widetilde{x}_t)|^2].
		\end{split}
	\end{align*}
	Then, by using Lemma~\ref{lemma.error of MY regularization} and \eqref{eq.2q-bound overdamped}, we obtain that
	\begin{align*}
		\frac{2}{m}\e[|h(\widetilde{x}_t)-h_{MY,\epsilon}(\widetilde{x}_t)|^2]&\leq \frac{2}{m} 2^{4r-2}L_h^4\e\left[(1+|\widetilde{x}_t|+R_0)^{8r-2}\right]\epsilon^2\\
		&\leq \frac{2^{4r-1}}{m}L_h^4 2^{8r-3}\left((1+R_0)^{8r-2}+\e\left[|\widetilde{x}_t|^{8r-2}\right]\right)\epsilon^2\\
		&\leq\frac{2^{12r-4}}{m} L_h^4\left((1+R_0)^{8r-2}+C_{\pi_\beta,8r-2}\right)\epsilon^2\leq c_\pi\epsilon^2,
	\end{align*}
	where $c_\pi\coloneq 2^{12r-4}L_h^4\left((1+R_0)^{8r-2}+C_{\pi_\beta,8r-2}\right)/m=\mathcal{O}((d/\beta)^{4r-1})$, and where $C_{\pi_\beta,8r-2}$ is given in Lemma~\ref{lemma.moment bounds of pi_beta}. Consequently, it holds that
	\begin{align*}
		\frac{\dd}{\dd t}\e[|\widetilde{x}_t-\widetilde{v}_t|^2]&\leq  -\frac{m}{2}\e[|\widetilde{x}_t-\widetilde{v}_t|^2]+c_\pi\epsilon^2,
	\end{align*}
	which implies
	\begin{align*}
		\e[|\widetilde{x}_t-\widetilde{v}_t|^2]&\leq c_\pi\epsilon^2\frac{2}{m}(1-e^{-mt/2})\leq \frac{2c_\pi}{m}\epsilon^2.
	\end{align*}
	By using the definitions of $(\widetilde{x}_t)_{t\geq 0}$ and $(\widetilde{v}_t)_{t\geq 0}$, we have, for any $t\geq 0$, that
	\[W_2(\pi_\beta,\mathcal{L}(\widetilde{v}_t))=W_2(\mathcal{L}(\widetilde{x}_t),\mathcal{L}(\widetilde{v}_t))\leq\left(\e[|\widetilde{x}_t-\widetilde{v}_t|^2]\right)^{1/2}\leq\sqrt{\frac{2c_\pi}{m}}\epsilon.\]
	Therefore, we obtain that
	\[W_2(\pi_\beta,\pi_\beta^\epsilon)\leq W_2(\pi_\beta,\mathcal{L}(\widetilde{v}_t))+W_2(\mathcal{L}(\widetilde{v}_t),\pi_\beta^\epsilon)\leq \bar{c}\epsilon+W_2(\mathcal{L}(\widetilde{v}_t),\pi_\beta^\epsilon),\]
	where $\bar{c}\coloneq \sqrt{2c_\pi/m}=\mathcal{O}((d/\beta)^{2r-1/2})$. By \cite[Proposition 1]{durmus2019high}, $\widetilde{v}_t$ converges to its invariant measure $\pi_\beta^\epsilon$ in Wasserstein-2 distance. Hence, taking the limit as $t\to\infty$ and noting the inequality $W_2(\Pi_\beta,\Pi_\beta^\epsilon)\leq W_2(\pi_\beta,\pi_\beta^\epsilon)$ yields the desired result.
\end{proof}

\subsection{Proofs of Results in Section~\ref{subsection.estimates}}

\begin{lemma}\label{Lemma.dissipativity of Expectation of SG}
	Let Assumptions~\ref{assumption1.initial condition}-\ref{assumption3.strong convexity} hold. Then, for all $\theta\in\mathbb{R}^d$ and $x\in\mathbb{R}^m$, $H_\gamma$ defined in \eqref{eq.tamedSG} satisfies
	\[\la\theta,H_\gamma(\theta,x)\ra\geq -m|\theta|^2-\widetilde{b}(x),\]
	where $\widetilde{b}(x)=K_H^2(1+|x|)^{2\rho}/(4m)$, see also Remark~\ref{remark.dissipativitiy of SG}. This implies
	\[\la \e[H_\gamma(\theta,X_0)],\theta\ra\geq\frac{m}{2}|\theta|^2-u(0).\]
\end{lemma}

\begin{proof}
	By the definition of $H_\gamma$ in \eqref{eq.tamedSG}, together with Remark~\ref{remark.dissipativitiy of SG} and  \eqref{eq.positive semi definited A}, we have that
	\begin{align*}
		\la\theta,H_\gamma(\theta,x)\ra&=m|\theta|^2+\frac{\la\theta,H(\theta,x)\ra-m|\theta|^2}{\sqrt{1+\gamma^{-1}|\theta|^{4r}}}\\
		&\geq\left(1-\frac{1}{\sqrt{1+\gamma^{-1}|\theta|^{4r}}}\right)m|\theta|^2+\frac{\la\theta,A(x)\theta\ra}{\sqrt{1+\gamma^{-1}|\theta|^{4r}}}-\frac{m|\theta|^2+\widetilde{b}(x)}{\sqrt{1+\gamma^{-1}|\theta|^{4r}}}\\
		&\geq -m|\theta|^2-\widetilde{b}(x).
	\end{align*}
	Moreover, By using the definition of $H_\gamma$ in \eqref{eq.tamedSG} again, we have, for all $\theta\in\mathbb{R}^d$, that
	\[\la \e[H_\gamma(\theta,X_0)],\theta\ra=\left\la\frac{\e[H(\theta,X_0)]-m\theta}{\sqrt{1+\gamma^{-1}|\theta|^{4r}}},\theta\right\ra+m|\theta|^2.\]
	First, we consider the case where $\la\e[H(\theta,X_0)]-m\theta,\theta\ra\geq 0$. Then, it follows that
	\[\left\la\frac{\e[H(\theta,X_0)]-m\theta}{\sqrt{1+\gamma^{-1}|\theta|^{4r}}},\theta\right\ra\geq 0,\]
	which implies that
	\[\la\e[H_\gamma(\theta,X_0)],\theta\ra\geq m|\theta|^2\geq\frac{m}{2}|\theta|^2-u(0).\]
	Next, we consider $\la\e[H(\theta,X_0)]-m\theta,\theta\ra< 0$, which implies that
	\[\left\la\frac{\e[H(\theta,X_0)]-m\theta}{\sqrt{1+\gamma^{-1}|\theta|^{4r}}},\theta\right\ra\geq \la\e[H(\theta,X_0)]-m\theta,\theta\ra.\]
	Thus we obtain
	\[\la\e[H_\gamma(\theta,X_0)],\theta\ra\geq\la\e[H(\theta,X_0)],\theta\ra=\la h(\theta),\theta\ra\geq\frac{m}{2}|\theta|^2-u(0),\]
	where the last inequality follows from Remark~\ref{remark.dissipativity of h}.
\end{proof}

\begin{lemma}\label{lemma.2nd bound of taming factor}
	Let Assumptions~\ref{assumption1.initial condition}-\ref{assumption3.strong convexity} hold. Then, for all $\theta\in\mathbb{R}^d$, $x\in\mathbb{R}^k$, and $\gamma\geq\gamma_{\min}$, we have
	\[|H_\gamma(\theta,x)|^2\leq 2m^2|\theta|^2+8\gamma\left(m^2+K_H^2(1+|x|)^{2\rho}\right),\]
	which implies
	\[\e[|H_\gamma(\theta,X_0)|^2]\leq 2m^2|\theta|^2+\gamma C_H,\]
	where $C_H\coloneq 8(m^2+K_H^2\e[(1+|X_0|)^{2\rho}])$ with $K_H$ given in Remark~\ref{remark.local Lip. of h}.
\end{lemma}

\begin{proof}
	Recall $H_\gamma$ defined in \eqref{eq.tamedSG} and $\gamma\geq\gamma_{\min}\geq 1$. For all $\theta\in\mathbb{R}^d$ and $x\in\mathbb{R}^k$, by using Remark~\ref{remark.local Lip. of h}, we obtain that
	\begin{align*}
		|H_\gamma(\theta,x)|&\leq m|\theta|+\frac{|H(\theta,x)|+m|\theta|}{\sqrt{1+\gamma^{-1}|\theta|^{4r}}}\\
		&\leq m|\theta|+\frac{\sqrt{2}(|H(\theta,x)|+m|\theta|)}{1+\gamma^{-1/2}|\theta|^{2r}}\\
		&\leq m|\theta|+\frac{\sqrt{2}\,m|\theta|}{\gamma^{-1/2}(1+|\theta|^{2r})}+\frac{\sqrt{2}\,|H(\theta,x)|}{\gamma^{-1/2}(1+|\theta|^{2r})}\\
		&\leq m|\theta|+\frac{\sqrt{2}m}{\gamma^{-1/2}}+\frac{\sqrt{2}K_H(1+|x|)^{\rho}}{\gamma^{-1/2}},
	\end{align*}
	where the second inequality follows from the inequality $(\sqrt{a}+\sqrt{b})/\sqrt{2}\leq\sqrt{a+b}$ for all $a,b\geq 0$. Taking squares on both sides yields
	\[|H_\gamma(\theta,x)|^2\leq 2m^2|\theta|^2+8\gamma\left(m^2+K_H^2(1+|x|)^{2\rho}\right),\]
	which establishes the desired result.
\end{proof}

To obtain moment estimates for tSGHMC \eqref{tSGHMC}-\eqref{eq.tamedSG}, we define the Lyapunov function
\begin{equation}\label{M_n}
	M_n\coloneq \frac{\gamma^2}{4}|\theta_n^\lambda+\gamma^{-1}\nu_n^\lambda|^2+\frac{1}{4}|\nu_n^\lambda|^2-\frac{\sigma\gamma^2}{4}|\theta_n^\lambda|^2,
\end{equation}
where $\sigma\coloneq m/(8\gamma^2)$.

\phantomsection
\begin{proof}[\textbf{Proof of Lemma~\ref{lemma.2 moment of algorithm}}]
	\label{Proof.second moment bounds}
	Recall that $\gamma_{\min}$ and $\lambda_{\max,\gamma}$ are defined in \eqref{eq.friction restriction} and \eqref{eq.step-size restriction}, respectively. For any $\gamma\geq\gamma_{\min}$, $\lambda\leq\lambda_{\max,\gamma}$, and $n\in\mathbb{N}_0$, define 
	\begin{equation}\label{eq.Delta n and En}
		\Delta_n\coloneq \theta_n^\lambda + \gamma^{-1}\nu_n^\lambda - \lambda\gamma^{-1} H_\gamma(\theta_n^\lambda,X_{n+1}),\quad E_n \coloneq  \nu_n^\lambda - \lambda[\gamma \nu_n^\lambda+H_\gamma(\theta_n^\lambda,X_{n+1})].
	\end{equation}
	By using \eqref{tSGHMC}-\eqref{eq.tamedSG}, we obtain that
	\begin{equation}\label{eq.expansion of theta gamma nu}
		|\theta_{n+1}^\lambda+\gamma^{-1}\nu_{n+1}^\lambda|^2=|\Delta_n|^2
		+2\sqrt{\frac{2\lambda}{\gamma\beta}}\la \Delta_n,\xi_{n+1}\ra
		+\frac{2\lambda}{\gamma\beta}|\xi_{n+1}|^2.
	\end{equation}
	Furthermore, by using \eqref{eq.Delta n and En}, Young's inequality and Cauchy-Schwarz inequality, we have
	\begin{align}\label{eq.upper bound of Delta_n}
		\begin{split}
			|\Delta_n|^2
			&=|\theta_n^\lambda+\gamma^{-1}\nu_n^\lambda|^2-2\lambda\gamma^{-1}\la \theta_n^\lambda+\gamma^{-1}\nu_n^\lambda, H_\gamma(\theta_n^\lambda,X_{n+1})\ra+\lambda^2\gamma^{-2}|H_\gamma(\theta_n^\lambda,X_{n+1})|^2\\
			&\leq |\theta_n^\lambda+\gamma^{-1}\nu_n^\lambda|^2 -2\lambda\gamma^{-1}{\la \theta_n^\lambda, H_\gamma(\theta_n^\lambda,X_{n+1})\ra}\\
			&\quad+2\lambda\gamma^{-2}|\nu_n^\lambda|\,|H_\gamma(\theta_n^\lambda,X_{n+1})|+\lambda^2\gamma^{-2}|H_\gamma(\theta_n^\lambda,X_{n+1})|^2\\
			&\leq|\theta_n^\lambda+\gamma^{-1}\nu_n^\lambda|^2 +\lambda(2\gamma)^{-1}|\nu_n^\lambda|^2+2\lambda\gamma^{-3}|H_\gamma(\theta_n^\lambda,X_{n+1})|^2\\
			&\quad-2\lambda\gamma^{-1}\la \theta_n^\lambda, H_\gamma(\theta_n^\lambda,X_{n+1})\ra+\lambda^2\gamma^{-2}|H_\gamma(\theta_n^\lambda,X_{n+1})|^2.
		\end{split}
	\end{align}
	Then, by using \eqref{eq.expansion of theta gamma nu} and the last term on the right hand side of \eqref{eq.upper bound of Delta_n}, we obtain that
	\begin{align}\label{eq.M_n-1}
		\begin{split}
			&\frac{\gamma^2}{4}\big(|\theta_{n+1}^\lambda+\gamma^{-1}\nu_{n+1}^\lambda|^2-|\theta_n^\lambda+\gamma^{-1}\nu_n^\lambda|^2\big)\\
			&\leq\frac{\lambda \gamma}{8}|\nu_n^\lambda|^2+\frac{\lambda}{2\gamma}|H_\gamma(\theta_n^\lambda,X_{n+1})|^2\\
			&\quad-\frac{\lambda\gamma}{2}\la \theta_n^\lambda, H_\gamma(\theta_n^\lambda,X_{n+1})\ra+\frac{\lambda^2}{4}|H_\gamma(\theta_n^\lambda,X_{n+1})|^2\\
			&\quad+\gamma^2\sqrt{\frac{\lambda}{2\gamma\beta}}\la \Delta_n,\xi_{n+1}\ra
			+\frac{\lambda\gamma}{2\beta}|\xi_{n+1}|^2.
		\end{split}
	\end{align}
	Similarly, by using \eqref{tSGHMC} and \eqref{eq.Delta n and En}, we have that
	\begin{align*}
		\nu_{n+1}^\lambda&=(1-\lambda\gamma)\nu_n^\lambda-\lambda H_\gamma(\theta_n^\lambda,X_{n+1})+\sqrt{2\lambda\gamma\beta^{-1}}\xi_{n+1}\\
		&=E_n+\sqrt{2\lambda\gamma\beta^{-1}}\xi_{n+1}.
	\end{align*}
	Further calculations yield
	\begin{align}\label{eq.v n+1 2}
		\begin{split}
			|\nu_{n+1}^\lambda|^2&=|E_n|^2+2\lambda\gamma\beta^{-1}|\xi_{n+1}|^2+2\sqrt{2\lambda\gamma\beta^{-1}}\la E_n,\xi_{n+1}\ra\\
			&=(1-\lambda\gamma)^2|\nu_n^\lambda|^2+\lambda^2|H_\gamma(\theta_n^\lambda,X_{n+1})|^2+2\lambda\gamma\beta^{-1}|\xi_{n+1}|^2\\
			&\quad -2\lambda(1-\lambda\gamma)\la\nu_n^\lambda,H_\gamma(\theta_n^\lambda,X_{n+1})\ra+2\sqrt{2\lambda\gamma\beta^{-1}}\la E_n,\xi_{n+1}\ra.
		\end{split}
	\end{align}
	By using \eqref{eq.Delta n and En}, \eqref{eq.v n+1 2}, and Young's inequality, we have
	\begin{align}\label{eq.M_n-2}
		\begin{split}
			\frac{1}{4}(|\nu_{n+1}^\lambda|^2-|\nu_n^\lambda|^2)
			&\leq -\frac{\lambda\gamma}{2}|\nu_n^\lambda|^2+\frac{\lambda^2\gamma^2}{4}|\nu_n^\lambda|^2+\frac{\lambda}{2}|\la\nu_n^\lambda,H_\gamma(\theta_n^\lambda,X_{n+1})\ra|\\
			&\quad+\frac{\lambda^2}{4}|H_\gamma(\theta_n^\lambda,X_{n+1})|^2+\frac{\lambda\gamma}{2\beta}|\xi_{n+1}|^2+\sqrt{\frac{\lambda\gamma}{2\beta}}\la E_n,\xi_{n+1}\ra\\
			&\leq -\frac{\lambda\gamma}{4}|\nu_n^\lambda|^2+\frac{\lambda^2\gamma^2}{4}|\nu_n^\lambda|^2+\frac{\lambda}{4\gamma}|H_\gamma(\theta_n^\lambda,X_{n+1})|^2\\
			&\quad+\frac{\lambda^2}{4}|H_\gamma(\theta_n^\lambda,X_{n+1})|^2+\frac{\lambda\gamma}{2\beta}|\xi_{n+1}|^2+\sqrt{\frac{\lambda\gamma}{2\beta}}\la E_n,\xi_{n+1}\ra.
		\end{split}
	\end{align}
	Furthermore, by using the definition of $M_n$ in \eqref{M_n}, we have that
	\begin{align*}
		M_n&=\frac{\gamma^2}{4}\left(|\theta_n^\lambda|^2+2\gamma^{-1}\la\theta_n^\lambda,\nu_n^\lambda\ra+\gamma^{-2}|\nu_n^\lambda|^2\right)+\frac{1}{4}|\nu_n^\lambda|^2-\frac{\sigma\gamma^2}{4}|\theta_n^\lambda|^2\\
		&\leq \frac{\gamma^2}{4}|\theta_n^\lambda|^2+\frac{\gamma}{2}\la\theta_n^\lambda,\nu_n^\lambda\ra+\frac{1}{2}|\nu_n^\lambda|^2.
	\end{align*}
	This, together with \eqref{tSGHMC}-\eqref{eq.tamedSG}, yields
	\begin{align}
		-\frac{\sigma\gamma^2}{4}(|\theta_{n+1}^\lambda|^2-|\theta_n^\lambda|^2)&=-\frac{\sigma\gamma^2}{4}(\lambda^2|\nu_n^\lambda|^2+2\lambda\la\theta_n^\lambda,\nu_n^\lambda\ra)\notag\\
		&\leq -\frac{\sigma\lambda^2\gamma^2}{4}|\nu_n^\lambda|^2-\lambda\sigma\gamma M_n+\frac{\lambda\sigma\gamma}{2}|\nu_n^\lambda|^2+\frac{\lambda\sigma\gamma^3}{4}|\theta_n^\lambda|^2\notag\\
		&\leq -\lambda\sigma\gamma M_n+\lambda\sigma\gamma^3|\theta_n^\lambda|^2+\frac{\lambda\sigma\gamma}{2}|\nu_n^\lambda|^2.\label{eq.M_n-3}
	\end{align}
	Using \eqref{M_n}, \eqref{eq.M_n-1}, \eqref{eq.M_n-2}, and \eqref{eq.M_n-3} yields
	\begin{align}\label{eq.M_n+1}
		\begin{split}
			M_{n+1}&\leq (1-\lambda\sigma\gamma)M_n+\lambda\left(\frac{\gamma}{8}-\frac{\gamma}{4}+\frac{\lambda\gamma^2}{4}+\frac{\sigma\gamma}{2}\right)|\nu_n^\lambda|^2+\lambda\sigma\gamma^3|\theta_n^\lambda|^2 \\
			&\quad+\lambda\left(\frac{1}{2\gamma}+\frac{\lambda}{2}+\frac{1}{4\gamma}\right){|H_\gamma(\theta_n^\lambda,X_{n+1})|^2}-\frac{\lambda \gamma}{2}{\la\theta_n^\lambda,H_\gamma(\theta_n^\lambda,X_{n+1})\ra}\\
			&\quad+\gamma^2\sqrt{\frac{\lambda}{2\gamma\beta}}\la\Delta_n,\xi_{n+1}\ra+\frac{\lambda\gamma}{\beta}|\xi_{n+1}|^2+\sqrt{\frac{\lambda\gamma}{2\beta}}\la E_n,\xi_{n+1}\ra. 
		\end{split}
	\end{align}
	Taking conditional expectations on both sides of \eqref{eq.M_n+1}, denoted as $\e_n[\cdot]=\e[\cdot\mid\theta_n^\lambda,\nu_n^\lambda]$ for notational simplicity, and using Lemmas~\ref{Lemma.dissipativity of Expectation of SG} and \ref{lemma.2nd bound of taming factor} along with the fact that $\Delta_n$ and $E_n$ are independent of $\xi_{n+1}$, yields
	\begin{align}\label{eq.coeffs of theta nu}
		\begin{split}
			\e_n[M_{n+1}]&\leq (1-\lambda\sigma\gamma)M_n+\lambda\left(-\frac{\gamma}{8}+\frac{\lambda\gamma^2}{4}+\frac{\sigma\gamma}{2}\right)|\nu_n^\lambda|^2+\lambda\sigma\gamma^3|\theta_n^\lambda|^2\\
			&\quad+\lambda\left(\frac{\lambda}{2}+\frac{3}{4\gamma}\right)\left(2m^2|\theta_n^\lambda|^2+C_H\gamma\right)-\frac{\lambda\gamma}{2}\left(\frac{m}{2}|\theta_n^\lambda|^2-u(0)\right)+\frac{\lambda\gamma}{\beta}d\\
			&= (1-\lambda\sigma\gamma)M_n+\lambda\left(-\frac{\gamma}{8}+\frac{\lambda\gamma^2}{4}+\frac{\sigma\gamma}{2}\right)|\nu_n^\lambda|^2\\
			&\quad+\lambda\left[\sigma\gamma^3+2m^2\left(\frac{\lambda}{2}+\frac{3}{4\gamma}\right)-\frac{\gamma m}{4}\right]|\theta_n^\lambda|^2\\
			&\quad +\lambda\left[C_H\gamma\left(\frac{\lambda}{2}+\frac{3}{4\gamma}\right)+\frac{\gamma}{2}u(0)+\frac{\gamma}{\beta}d\right].
		\end{split}
	\end{align}
	Recall that $\sigma=m/(8\gamma^2)$. Since $\gamma\geq \gamma_{\min}\geq \max\{1,14m\}$ and $\lambda\leq \lambda_{\max,\gamma}\leq 1/(4\gamma)\leq 1/4$, the coefficients of $|\nu_n^\lambda|^2$ and $|\theta_n^\lambda|^2$ in \eqref{eq.coeffs of theta nu} are non-positive, i.e.,
	\begin{equation}\label{eq.non-positive coefficients of theta and nu}
		\lambda\left(-\frac{\gamma}{8}+\frac{\lambda\gamma^2}{4}+\frac{\sigma\gamma}{2}\right)\leq 0,\quad \lambda\left[\sigma\gamma^3+2m^2\left(\frac{\lambda}{2}+\frac{3}{4\gamma}\right)-\frac{\gamma m}{4}\right]\leq 0.
	\end{equation}
	Therefore, by using \eqref{eq.non-positive coefficients of theta and nu}, $\gamma^{-1}\leq 1$, and $\lambda\leq 1/(4\gamma)\leq 1/4$, we obtain that
	\begin{align}
		\e_n[M_{n+1}]&\leq\left(1-\frac{\lambda m}{8\gamma}\right)M_n+\lambda\gamma\left[C_H\left(\frac{\lambda}{2}+\frac{3}{4\gamma}\right)+\frac{u(0)}{2}+\frac{d}{\beta}\right]\notag\\
		&\leq\left(1-\frac{\lambda m}{8\gamma}\right)M_n+\lambda\gamma\left[\frac{7C_H}{8}+\frac{u(0)}{2}+\frac{d}{\beta}\right].\label{eq.expectation Lyapunov}
	\end{align}
	Consequently, by induction and by taking expectations, we obtain
	\begin{equation}\label{expectation of M_{n+1}}
		\e[M_{n+1}]\leq \left(1-\frac{\lambda m}{8\gamma}\right)^{n+1} \e[M_0]+\frac{8\gamma^2}{m}\left[\frac{7C_H}{8}+\frac{u(0)}{2}+\frac{d}{\beta}\right].
	\end{equation}
	Moreover, since for $\gamma\geq\gamma_{\min}\geq \sqrt{m}$, we have that 
	\begin{equation}\label{eq.1-2sigma}
		1-2\sigma\geq\frac{3}{4}>0,
	\end{equation}
	and therefore, we have 
	\begin{align}
		M_n&=\frac{\gamma^2}{4}\left(|\theta_n^\lambda+\gamma^{-1}\nu_n^\lambda|^2+|\gamma^{-1}\nu_n^\lambda|^2-\sigma|\theta_n^\lambda+\gamma^{-1}\nu_n^\lambda-\gamma^{-1}\nu_n^\lambda|^2\right)\notag\\
		&\geq \frac{\gamma^2}{4}\left(|\theta_n^\lambda+\gamma^{-1}\nu_n^\lambda|^2+|\gamma^{-1}\nu_n^\lambda|^2-2\sigma(|\theta_n^\lambda+\gamma^{-1}\nu_n^\lambda|^2+|\gamma^{-1}\nu_n^\lambda|^2)\right)\notag\\
		&=\frac{\gamma^2(1-2\sigma)}{4}|\theta_n^\lambda+\gamma^{-1}\nu_n^\lambda|^2+\frac{1-2\sigma}{4}|\nu_n^\lambda|^2\notag\\
		&\geq \max\left\{\frac{(1-2\sigma)\gamma^2}{8}|\theta_n^\lambda|^2,\frac{1-2\sigma}{4}|\nu_n^\lambda|^2\right\},\label{eq.lower bound of M_n}
	\end{align}
	where the last inequality follows from the inequality $|a+b|^2\geq |a|^2/2-|b|^2$ for $a,b>0$. This further implies $M_n\geq 0$ for all $n\in\mathbb{N}_0$. In addition,
	\begin{equation}\label{eq.M_0}
		M_0=\frac{\gamma^2}{4}|\theta_0+\gamma^{-1}\nu_0|^2+\frac{1}{4}|\nu_0|^2-\frac{\sigma\gamma^2}{4}|\theta_0|^2.
	\end{equation}
	Then, for $\gamma\geq\gamma_{\min}\geq 1$, combining \eqref{expectation of M_{n+1}}-\eqref{eq.M_0}, we have that
	\begin{align*}
		\e[|\theta_n^\lambda|^2]&\leq \frac{8}{(1-2\sigma)\gamma^2}\e[M_n]\\
		&\leq \frac{8}{(1-2\sigma)\gamma^2}\e[M_0]+\frac{64}{(1-2\sigma)m}\left[\frac{7C_H}{8}+\frac{u(0)}{2}+\frac{d}{\beta}\right]\\
		&\leq \frac{8}{3}\e[|\theta_0+\gamma^{-1}\nu_0|^2]+\frac{8}{3}\e[|\nu_0|^2]+\frac{256}{3m}\left[\frac{7C_H}{8}+\frac{u(0)}{2}+\frac{d}{\beta}\right]\\
		&\leq 6\e[|\theta_0|^2]+8\e[|\nu_0|^2]+\frac{86}{m}\left[\frac{7C_H}{8}+\frac{u(0)}{2}+\frac{d}{\beta}\right].
	\end{align*}
	Thus, we obtain that
	\begin{equation}\label{eq.sup-position}
		\sup_{n\in\mathbb{N}_0}\e[|\theta_n^\lambda|^2]\leq \bar{C}_2,
	\end{equation}
	where $\bar{C}_2\coloneq 6\e[|\theta_0|^2]+8\e[|\nu_0|^2]+86\big(7C_H/8+u(0)/2+d/\beta\big)/m=\mathcal{O}(d/\beta)$.
	
	Finally, we establish a bound for the second moment of $\nu_n^\lambda$. By using \eqref{eq.v n+1 2}, we obtain that
	\begin{align*}
		|\nu_{n+1}^\lambda|^2&\leq (1-\lambda\gamma)|\nu_n^\lambda|^2+(\lambda^2+2\lambda\gamma^{-1})|H_\gamma(\theta_n^\lambda,X_{n+1})|^2+\frac{\lambda\gamma}{2}|\nu_n^\lambda|^2\\
		&\quad+2\lambda\gamma\beta^{-1}|\xi_{n+1}|^2+2\sqrt{2\lambda\gamma\beta^{-1}}\la E_n,\xi_{n+1}\ra.
	\end{align*}
	By taking conditional expectations, using Lemma~\ref{lemma.2nd bound of taming factor}, \eqref{eq.sup-position}, $\gamma\geq\gamma_{\min}\geq 1$ and $\lambda\leq\lambda_{\max,\gamma}\leq 1/(4\gamma)\leq 1/4$, we obtain that
	\begin{align*}
		\e_n[|\nu_{n+1}^\lambda|^2]&\leq\left(1-\frac{\lambda\gamma}{2}\right)|\nu_n^\lambda|^2+(\lambda^2+2\lambda\gamma^{-1})(2m^2|\theta_n^\lambda|^2+C_H\gamma)+\frac{2\lambda\gamma}{\beta}d\\
		&\leq \left(1-\frac{\lambda\gamma}{2}\right)|\nu_n^\lambda|^2+\lambda\gamma\left[2m^2(\lambda\gamma^{-1}+2\gamma^{-2})\bar{C}_2+(\lambda+2\gamma^{-1})C_H+\frac{2d}{\beta}\right]\\
		&\leq \left(1-\frac{\lambda\gamma}{2}\right)|\nu_n^\lambda|^2+\lambda\gamma\left[\frac{9m^2}{2}\bar{C}_2+\frac{9}{4}C_H+\frac{2d}{\beta}\right].
	\end{align*}
	Iterating the above inequality and taking expectations yield,
	\[\e[|\nu_{n+1}^\lambda|^2]\leq \left(1-\frac{\lambda\gamma}{2}\right)^{n+1}\e[|\nu_0|^2]+9m^2\bar{C}_2+\frac{9}{2}C_H+\frac{4d}{\beta}.\]
	So we conclude that
	\[\sup_{n\in\mathbb{N}_0}\e[|\nu_n^\lambda|^2]\leq\bar{B}_2,\]
	where $\bar{B}_2\coloneq \e[|\nu_0|^2]+9m^2\bar{C}_2+9C_H/2+4d/\beta=\mathcal{O}(d/\beta)$. This completes the proof.
\end{proof}

\phantomsection
\begin{proof}[\textbf{Proof of Lemma~\ref{lemma.2q moment of algorithm}}]
	\label{Proof.Even Order Moment bounds}
	Recall that $\gamma_{\min}$ and $\lambda_{\max,\gamma}$ are defined in \eqref{eq.friction restriction} and \eqref{eq.step-size restriction}, respectively. For any $q\in[2,\infty)\cap\mathbb{N}$, $\gamma\geq\gamma_{\min}$, $\lambda\leq\lambda_{\max,\gamma}$, and $n\in\mathbb{N}_0$, recall the definition of $\Delta_n$ and $E_n$ defined in \eqref{eq.Delta n and En}. Since $\sigma=m/(8\gamma^2)$, $\gamma\geq\gamma_{\min}\geq \max\{1,14m,\sqrt{m}\}$ and $\lambda\leq\lambda_{\max,\gamma}\leq 1/(4\gamma)$, the coefficients of $|\nu_n^\lambda|^2$ in \eqref{eq.M_n+1} are non-positive, i.e.,
	\[\lambda\left(\frac{\gamma}{8}-\frac{\gamma}{4}+\frac{\lambda\gamma^2}{4}+\frac{\sigma\gamma}{2}\right)\leq 0,\]
	and hence
	\begin{align}\label{upper bound of M_{n+1}}
		M_{n+1}\leq (1-\lambda\sigma\gamma)M_n+K_n, 
	\end{align}
	where
	\begin{align}\label{eq.K_n}
		\begin{split}
			K_n&\coloneq \lambda\sigma\gamma^3|\theta_n^\lambda|^2+\lambda\left(\frac{\lambda}{2}+\frac{3}{4\gamma}\right){|H_\gamma(\theta_n^\lambda,X_{n+1})|^2}-\frac{\lambda \gamma}{2}{\la\theta_n^\lambda,H_\gamma(\theta_n^\lambda,X_{n+1})\ra}\\
			&\quad+\gamma^2\sqrt{\frac{\lambda}{2\gamma\beta}}\la\Delta_n,\xi_{n+1}\ra+\frac{\lambda\gamma}{\beta}|\xi_{n+1}|^2+\sqrt{\frac{\lambda\gamma}{2\beta}}\la E_n,\xi_{n+1}\ra.
		\end{split}
	\end{align}
	Let $s\coloneq 1-\lambda\sigma\gamma$. By \eqref{eq.1-2sigma} and \eqref{eq.lower bound of M_n}, we note that $M_n\geq 0$ for all $n\in\mathbb{N}_0$. Recall the notation $\e_n[\cdot]=\e[\cdot\mid\theta_n^\lambda,\nu_n^\lambda]$. Then, by taking conditional expectations on both sides of \eqref{upper bound of M_{n+1}}, using binomial theorem and the fact that $\Delta_n$ and $E_n$ are independent of $\xi_{n+1}$, we obtain that
	\begin{align}\label{eq.even-bounds of M_n}
		\begin{split}
			\e_n[M_{n+1}^q]&\leq(sM_n)^{q}+q(sM_n)^{q-1}\e_n[K_n]+\sum_{k=2}^{q}\binom{q}{k}\e_n\left[(sM_n)^{q-k}K_n^k\right]\\
			&= (sM_n)^q+q(sM_n)^{q-1}\e_n[K_n]+\sum_{k=0}^{q-2}\binom{q}{k+2}\e_n\left[(sM_n)^{q-k-2}K_n^{k+2}\right]\\ 
			&\leq (sM_n)^q+q(sM_n)^{q-1}\e_n[K_n]+\binom{q}{2}\sum_{k=0}^{q-2}\binom{q-2}{k}\e_n\left[|sM_n|^{q-k-2}|K_n|^{k+2}\right] \\
			&\leq (sM_n)^q+q(sM_n)^{q-1}\e_n[K_n] \\
			&\quad+q(q-1)\e_n\left[|K_n|^2\sum_{k=0}^{q-2}\binom{q-2}{k}|sM_n|^{q-k-2}|K_n|^k\right] \\
			&=(sM_n)^q+q(sM_n)^{q-1}\e_n[K_n]+q(q-1)\e_n\left[|K_n|^2 (|sM_n|+|K_n|)^{q-2}\right] \\
			&\leq (sM_n)^q+q(sM_n)^{q-1}\e_n[K_n]\\
			&\quad+q(q-1)2^{q-3}|sM_n|^{q-2}\e_n[|K_n|^2]+q(q-1)2^{q-3}\e_n[|K_n|^{q}].
		\end{split}
	\end{align}
	To establish an upper bound for the second term on the right-hand side of \eqref{eq.even-bounds of M_n}, we use similar argument to \eqref{eq.M_n+1}-\eqref{eq.expectation Lyapunov}, $\gamma\geq\gamma_{\min}\geq 1$ and $\lambda\leq\lambda_{\max,\gamma}\leq 1/(4\gamma)\leq 1/4$, to obtain
	\begin{equation}\label{eq.condition expectation of K_n}
		\e_n[K_n]\leq\lambda\gamma \bar{C}_K,
	\end{equation}
	where $\bar{C}_K\coloneq 7C_H/8+u(0)/2+d/\beta=\mathcal{O}(d/\beta)$. Recall that $\sigma=m/(8\gamma^2)$. By \eqref{eq.K_n}, Lemmas~\ref{Lemma.dissipativity of Expectation of SG} and \ref{lemma.2nd bound of taming factor}, $\gamma\geq 1$ and $\lambda\leq 1/4$, we have
	\begin{align}
		K_n&\leq \frac{\lambda m\gamma}{8}|\theta_n^\lambda|^2+\lambda\left(\frac{\lambda}{2}+\frac{3}{4\gamma}\right)\left(2m^2|\theta_n^\lambda|^2+8\gamma\big[m^2+K_H^2(1+|X_{n+1}|)^{2\rho}\big]\right)\notag\\
		&\quad+\frac{\lambda\gamma}{2}\left(m|\theta_n^\lambda|^2+\widetilde{b}(X_{n+1})\right)+\gamma^2\sqrt{\frac{\lambda}{2\gamma\beta}}\la\Delta_n,\xi_{n+1}\ra+\frac{\lambda\gamma}{\beta}|\xi_{n+1}|^2+\sqrt{\frac{\lambda\gamma}{2\beta}}\la E_n,\xi_{n+1}\ra\notag\\
		&\leq\left(\frac{\lambda m\gamma}{8}+2m^2\lambda\left(\frac{\lambda}{2}+\frac{3}{4\gamma}\right)+\frac{\lambda m\gamma}{2}\right)|\theta_n^\lambda|^2+\gamma^2\sqrt{\frac{\lambda}{2\gamma\beta}}\la\Delta_n,\xi_{n+1}\ra+\sqrt{\frac{\lambda\gamma}{2\beta}}\la E_n,\xi_{n+1}\ra \notag\\
		&\quad+\frac{\lambda\gamma}{\beta}|\xi_{n+1}|^2+8\lambda\gamma\left(\frac{\lambda}{2}+\frac{3}{4\gamma}\right)\left[m^2+K_H^2(1+|X_{n+1}|)^{2\rho}\right]+\frac{\lambda\gamma}{2}\widetilde{b}(X_{n+1})\notag\\
		&\leq \lambda\gamma C_m|\theta_n^\lambda|^2+f_n+g_n,\label{eq.2 upper bound of K_n}
	\end{align}
	where $C_m\coloneq 5m/8+7m^2/4$ and
	\begin{align*}
		f_n &\coloneq \gamma^2\sqrt{\frac{\lambda}{2\gamma\beta}}\la\Delta_n,\xi_{n+1}\ra+\sqrt{\frac{\lambda\gamma}{2\beta}}\la E_n,\xi_{n+1}\ra,\\
		g_n &\coloneq \frac{\lambda\gamma}{\beta}|\xi_{n+1}|^2+\lambda\gamma\left(\frac{1}{8m}+7\right)\left[m^2+K_H^2(1+|X_{n+1}|)^{2\rho}\right].
	\end{align*}
	To establish an upper estimate for the third term on the right-hand side of \eqref{eq.even-bounds of M_n}, we use \eqref{eq.1-2sigma} to obtain that 
	\begin{equation}\label{eq.position, momentum and M_n}
		|\theta_n^\lambda|^2
		\leq \frac{32}{3\gamma^2}M_n,\quad|\nu_n^\lambda|^2\leq \frac{16}{3}M_n.
	\end{equation}
	By using \eqref{eq.Delta n and En}, \eqref{eq.position, momentum and M_n}, Lemma~\ref{lemma.2nd bound of taming factor}, Young's inequality, $\gamma\geq 1$, and $\lambda\leq 1/(4\gamma)$, we obtain
	\begin{align}
		&|f_n|^2\leq \frac{\lambda\gamma^3}{\beta}|\Delta_n|^2|\xi_{n+1}|^2+\frac{\lambda\gamma}{\beta}|E_n|^2|\xi_{n+1}|^2\notag\\
		&\leq \frac{\lambda\gamma^3}{\beta}|\xi_{n+1}|^2\left|\theta_n^\lambda+\gamma^{-1}\nu_n^\lambda-\lambda\gamma^{-1} H_\gamma(\theta_n^\lambda,X_{n+1})\right|^2\notag\\
		&\quad+\frac{\lambda\gamma}{\beta}|\xi_{n+1}|^2\left|(1-\lambda\gamma)\nu_n^\lambda-\lambda H_\gamma(\theta_n^\lambda,X_{n+1})\right|^2\notag\\
		&\leq \frac{3\lambda\gamma^3}{\beta}|\xi_{n+1}|^2\left(|\theta_n^\lambda|^2+\gamma^{-2}|\nu_n^\lambda|^2+\lambda^2\gamma^{-2}\big(2m^2|\theta_n^\lambda|^2+8\gamma[m^2+K_H^2(1+|X_{n+1}|)^{2\rho}]\big)\right)\notag\\
		&\quad+\frac{2\lambda\gamma}{\beta}|\xi_{n+1}|^2\left((1-\lambda\gamma)^2|\nu_n^\lambda|^2+\lambda^2\big(2m^2|\theta_n^\lambda|^2+8\gamma[m^2+K_H^2(1+|X_{n+1}|)^{2\rho}]\big)\right)\notag\\
		&\leq \frac{48\lambda\gamma}{\beta}|\xi_{n+1}|^2 M_n+\frac{64m^2\lambda^3}{\gamma\beta}|\xi_{n+1}|^2M_n+\frac{24\lambda^3\gamma^2}{\beta}|\xi_{n+1}|^2[m^2+K_H^2(1+|X_{n+1}|)^{2\rho}]\notag\\
		&\quad+\frac{32\lambda\gamma}{3\beta}|\xi_{n+1}|^2M_n+\frac{128m^2\lambda^3}{3\gamma\beta}|\xi_{n+1}|^2M_n+\frac{16\lambda^3\gamma^2}{\beta}|\xi_{n+1}|^2 [m^2+K_H^2(1+|X_{n+1}|)^{2\rho}] \notag\\
		&\leq C_f\lambda\gamma\beta^{-1}M_n|\xi_{n+1}|^2+10\lambda\gamma\beta^{-1}|\xi_{n+1}|^2\left(m^2+K_H^2(1+|X_{n+1}|)^{2\rho}\right),\label{eq.upper bound of |f_n|^2}
	\end{align}
	where $C_f\coloneq 59(1+m^2)$. Similarly, we have 
	\begin{align}
		|g_n|^2&\leq \frac{2\lambda^2\gamma^2}{\beta^2}|\xi_{n+1}|^4 + 4\lambda^2\gamma^2\left(\frac{1}{8m}+7\right)^2[m^4+K_H^4(1+|X_{n+1}|)^{4\rho}]\notag\\
		&\leq \frac{2\lambda\gamma}{\beta^2}|\xi_{n+1}|^4 + 4\lambda\gamma\left(\frac{1}{8m}+7\right)^2[m^4+K_H^4(1+|X_{n+1}|)^{4\rho}].\label{eq.upper bound of |g_n|^2}
	\end{align}
	Substituting \eqref{eq.upper bound of |f_n|^2}-\eqref{eq.upper bound of |g_n|^2} back into \eqref{eq.2 upper bound of K_n} and using \eqref{eq.position, momentum and M_n} yields
	\begin{align}
		\e_n[|K_n|^2]&\leq \e_n\left[3\lambda^2\gamma^2 C_m^2|\theta_n^\lambda|^4+3|f_n|^2+3|g_n|^2\right]\notag\\
		&\leq \frac{32^2 C_m^2}{3\gamma^4}\lambda^2\gamma^2|M_n|^2+3C_f\lambda\gamma\beta^{-1} d M_n+c_f\lambda\gamma, \label{eq.2nd bound of K_n}
	\end{align}
	where $c_f\coloneq 6d(d+2)/\beta^2+4 C_H d/\beta+12(1/(8m)+7)^2(m^4+K_H^4\e[(1+|X_0|)^{4\rho}])=\mathcal{O}((d/\beta)^2)$. To establish an upper bound for the last term on the right-hand side of \eqref{eq.even-bounds of M_n}, we first note, for any $q\in[2,\infty)\cap\mathbb{N}$, that \eqref{eq.upper bound of |f_n|^2} implies
	\begin{align}\label{eq.f_n q}
		\begin{split}
			|f_n|^q=(|f_n|^2)^{q/2}&\leq 2^{q/2-1}C_f^{q/2}\lambda^{q/2}\gamma^{q/2}\beta^{-q/2}|M_n|^{q/2}|\xi_{n+1}|^{q}\\
			&\quad+10^{q/2} 2^{q-2}\lambda^{q/2}\gamma^{q/2}\beta^{-q/2}|\xi_{n+1}|^{q}\big(m^{q}+K_H^{q}(1+|X_{n+1}|)^{\rho q}\big).
		\end{split}
	\end{align}
	By taking conditional expectations on both sides of \eqref{eq.f_n q}, by noticing $q/2\geq 1$, and by using the fact that
	\[\e[|\xi_{n+1}|^q]=2^{q/2}\frac{\Gamma\left((d+q)/2\right)}{\Gamma\left(d/2\right)}\leq (d+q)^{q/2},\]
	we obtain
	\begin{align}
		\e_n[|f_n|^{q}]&\leq 2^{q/2-1}C_f^{q/2}\lambda\gamma\beta^{-q/2}(d+q)^{q/2}\,|M_n|^{q/2}\notag\\
		&\quad+10^{q/2} 2^{q-2}\lambda\gamma\beta^{-q/2}(d+q)^{q/2}\big(m^{q}+K_H^{q}\e[(1+|X_0|)^{\rho q}]\big)\notag\\
		&\leq K_f\lambda\gamma(|M_n|^{q/2}+1),\label{eq.estimates of f_n}
	\end{align}
	where $K_f\coloneq 2^{3q-2}((d+q)/\beta)^{q/2}(C_f^{q/2}+m^{q}+K_H^{q}\e[(1+|X_0|)^{\rho q}])=\mathcal{O}((d/\beta)^{q/2})$. Similarly, \eqref{eq.upper bound of |g_n|^2} implies
	\begin{align}\label{eq.g_n q}
		\begin{split}
			|g_n|^q&\leq 2^{q-1}\lambda^{q/2}\gamma^{q/2}\beta^{-q}|\xi_{n+1}|^{2q}\\
			&\quad+2^{2q-2}\lambda^{q/2}\gamma^{q/2}\left(\frac{1}{8m}+7\right)^q\big(m^{2q}+K_H^{2q}(1+|X_{n+1}|)^{2\rho q}\big).
		\end{split}
	\end{align}
	Taking conditional expectations on both sides of \eqref{eq.g_n q} yields
	\begin{align}\label{eq.estimates of g_n}
		\e_n[|g_n|^q]\leq K_g\lambda\gamma,
	\end{align}
	where $K_g\coloneq 2^{q-1}((d+2q)/\beta)^q+2^{2q-2}(1/(8m)+7)^q(m^{2q}+K_H^{2q}\e[(1+|X_0|)^{2\rho q}])=\mathcal{O}((d/\beta)^q)$. Combining \eqref{eq.2 upper bound of K_n}, \eqref{eq.estimates of f_n}, and \eqref{eq.estimates of g_n}, and using \eqref{eq.position, momentum and M_n}, yields
	\begin{align}
		\e_n[|K_n|^q]&\leq 3^{q-1}\lambda^q\gamma^q C_m^q|\theta_n^\lambda|^{2q}+3^{q-1}\e_n[|f_n|^q]+3^{q-1}\e_n[|g_n|^q]\notag\\
		&\leq 3^{q-1}C_m^q\frac{32^q\lambda^q\gamma^q}{3^q\gamma^{2q}}|M_n|^q+3^{q-1}K_f\lambda\gamma(|M_n|^{q/2}+1)+3^{q-1} K_g \lambda\gamma \notag\\
		&\leq\frac{32^q C_m^q}{3\gamma^{2q}}\lambda^q\gamma^q|M_n|^q+3^{q-1}K_f\lambda\gamma|M_n|^{q/2}+C_K\lambda\gamma,\label{eq.even-bounds of K_n}
	\end{align}
	where $C_K\coloneq 3^{q-1}(K_f+K_g)=\mathcal{O}((d/\beta)^q)$. By noticing $0<s=1-\lambda\sigma\gamma=1-\lambda m/(8\gamma)<1$, and by substituting \eqref{eq.condition expectation of K_n}, \eqref{eq.2nd bound of K_n} and \eqref{eq.even-bounds of K_n} into \eqref{eq.even-bounds of M_n}, we obtain
	\begin{align}\label{eq.even-bound of M_n+1}
		\begin{split}
			\e_n[M_{n+1}^q]&\leq (sM_n)^q+q(sM_n)^{q-1}\lambda\gamma \bar{C}_K\\
			&\quad+q(q-1)2^{q-3}|sM_n|^{q-2}\left(\frac{32^2 \lambda^2 C_m^2}{3\gamma^2}|M_n|^2+3C_f\lambda\gamma dM_n/\beta+c_f\lambda\gamma\right)\\
			&\quad+q(q-1)2^{q-3}\left(\frac{32^{q}\lambda^{q}C_m^{q}}{3\gamma^{q}}|M_n|^q+3^{q-1}K_f\lambda\gamma|M_n|^{q/2}+C_K\lambda\gamma\right)\\
			&\leq \left(1-\frac{\lambda m}{8\gamma}+q(q-1)2^{q-3}\frac{32^2\lambda^2C_m^2}{3\gamma^2}+q(q-1)2^{q-3}\frac{32^{q}\lambda^{q}C_m^{q}}{3\gamma^{q}}\right)|M_n|^q\\
			&\quad+\left(q \bar{C}_K+3q(q-1)2^{q-3}C_f \beta^{-1}d\right)\lambda\gamma|M_n|^{q-1}\\
			&\quad+q(q-1)2^{q-3}c_f\lambda\gamma|M_n|^{q-2}+q(q-1)2^{q-3}3^{q-1}K_f\lambda\gamma|M_n|^{q/2}\\
			&\quad+q(q-1)2^{q-3}C_K\lambda\gamma.
		\end{split}
	\end{align}
	By using $\gamma\geq\gamma_{\min}\geq 32 C_m$ and $\lambda\leq\lambda_{\max,\gamma}\leq \min\{1, 3m/(16q(q-1)2^{q-2}\gamma)\}$, we have that
	\begin{align}\label{eq.Mq lambda}
		\begin{split}
			&1-\frac{\lambda m}{8\gamma}+q(q-1)2^{q-3}\frac{32^2\lambda^2C_m^2}{3\gamma^2}+q(q-1)2^{q-3}\frac{32^{q}\lambda^{q}C_m^{q}}{3\gamma^{q}}\\
			&\leq 1-\frac{\lambda m}{8\gamma}+\frac{q(q-1)2^{q-2}\lambda^2}{3}\leq 1-\frac{\lambda m}{16\gamma}
		\end{split}
	\end{align}
	Substituting \eqref{eq.Mq lambda} into \eqref{eq.even-bound of M_n+1} yields
	\begin{align}
		\e_n[M_{n+1}^q]&\leq\left(1-\frac{\lambda m}{16\gamma}\right)|M_n|^q+\left(q\bar{C}_K+3q(q-1)2^{q-3}C_f \beta^{-1}d\right)\lambda\gamma|M_n|^{q-1}\notag\\
		&\quad+q(q-1)2^{q-3}c_f\lambda\gamma|M_n|^{q-2}+q(q-1)2^{q-3}3^{q-1}K_f\lambda\gamma|M_n|^{q/2}\notag\\
		&\quad+q(q-1)2^{q-3}C_K\lambda\gamma\notag\\
		&\leq \left(1-\frac{\lambda m}{64\gamma}\right)|M_n|^q+q(q-1)2^{q-3}C_K\lambda\gamma+\sum_{i=1}^3 T_i(|M_n|),\label{eq.upper bound of Lyapunov-2q}
	\end{align}
	where for any $s\in\mathbb{R}^d$,
	\begin{align*}
		T_1(s)&\coloneq \left(q \bar{C}_K+3q(q-1)2^{q-3}C_f\beta^{-1}d\right)\lambda\gamma|s|^{q-1}-\frac{\lambda m}{64\gamma}|s|^{q},\\
		T_2(s)&\coloneq q(q-1)2^{q-3}c_f\lambda\gamma|s|^{q-2}-\frac{\lambda m}{64\gamma}|s|^q,\\
		T_3(s)&\coloneq q(q-1)2^{q-3}3^{q-1}K_f\lambda\gamma|s|^{q/2}-\frac{\lambda m}{64\gamma}|s|^q.
	\end{align*}
	Define
	\begin{align*}
		\mathsf{M}_{q,1}&\coloneq \frac{64\gamma^2\left(q \bar{C}_K+3q(q-1)2^{q-3}C_f\beta^{-1}d\right)}{m},\\
		\mathsf{M}_{q,2}&\coloneq 8\gamma\left(\frac{q(q-1)2^{q-3}c_f}{m}\right)^{1/2},\\
		\mathsf{M}_{q,3}&\coloneq \gamma^{4/q}\left(\frac{64q(q-1)2^{q-3}3^{q-1}K_f}{m}\right)^{2/q}.
	\end{align*}
	Then, $|s|\geq \mathsf{M}_{q,i}$ implies $T_i(s)\leq 0$, for all $i=1,2,3$. Thus, we have
	\[T_i(s)=T_i(s)\II_{\{s\leq \mathsf{M}_{q,i}\}}+T_i(s)\II_{\{s\geq \mathsf{M}_{q,i}\}}\leq T_i(s)\II_{\{s\leq \mathsf{M}_{q,i}\}}.\]
	By using the above inequality, we obtain, for all $s\in\mathbb{R}$, that
	\begin{align}
		T_1(s)&\leq \left(q \bar{C}_K+3q(q-1)2^{q-3}C_f\beta^{-1}d\right)\lambda\gamma|s|^{q-1}\II_{\{s\leq \mathsf{M}_{q,1}\}}\notag\\
		&\leq \left(q \bar{C}_K+3q(q-1)2^{q-3}C_f\beta^{-1}d\right)\lambda\gamma \mathsf{M}_{q,1}^{q-1}\notag\\
		&=\left(\frac{64}{m}\right)^{q-1}\left(q \bar{C}_K+3q(q-1)2^{q-3}C_f\beta^{-1}d\right)^{q}\lambda\gamma^{2q-1}\notag\\
		&=N_1\lambda\gamma^{2q-1},\label{eq.T1}
	\end{align}
	where $N_1\coloneq(64/m)^{q-1}(q \bar{C}_K+3q(q-1)2^{q-3}C_f\beta^{-1}d)^{q}=\mathcal{O}((d/\beta)^{q})$. Moreover, by $\gamma\geq\gamma_{\min}\geq 1$, it holds that
	\begin{align}
		T_2(s)&\leq q(q-1)2^{q-3}c_f\lambda\gamma|s|^{q-2}\II_{\{s\leq \mathsf{M}_{q,2}\}}\notag\\
		&\leq q(q-1)2^{q-3}c_f\lambda\gamma \mathsf{M}_{q,2}^{q-2}\notag\\
		&= \left(\frac{8}{\sqrt{m}}\right)^{q-2}\left(q(q-1)2^{q-3}c_f\right)^{q/2}\lambda\gamma^{q-1}\notag\\
		&\leq N_2\lambda\gamma^{2q-1},\label{eq.T2}
	\end{align}
	where $N_2\coloneq (8/\sqrt{m})^{q-2}(q(q-1)2^{q-3}c_f)^{q/2}=\mathcal{O}((d/\beta)^q)$. Similarly, by $\gamma\geq\gamma_{\min}\geq 1$ and $q\geq 2$, we have that
	\begin{align}
		T_3(s)&\leq q(q-1)2^{q-3}3^{q-1}K_f\lambda\gamma|s|^{q/2}\II_{\{s\leq \mathsf{M}_{q,3}\}}\notag\\
		&\leq q(q-1)2^{q-3}3^{q-1}K_f\lambda\gamma \mathsf{M}_{q,3}^{q/2}\notag\\
		&=\frac{64}{m}\left(q(q-1)2^{q-3}3^{q-1}K_f\right)^2\lambda\gamma^3\notag\\
		&\leq N_3\lambda\gamma^{2q-1},\label{eq.T3}
	\end{align}
	where $N_3\coloneq 64(q(q-1)2^{q-3}3^{q-1}K_f)^2/m=\mathcal{O}((d/\beta)^q)$.
	Substituting \eqref{eq.T1}-\eqref{eq.T3} into \eqref{eq.upper bound of Lyapunov-2q} yields
	\begin{align}\label{eq.upper bound Mn1}
		\e_n[|M_{n+1}|^q]\leq\left(1-\frac{\lambda m}{64\gamma}\right)|M_n|^q+N_0\lambda\gamma^{2q-1},
	\end{align}
	where $N_0\coloneq N_1+N_2+N_3+q(q-1)2^{q-3}C_K=\mathcal{O}((d/\beta)^q)$. Iterating \eqref{eq.upper bound Mn1} yields
	\[\sup_{n\in\mathbb{N}_0}\e[|M_{n+1}|^q]\leq \e[|M_0|^q]+\frac{64N_0}{m}\gamma^{2q}.\]
	By \eqref{eq.position, momentum and M_n}, \eqref{eq.upper bound Mn1}, and $\gamma\geq 1$, we obtain
	\begin{align*}
		\sup_{n\in\mathbb{N}_0}\e[|\theta_n^\lambda|^{2q}]&\leq \frac{32^q}{3^q\gamma^{2q}}\left(\e[|M_0|^q]+\frac{64N_0}{m}\gamma^{2q}\right)\\
		&\leq \frac{8^q 4^q}{3^q\gamma^{2q}}\e[|M_0|^q]+\frac{8^{q+2}4^q N_0}{3^qm}\\
		&\leq \frac{8^q 2^q}{\gamma^{2q}}\left(2^{q-2}\gamma^{2q}\e[|\theta_0|^{2q}]+2^{q-1}\e[|\nu_0|^{2q}]\right)+\frac{8^{q+2}2^qN_0}{m}\\
		&\leq 2^{5q-2}\e[|\theta_0|^{2q}]+2^{5q-1}\e[|\nu_0|^{2q}]+2^{4q+6}N_0/m\eqcolon\bar{C}_{2q},
	\end{align*}
	where $\bar{C}_{2q}=\mathcal{O}((d/\beta)^q)$.
\end{proof}

\phantomsection
\begin{proof}[\textbf{Proof of Lemma~\ref{lemma.mse of interpolation}}]
	\label{Proof.MSE of interpolation}
	For any $t\geq 0$, by \eqref{Continuous-time interpolation}, we have that
	\[V_t^\lambda=V_{\bt}^\lambda-\lambda\gamma\int_{\bt}^{t}V_{\bs}^\lambda\,\dd s-\lambda\int_{\bt}^t H_\gamma(Z_{\bs}^\lambda,X_{\us})\,\dd s+\sqrt{2\lambda\gamma\beta^{-1}}(W_t^\lambda-W_{\bt}^\lambda).\]
	By using the inequality that $|a+b+c|^2\leq 3(|a|^2+|b|^2+|c|^2)$ for $a,b,c>0$ and Lemma~\ref{lemma.2nd bound of taming factor}, we have that
	\begin{align}
		\e[|V_t^\lambda-V_{\bt}^\lambda|^2]&\leq 3\lambda^2\gamma^2\e\left[\bigg|\int_{\bt}^t V_{\bs}^\lambda\,\dd s\bigg|^2\right]+3\lambda^2\e\left[\bigg|\int_{\bt}^t H_\gamma(Z_{\bs}^\lambda,X_{\us})\,\dd s\bigg|^2\right]\notag\\
		&\quad+\frac{6\lambda\gamma}{\beta}\e[|W_t^\lambda-W_{\bt}^\lambda|^2]\notag\\
		&\leq 3\lambda^2\gamma^2\e[|V_{\bs}^\lambda|^2]+3\lambda^2\left(2m^2\e[|Z_{\bs}^\lambda|^2]+C_H\gamma\right)+\frac{6\lambda\gamma d}{\beta}.\label{eq.lemma.mse of interpolation}
	\end{align}
	Since the law of the continuous-time interpolation process \eqref{Continuous-time interpolation} agrees with tSGHMC \eqref{tSGHMC}-\eqref{eq.tamedSG} at grid points, applying Lemma~\ref{lemma.2 moment of algorithm} yields
	\[\e[|V_{\bs}^\lambda|^2]\leq \bar{B}_2,\quad \e[|Z_{\bs}^\lambda|^2]\leq \bar{C}_2,\]
	where $\bar{B}_2=\mathcal{O}(d/\beta)$ and $\bar{C}_2=\mathcal{O}(d/\beta)$ are given explicitly in Table~\ref{tab.constants}. Hence, we obtain, by noticing $\gamma\geq\gamma_{\min}\geq 1$ and $\lambda\leq\lambda_{\max,\gamma}\leq 1/(4\gamma)\leq 1/4$, that
	\begin{align*}
		\e[|V_t^\lambda-V_{\bt}^\lambda|^2]\leq \lambda\gamma C_{1,v},
	\end{align*}
	where $C_{1,v}\coloneq 3\bar{B}_2/4+3m^2\bar{C}_2/2+3C_H/4+6d/\beta=\mathcal{O}(d/\beta)$. Moreover, by similar argument, we obtain that
	\[\e[|Z_t^\lambda-Z_{\bt}^\lambda|^2]\leq \lambda\bar{B}_2,\]
	which completes the proof.
\end{proof}

\phantomsection
\begin{proof}[\textbf{Proof of Lemma~\ref{lemma.2nd bound MY kinetic}}]
	\label{proof.2nd bound MY kinetic}
	For all $(\theta,\nu)\in\mathbb{R}^{2d}$, consider the following Lyapunov function 
	\begin{equation}\label{eq.bar V}
		\bar{\mathcal{V}}(\theta,\nu)\coloneq\beta u_{MY,\epsilon}(\theta)+\frac{\beta}{4}\gamma^2\left(|\theta+\gamma^{-1}\nu|^2+|\gamma^{-1}\nu|^2-\sigma|\theta|^2\right),
	\end{equation}
	where $\sigma=m/(8\gamma^2)$ as before. Moreover, by Lemma~\ref{lemma.hMY leq h} and Remark~\ref{remark.MY}, we have
	\begin{equation}\label{eq.Lip MY h}
		|h_{MY,\epsilon}(\theta)|\leq K|\theta|+|h_{MY,\epsilon}(0)|\leq K|\theta|+|h(0)|.
	\end{equation}
	By the fundamental theorem of calculus, Cauchy-Schwarz inequality, and \eqref{eq.Lip MY h}, we obtain 
	\begin{align}
		u_{MY,\epsilon}(\theta)-u_{MY,\epsilon}(0)&=\int_0^1 \la\theta,h_{MY,\epsilon}(t\theta)\ra\,\dd t\leq \int_0^1|h_{MY,\epsilon}(t\theta)|\,|\theta|\,\dd t\notag\\
		&\leq \int_0^1\left(Kt|\theta|+|h(0)|\right)|\theta|\,\dd t=\frac{K}{2}|\theta|^2+|h(0)|\,|\theta|.\label{eq.upper bound uMY}
	\end{align}
	Since $\gamma\geq\gamma_{\min}\geq \sqrt{2(K+2|h(0)|)/3}$, we have
	\[\frac{m}{4}\geq 2\sigma\left(\frac{K}{2}+|h(0)|+\frac{\gamma^2}{4}\right).\]
	Therefore, by Remark~\ref{remark.MY}, \eqref{eq.upper bound uMY}, the inequality $|\theta|\leq 1+|\theta|^2$ for all $\theta\in\mathbb{R}^d$, and $\gamma\geq\gamma_{\min}\geq \sqrt{m}$, we obtain that
	\begin{align}
		\la\theta,h_{MY,\epsilon}(\theta)\ra&\geq \frac{m}{4}|\theta|^2-u(0)\notag\\
		&\geq 2\sigma\left(\frac{K}{2}+|h(0)|+\frac{\gamma^2}{4}\right)|\theta|^2-u(0)\notag\\
		&\geq 2\sigma\left(\frac{K}{2}|\theta|^2+|h(0)|\,|\theta|-|h(0)|\,|\theta|+|h(0)|\,|\theta|^2+\frac{\gamma^2}{4}|\theta|^2\right)-u(0)\notag\\
		&\geq 2\sigma\left(u_{MY,\epsilon}(\theta)-u_{MY,\epsilon}(0)-|h(0)|(1+|\theta|^2)+|h(0)|\,|\theta|^2+\frac{\gamma^2}{4}|\theta|^2\right)-u(0)\notag\\
		&\geq 2\sigma\left(u_{MY,\epsilon}(\theta)+\frac{\gamma^2}{4}|\theta|^2\right)-2\sigma \left(u_{MY,\epsilon}(0)+|h(0)|\right)-u(0)\notag\\
		&\geq 2\sigma\left(u_{MY,\epsilon}(\theta)+\frac{\gamma^2}{4}|\theta|^2\right)-\frac{1}{4}\left(u(0)+|h(0)|\right)-u(0)\notag\\
		&\geq 2\sigma\left(u_{MY,\epsilon}(\theta)+\frac{\gamma^2}{4}|\theta|^2\right)-\frac{2A_c}{\beta},\label{eq.u h dissipativity}
	\end{align}
	where $A_c\coloneq (5u(0)+|h(0)|)\beta/8=\mathcal{O}(\beta)$. 
	
	Next, we follow the argument in \cite[Lemma~5.2]{liang2024non} to establish an upper bound for $\e[|r_t|^2]$. By Remark~\ref{remark.MY} and \eqref{eq.u h dissipativity}, we note that $u_{MY,\epsilon}$ is gradient Lipschitz and satisfies a dissipativity condition. Moreover, since $\gamma\geq\gamma_{\min}\geq\sqrt{m}$, $\sigma=m/(8\gamma^2)\leq 1/4$. These properties fulfill the assumptions of \cite[Lemma~5.2]{liang2024non}, see also Remarks~2.2-2.4 therein. Therefore, we can apply \cite[Lemma~5.2]{liang2024non} with $\lambda\curvearrowleft\sigma$, $\mathcal{V}\curvearrowleft\bar{\mathcal{V}}$, $\theta_t\curvearrowleft r_t$, $V_0\curvearrowleft \nu_0$ to obtain
	\[\frac{\beta}{8}(1-2\sigma)\gamma^2\e[|r_t|^2]\leq\e[\bar{\mathcal{V}}(\theta_0,\nu_0)]+\frac{d+A_c}{\sigma}.\]
	Using $\gamma\geq\gamma_{\min}\geq \max\{1,\sqrt{m}\}$, \eqref{eq.M_0}, \eqref{eq.bar V} and Remark~\ref{remark.MY}, we have
	\begin{align}
		\e[|r_t|^2]&\leq \frac{\e[\bar{\mathcal{V}}(\theta_0,\nu_0)]+(d+A_c)/\sigma}{\beta(1-2\sigma)\gamma^2/8}\notag\\
		&\leq \frac{\beta\e[u_{MY,\epsilon}(\theta_0)]+\beta\e[M_0]+(d+A_c)/\sigma}{\beta(1-2\sigma)\gamma^2/8}\notag\\
		&\leq \frac{8\e[u(\theta_0)]}{(1-2\sigma)\gamma^2}+\frac{8\e[M_0]}{(1-2\sigma)\gamma^2}+\frac{8(d+A_c)}{\beta\sigma(1-2\sigma)\gamma^2}.\label{eq.2nd moment of Zt}
	\end{align}
	Moreover, by \eqref{eq.1-2sigma} and $\gamma\geq 1$, we have
	\begin{equation}\label{eq.Zt1}
		\frac{8\e[u(\theta_0)]}{(1-2\sigma)\gamma^2}\leq 11\e[u(\theta_0)].
	\end{equation}
	Besides, by \eqref{eq.M_0}, we obtain that
	\[M_0\leq\frac{\gamma^2}{4}\left(2|\theta_0|^2+2\gamma^{-2}|\nu_0|^2\right)+\frac{1}{4}|\nu_0|^2\leq \frac{\gamma^2}{2}|\theta_0|^2+\frac{3}{4}|\nu_0|^2,\]
	which, by \eqref{eq.1-2sigma} and $\gamma\geq 1$, implies
	\begin{equation}\label{eq.Zt2}
		\frac{8\e[M_0]}{(1-2\sigma)\gamma^2}\leq 6\e[|\theta_0|^2]+8\e[|\nu_0|^2].
	\end{equation}
	Finally, by $\sigma=m/(8\gamma^2)$ and \eqref{eq.1-2sigma}, we have
	\begin{equation}\label{eq.Zt3}
		\frac{8(d+A_c)}{\beta\sigma(1-2\sigma)\gamma^2}=\frac{64(d+A_c)}{\beta m(1-2\sigma)}\leq \frac{86(d+A_c)}{\beta m}.
	\end{equation}
	Substituting \eqref{eq.Zt1}-\eqref{eq.Zt3} into \eqref{eq.2nd moment of Zt} yields
	\begin{equation}\label{eq.2nd Zt}
		\sup_{t\geq 0}\e[|r_t|^2]\leq C_r,
	\end{equation}
	where $C_r\coloneq 11\e[u(\theta_0)]+6\e[|\theta_0|^2]+8\e[|\nu_0|^2]+86(d+A_c)/(\beta m)=\mathcal{O}(d/\beta)$.
	
	To establish an upper bound for $\e[|R_t|^2]$, we first set
	$a\land b\coloneq\min\{a,b\}$ for $a,b\in\mathbb{R}$ and define the stopping time
	\[
	\tau_m\coloneq\inf\{t\geq 0: |R_t|\geq m\}.
	\]
	Applying Itô's formula to the stopped process $e^{\gamma (t\land\tau_m)}|R_{t\land\tau_m}|^2$ yields
	\begin{align}\label{eq.stopped process}
		\begin{split}
			e^{\gamma(t\land\tau_m)}|R_{t\land\tau_m}|^2-|R_0|^2&=\int_0^{t\land\tau_m}e^{\gamma s}\left(-\gamma |R_s|^2-2\la R_s,h_{MY,\epsilon}(r_s)\ra+2\gamma\beta^{-1}d\right)\,\dd s\\
			&\quad+2\sqrt{2\gamma\beta^{-1}}\int_0^{t\land\tau_m}e^{\gamma s}\la R_s,\mathrm dW_s\ra .
		\end{split}
	\end{align}
	By taking expectations and by using Young's inequality together with \eqref{eq.Lip MY h} and \eqref{eq.2nd Zt}, we obtain
	\begin{align}
		&\e[e^{\gamma(t\land\tau_m)}|R_{t\land\tau_m}|^2]-\e[|R_0|^2]\notag\\
		&\leq\e\left[\int_0^{t\land\tau_m}e^{\gamma s}\left(-\gamma|R_s|^2+\gamma|R_s|^2+\gamma^{-1}|h_{MY,\epsilon}(r_s)|^2+\frac{2\gamma d}{\beta}\right)\,\dd s\right]\notag\\
		&\leq \e\left[\int_0^{t\land\tau_m}e^{\gamma s}\left(\frac{2|h(0)|^2+2K^2|r_s|^2}{\gamma}+\frac{2\gamma d}{\beta}\right)\,\dd s\right]\notag\\
		&\leq \int_0^t e^{\gamma s}\left(\frac{2|h(0)|^2+2K^2\e[|r_s|^2]}{\gamma}+\frac{2\gamma d}{\beta}\right)\,\dd s\notag\\
		&\leq \int_0^t e^{\gamma s}\left(\frac{2|h(0)|^2+2K^2 C_r}{\gamma}+\frac{2\gamma d}{\beta}\right)\,\dd s\notag\\
		&=\left(\frac{2|h(0)|^2+2K^2 C_r}{\gamma^2}+\frac{2 d}{\beta}\right)(e^{\gamma t}-1),\label{eq.stop-inequality}
	\end{align}
	where the third inequality is satisfied since $t\land\tau_m\leq t$.
	By \cite[Theorem 8.3]{le2016brownian} and Remark~\ref{remark.MY}, $R_t$ is a strong solution of \eqref{eq.MY kinetic Langevin} with continuous sample paths. Therefore, we have $\tau_m\uparrow\infty$ almost surely. As a result, $t\land\tau_m\uparrow t$ and $e^{\gamma(t\land\tau_m)}R_{t\land \tau_m}\to e^{\gamma t}R_t$ almost surely as $m\to\infty$. By Fatou's lemma, we have
	\begin{align}\label{eq.Fatou}
		\e[e^{\gamma t}|R_t|^2]\leq \liminf_{n\to\infty}\e[e^{\gamma(t\land\tau_m)}|R_{t\land\tau_m}|^2].
	\end{align}
	Combining \eqref{eq.stop-inequality} and \eqref{eq.Fatou}, and letting $m\to\infty$, we obtain
	\begin{align*}
		e^{\gamma t}\e[|R_t|^2]\leq\e[|R_0|^2]+\left(\frac{2|h(0)|^2+2K^2 C_r}{\gamma^2}+\frac{2 d}{\beta}\right)(e^{\gamma t}-1),
	\end{align*}
	which implies that
	\begin{align*}
		\e[|R_t|^2]\leq e^{-\gamma t}\e[|R_0|^2]+\left(\frac{2|h(0)|^2+2K^2 C_r}{\gamma^2}+\frac{2 d}{\beta}\right)(1-e^{-\gamma t}).
	\end{align*}
	Since $\gamma\geq\gamma_{\min}\geq \max\{1,K\}$ and $R_0=\nu_0$, we have that
	\begin{align*}
		\sup_{t\geq 0}\e[|R_t|^2]\leq C_R,
	\end{align*}
	where $C_R\coloneq \e[|\nu_0|^2]+2|h(0)|^2+2C_r+2d/\beta=\mathcal{O}(d/\beta)$.
\end{proof}

\phantomsection
\begin{proof}[\textbf{Proof of Lemma~\ref{lemma.2nd bound of auxiliary process}}]
	\label{proof.2nd bound of auxiliary process}
	Recall the process $(\zeta_t^{\lambda,n},V_t^{\lambda,n})_{t\geq nT}$ defined in \eqref{Auxiliary process with certain initial}. According to Lemma~\ref{lemma.2nd bound MY kinetic}, we have, for all $n\in\mathbb{N}_0$, that
	\begin{align*}
		\sup_{t\geq nT}\e[|\zeta_t^{\lambda,n}|^2]\leq \frac{\e[\bar{\mathcal{V}}(Z_{nT}^\lambda,V_{nT}^\lambda)]+(d+A_c)/\sigma}{\beta(1-2\sigma)\gamma^2/8}.
	\end{align*}
	By Remark~\ref{remark.MY} and \eqref{eq.u h dissipativity}, we note that $u_{MY,\epsilon}$ is gradient Lipschitz and satisfies a dissipativity condition. Moreover, since $\gamma\geq\gamma_{\min}\geq\sqrt{m}$, $\sigma=m/(8\gamma^2)\leq 1/4$. These properties fulfill the assumptions of \cite[Lemma~5.4]{liang2024non}, see also Remarks~2.2-2.4 therein.
	Thus, we apply \cite[Lemma~5.4]{liang2024non} with $\bar\zeta_t^{\eta,n}\curvearrowleft\zeta_t^{\lambda,n}$, $\bar{Z}_t^{\eta,n}\curvearrowleft V_t^{\lambda,n}$  $\bar\theta_{nT}^\eta\curvearrowleft Z_{nT}^\lambda$, $\bar V_{nT}^\eta\curvearrowleft V_{nT}^\lambda$, $\eta\curvearrowleft\lambda$, $\lambda\curvearrowleft\sigma$ and $\mathcal{V}\curvearrowleft\bar{\mathcal{V}}$ to obtain
	\begin{align*}
		\sup_{t\geq nT}\e[|\zeta_t^{\lambda,n}|^2]&\leq \frac{\e[\bar{\mathcal{V}}(\theta_0,\nu_0)]+5(d+A_c)/\beta}{\beta(1-2\sigma)\gamma^2/8}\\
		&\leq 11\e[u(\theta_0)]+6\e[|\theta_0|^2]+8\e[|\nu_0|^2]+\frac{430(d+A_c)}{\beta m},
	\end{align*}
	where the last inequality follows using the same argument from \eqref{eq.2nd moment of Zt} to \eqref{eq.2nd Zt} in the proof of Lemma~\ref{lemma.2nd bound MY kinetic}. Thus we have
	\begin{equation}\label{eq.bound zeta}
		\sup_{n\in\mathbb{N}_0}\sup_{t\geq nT}\e[|\zeta_t^{\lambda,n}|^2]\leq C_\zeta^{\#},
	\end{equation}
	where $C_\zeta^\#\coloneq 11\e[u(\theta_0)]+6\e[|\theta_0|^2]+8\e[|\nu_0|^2]+430(d+A_c)/(\beta m)=\mathcal{O}(d/\beta)$.
	
	To establish an upper bound for $\e[|V_t^{\lambda,n}|^2]$, fix $n\in\mathbb{N}_0$ and $t\geq nT$. For each $m\in\mathbb{N}$, define the stopping time
	\[\rho_m\coloneq\inf\{t\geq nT:|V_t^{\lambda,n}|\geq m\}.\]
	Applying Itô's formula to the stopped process $e^{\lambda\gamma(t\land\rho_m-nT)}|V_{t\land\rho_m}^{\lambda,n}|^2$, we obtain
	\begin{align}\label{eq.ito V}
		\begin{split}
			&e^{\lambda\gamma((t\land\rho_m)-nT)}|V_{t\land\rho_m}^{\lambda,n}|^2-|V_{nT}^{\lambda,n}|^2\\
			&=\int_{nT}^{t\land\rho_m}e^{\lambda\gamma(s-nT)}
			\left(-\lambda\gamma |V_s^{\lambda,n}|^2-2\lambda\la V_s^{\lambda,n},h_{MY,\epsilon}(\zeta_s^{\lambda,n})\ra+\frac{2\lambda\gamma d}{\beta}\right)\,\dd s\\
			&\quad+2\sqrt{2\lambda\gamma\beta^{-1}}\int_{nT}^{t\land\rho_m}e^{\lambda\gamma(s-nT)}\la V_s^{\lambda,n},\dd W_s^\lambda\ra.
		\end{split}
	\end{align}
	Then, it holds that
	\begin{equation}\label{eq.martingale2}
		2\sqrt{2\lambda\gamma\beta^{-1}}\e\left[\int_{nT}^{t\land\rho_m}e^{\lambda\gamma(t-nT)}\la V_s^{\lambda,n},\dd W_s^\lambda\ra\right]=0.
	\end{equation}
	Taking expectation in \eqref{eq.ito V}, using \eqref{eq.Lip MY h}, \eqref{eq.bound zeta} and \eqref{eq.martingale2}, and applying Young's inequality, we obtain
	\begin{align}
		&\e[e^{\lambda\gamma((t\land\rho_m)-nT)}|V_{t\land\rho_m}^{\lambda,n}|^2]-\e[|V_{nT}^{\lambda,n}|^2]\notag\\
		&\leq \e\left[\int_{nT}^{t\land\rho_m}e^{\lambda\gamma(s-nT)}
		\left(-\lambda\gamma |V_s^{\lambda,n}|^2+\lambda\gamma |V_s^{\lambda,n}|^2+\lambda\gamma^{-1}|h_{MY,\epsilon}(\zeta_s^{\lambda,n})|^2+\frac{2\lambda\gamma d}{\beta}\right)\,\dd s\right]\notag\\
		&\leq \e\left[\int_{nT}^{t\land\rho_m}e^{\lambda\gamma(s-nT)}
		\left(\frac{2\lambda(|h(0)|^2+K^2|\zeta_s^{\lambda,n}|^2)}{\gamma}+\frac{2\lambda\gamma d}{\beta}\right)\,\dd s\right]\notag\\
		&\leq \int_{nT}^{t}e^{\lambda\gamma(s-nT)}
		\left(\frac{2\lambda(|h(0)|^2+K^2\mathbb E[|\zeta_s^{\lambda,n}|^2])}{\gamma}+\frac{2\lambda\gamma d}{\beta}\right)\,\dd s\notag\\
		&=\left(\frac{2|h(0)|^2+2K^2C_\zeta^\#}{\gamma^2}+\frac{2d}{\beta}\right)(e^{\lambda\gamma(t-nT)}-1).\label{eq.stop-inequality1}
	\end{align}
	where the third inequality is satisfied since $t\land\rho_m\leq t$. By \cite[Theorem 8.3]{le2016brownian} and Remark~\ref{remark.MY}, $V_t^{\lambda,n}$ is a strong solution of \eqref{Auxiliary process with certain initial} with continuous sample paths. Thus, we have $\rho_m\uparrow\infty$ almost surely. Hence, $t\land\rho_m\uparrow t$ and $e^{\lambda\gamma((t\land\rho_m)-nT)}|V_{t\land\rho_m}^{\lambda,n}|^2 \to e^{\lambda\gamma(t-nT)} |V_t^{\lambda,n}|^2$ almost surely as $m\to\infty$. Therefore, by Fatou's lemma,
	\[\e\left[e^{\lambda\gamma(t-nT)}|V_t^{\lambda,n}|^2\right]\leq\liminf_{m\to\infty}\e\left[e^{\lambda\gamma((t\land\rho_m)-nT)}|V_{t\land\rho_m}^{\lambda,n}|^2\right].\]
	Letting $m\to\infty$ in \eqref{eq.stop-inequality1}, we obtain
	\[e^{\lambda\gamma(t-nT)}\e[|V_t^{\lambda,n}|^2]\le\e[|V_{nT}^{\lambda,n}|^2]+\left(\frac{2|h(0)|^2+2K^2C_\zeta^\#}{\gamma^2}+\frac{2d}{\beta}\right)(e^{\lambda\gamma(t-nT)}-1).\]
	By \eqref{Continuous-time interpolation} and \eqref{Auxiliary process with certain initial}, we have $V_{nT}^{\lambda,n}=V_{nT}^\lambda=\nu_{nT}^\lambda$. Combining this and the above inequality yields
	\[\e[|V_t^{\lambda,n}|^2]\leq e^{-\lambda\gamma(t-nT)}\e[|\nu_{nT}^\lambda|^2]+\left(\frac{2|h(0)|^2+2K^2C_\zeta^\#}{\gamma^2}+\frac{2d}{\beta}\right)(1-e^{-\lambda\gamma(t-nT)}).\]
	Using Lemma~\ref{lemma.2 moment of algorithm} and $\gamma\ge\gamma_{\min}\geq\max\{1,K\}$, we conclude that
	\[\e[|V_t^{\lambda,n}|^2]\leq \bar B_2+2|h(0)|^2+2C_\zeta^\#+\frac{2d}{\beta}.\]
	Thus we have
	\[\sup_{n\in\mathbb{N}_0}\sup_{t\geq nT}\e[|V_t^{\lambda,n}|^2]\leq C_V^\#,\]
	where $C_V^\#\coloneq \bar{B}_2+2|h(0)|^2+2C_\zeta^\#+2d/\beta=\mathcal{O}(d/\beta)$.
\end{proof}

\subsection{Proofs of Main Theorems in Section~\ref{Proof of the Main Theorems}}
\label{Proofs of Convergence Results}

\begin{lemma}\label{lemma.mse - H_lambda and H}
	Let Assumptions~\ref{assumption2.local Lip+growth} and \ref{assumption3.strong convexity} hold. Then, we have, for all $\theta\in\mathbb{R}^d$, $x\in\mathbb{R}^k$, and $\gamma\geq\gamma_{\min}$, that
	\[|H_\gamma(\theta,x)-H(\theta,x)|^2\leq  2\gamma^{-2}\big[4K_H^2(1+|x|)^{2\rho}+m^2\big](1+|\theta|^{12r}),\]
	where $K_H$ is given in Remark~\ref{remark.local Lip. of h}.
\end{lemma}

\begin{proof}
	Recall the expression of $H_\gamma$ defined in \eqref{eq.tamedSG}. For all $\theta\in\mathbb{R}^d$ and $x\in\mathbb{R}^k$, we have
	\begin{align*}
		H_\gamma(\theta,x)-H(\theta,x)&=\frac{1-\sqrt{1+\gamma^{-1}|\theta|^{4r}}}{\sqrt{1+\gamma^{-1}|\theta|^{4r}}}\left(H(\theta,x)-m\theta\right)\\
		&=-\frac{\gamma^{-1}|\theta|^{4r}}{\sqrt{1+\gamma^{-1}|\theta|^{4r}}(1+\sqrt{1+\gamma^{-1}|\theta|^{4r}})}(H(\theta,x)-m\theta).
	\end{align*}
	By Remark~\ref{remark.local Lip. of h}, and the inequalities $|\theta|^{8r}\leq 1+|\theta|^{12r}$ and $|\theta|^{8r+2}\leq 1+|\theta|^{12r}$ for $r\geq 1/2$, we have
	\begin{align*}
		|H_\gamma(\theta,x)-H(\theta,x)|^2&\leq 2\gamma^{-2}|\theta|^{8r}\left(|H(\theta,x)|^2+m^2|\theta|^2\right)\\
		&\leq 2\gamma^{-2}|\theta|^{8r}\left(2K_H^2(1+|x|)^{2\rho}(1+|\theta|^{4r})+m^2|\theta|^2\right)\\
		&\leq 2\gamma^{-2}\left(2K_H^2(1+|x|)^{2\rho}(|\theta|^{8r}+|\theta|^{12r})+m^2|\theta|^{8r+2}\right)\\
		&\leq 2\gamma^{-2}\left(4K_H^2(1+|x|)^{2\rho}(1+|\theta|^{12r})+m^2(1+|\theta|^{12r})\right)\\
		&\leq 2\gamma^{-2}\big[4K_H^2(1+|x|)^{2\rho}+m^2\big](1+|\theta|^{12r}).
	\end{align*}
	This completes the proof.
\end{proof}

\phantomsection
\begin{proof}[\textbf{Proof of Proposition~\ref{proposition: W_2 - interpolation and auxiliary}}]
	\label{Proof.W_2 - interpolation and auxiliary}
	Throughout the proof, we use symbols, e.g., $\mathsf{V},\mathsf{Z}$, to denote the expressions on the right-hand side of the corresponding (in)equalities. By the definition of Wasserstein-2 distance, we have
	\begin{equation}\label{eq.auxiliary-decomposition}
		W_2(\mathcal{L}(Z_t^\lambda,V_t^\lambda),\mathcal{L}(\zeta_t^{\lambda,n},V_t^{\lambda,n}))\leq (\e[|V_t^\lambda-V_t^{\lambda,n}|^2])^{1/2}+(\e[|Z_t^\lambda-\zeta_t^{\lambda,n}|^2])^{1/2}.
	\end{equation}
	We first establish an upper bound for the first term on the right-hand side of \eqref{eq.auxiliary-decomposition}. By Itô's formula together with the definitions in \eqref{Continuous-time interpolation} and \eqref{Auxiliary process with certain initial}, we obtain, for all $t\in(nT,(n+1)T]$, that
	\begin{align*}
		\e[|V_t^\lambda-V_t^{\lambda,n}|^2]&=\e[|V_{nT}^\lambda-V_{nT}^{\lambda,n}|^2]-2\lambda\gamma\int_{nT}^t\e[\la V_s^\lambda-V_s^{\lambda,n},V_{\bs}^\lambda-V_s^{\lambda,n}\ra]\,\dd s\\
		&\quad-2\lambda\int_{nT}^t\e[\la V_s^\lambda-V_s^{\lambda,n},H_\gamma(Z_{\bs}^\lambda,X_{\us})-h_{MY,\epsilon}(\zeta_s^{\lambda,n})\ra]\,\dd s,
	\end{align*}
	which implies that
	\begin{align}
		\frac{\dd}{\dd t}\e[|V_t^\lambda-V_t^{\lambda,n}|^2]&=-2\lambda\gamma\e[\la V_t^\lambda-V_t^{\lambda,n},V_{\bt}^{\lambda}-V_t^{\lambda,n}\ra]\tag{$\mathsf{V}_1$}\label{eq.V_11}\\
		&\quad-2\lambda\e[\la V_t^\lambda-V_t^{\lambda,n},H_\gamma(Z_{\bt}^\lambda,X_{\ut})-h_{MY,\epsilon}(\zeta_t^{\lambda,n})\ra].\tag{$\mathsf{V}_2$}\label{eq.V_21}
	\end{align}
	By using \eqref{Continuous-time interpolation}, we can further decompose \hyperref[eq.V_11]{$\mathsf{V}_1$} as	
	\begin{align}
		\mathsf{V}_1&=2\lambda\gamma\e[\la V_t^\lambda-V_t^{\lambda,n},V_t^\lambda-V_{\bt}^\lambda\ra]-2\lambda\gamma\e[|V_t^\lambda-V_t^{\lambda,n}|^2] \notag\\
		&=-2\lambda^2\gamma\e\left[\left\la V_t^\lambda-V_t^{\lambda,n},\int_{\bt}^t\big(\gamma V_{\bs}^\lambda+H_\gamma(Z_{\bs}^\lambda,X_{\us})\big)\,\dd s\right\ra\right] \tag{$\mathsf{V}_{1,1}$}\\
		&\quad+2\lambda\gamma\e\left[\left\la V_t^\lambda-V_t^{\lambda,n},\sqrt{2\lambda\gamma\beta^{-1}}(W_t^\lambda-W_{\bt}^\lambda)\right\ra\right]\tag{$\mathsf{V}_{1,2}$}\\
		&\quad-2\lambda\gamma\e[|V_t^\lambda-V_t^{\lambda,n}|^2].\notag
	\end{align}
	By Lemmas~\ref{lemma.2 moment of algorithm} and \ref{lemma.2nd bound of taming factor}, together with Young's inequality, we have
	\begin{align}\label{eq.V_1,1}
		\mathsf{V}_{1,1}&\leq \frac{\lambda\gamma}{9}\e[|V_t^\lambda-V_t^{\lambda,n}|^2]+18\lambda^3\gamma \left(\gamma^2\e[|V_{\bt}^\lambda|^2]+\e[|H_\gamma(Z_{\bt}^\lambda,X_{\ut})|^2]\right)\notag\\
		&\leq \frac{\lambda\gamma}{9}\e[|V_t^\lambda-V_t^{\lambda,n}|^2]+18\lambda^3\gamma^3\bar{B}_2+36m^2\lambda^3\gamma \bar{C}_2+18\lambda^3\gamma^2C_H.
	\end{align}
	Furthermore, by using \eqref{Continuous-time interpolation} and \eqref{Auxiliary process with certain initial}, we have 
	\begin{align}
		\mathsf{V}_{1,2}&=2\lambda\gamma\e[\la V_{\bt}^\lambda-V_{\bt}^{\lambda,n},\sqrt{2\lambda\gamma\beta^{-1}}(W_t^\lambda-W_{\bt}^\lambda)\ra]\tag{$\mathsf{V}_{1,2,1}$}\\
		&\quad-2\lambda\gamma\e\left[\left\la\int_{\bt}^t\lambda\gamma(V_{\bs}^\lambda-V_s^{\lambda,n})\,\dd s,\sqrt{2\lambda\gamma\beta^{-1}}(W_t^\lambda-W_{\bt}^\lambda)\right\ra\right]\tag{$\mathsf{V}_{1,2,2}$}\\
		&\quad-2\lambda\gamma\e\left[\left\la \int_{\bt}^t\lambda(H_\gamma(Z_{\bs}^\lambda,X_{\us})-h_{MY,\epsilon}(\zeta_s^{\lambda,n}))\,\dd s, \sqrt{2\lambda\gamma\beta^{-1}}(W_t^\lambda-W_{\bt}^\lambda)\right\ra\right]\tag{$\mathsf{V}_{1,2,3}$}.
	\end{align}
	Define the filtration $\mathcal{H}_t^\lambda\coloneq \mathcal{F}_t^\lambda\vee\mathcal{G}_{\bt}\vee\sigma(\theta_0,\nu_0)$ for all $t\geq 0$. We notice that
	\begin{align}
		\mathsf{V}_{1,2,1}&=2\lambda\gamma\e[\e[\la V_{\bt}^\lambda-V_{\bt}^{\lambda,n},\sqrt{2\lambda\gamma\beta^{-1}}(W_t^\lambda-W_{\bt}^\lambda)\ra\mid\mathcal{H}_{\bt}^\lambda]]\notag\\
		&=2\lambda\gamma\e[\la V_{\bt}^\lambda-V_{\bt}^{\lambda,n},\sqrt{2\lambda\gamma\beta^{-1}}\e[W_t^\lambda-W_{\bt}^\lambda\mid\mathcal{H}_{\bt}^\lambda]\ra]=0.\label{eq.V_1,2,1}
	\end{align}
	Next, by applying Young's inequality, Cauchy-Schwarz inequality, and Lemma~\ref{lemma.mse of interpolation}, we derive
	\begin{align}
		\mathsf{V}_{1,2,2}&\leq\frac{2\lambda^2\gamma^2}{9}\int_{\bt}^t\e[|V_s^\lambda-V_s^{\lambda,n}|^2]\,\dd s+\frac{2\lambda^2\gamma^2}{9}\int_{\bt}^t\e[|V_{\bs}^\lambda-V_s^\lambda|^2]\,\dd s\notag\\
		&\quad+18\lambda^3\gamma^3\beta^{-1}\e[|W_t^\lambda-W_{\bt}^\lambda|^2]\notag\\
		&\leq \frac{2\lambda\gamma}{9}\sup_{t\in(nT,(n+1)T]}\e[|V_t^\lambda-V_t^{\lambda,n}|^2]+\frac{2}{9}\lambda^3\gamma^3 C_{1,v}+18\lambda^3\gamma^3 d/\beta.\label{eq.V_1,2,2}
	\end{align}
	By Lemmas~\ref{lemma.2 moment of algorithm} and \ref{lemma.2nd bound of taming factor},
	together with \eqref{eq.Lip MY h}, the condition
	$\gamma\geq\gamma_{\min}\geq \max\{1,K\}$, and Young's inequality, we obtain
	\begin{align}
		\mathsf{V}_{1,2,3}&\leq 2\lambda^2\gamma^{-1}\left(\int_{\bt}^t\e[|H_\gamma(Z_{\bs}^\lambda,X_{\us})|^2]\,\dd s+\int_{\bt}^t\e[|h_{MY,\epsilon}(\zeta_s^{\lambda,n})|^2]\,\dd s\right)+2\lambda^3\gamma^4 d/\beta\notag\\
		&\leq 2\lambda^2\gamma^{-1}(2m^2\bar{C}_2+C_H\gamma)+4\lambda^2\gamma^{-1}(K^2 C_\zeta^\#+|h(0)|^2)+2\lambda^3\gamma^4 d/\beta\notag\\
		&\leq 4m^2\lambda^2\bar{C}_2+2\lambda^2 C_H+4\lambda^2\gamma C_\zeta^\#+4\lambda^2|h(0)|^2+2\lambda^3\gamma^4 d/\beta.\label{eq.V_1,2,3}
	\end{align}
	Combining \eqref{eq.V_1,2,1}-\eqref{eq.V_1,2,3} and using
	$\gamma\geq\gamma_{\min}\geq \max\{1,K\}$ yields
	\begin{align}\label{eq.V_1,2}
		\begin{split}
			\mathsf{V}_{1,2}&\leq\frac{2\lambda\gamma}{9}\sup_{t\in(nT,(n+1)T]}\e[|V_t^\lambda-V_t^{\lambda,n}|^2]+\lambda^2\gamma\left(4m^2\bar{C}_2+2C_H+4C_\zeta^\#+4|h(0)|^2\right)\\
			&\quad+\lambda^3\gamma^4\left(\frac{2}{9}C_{1,v}+\frac{20d}{\beta}\right).
		\end{split}
	\end{align}
	Hence, substituting \eqref{eq.V_1,1} and \eqref{eq.V_1,2} into \hyperref[eq.V_11]{$\mathsf{V}_1$}, together with $\gamma\geq\gamma_{\min}\geq 1$, yields
	\begin{align}\label{eq.V_1}
		\begin{split}
			\mathsf{V}_1&\leq-2\lambda\gamma\e[|V_t^\lambda-V_t^{\lambda,n}|^2]+ \frac{3\lambda\gamma}{9}\sup_{t\in(nT,(n+1)T]}\e[|V_t^\lambda-V_t^{\lambda,n}|^2]\\
			&\quad+\lambda^2\gamma\left(4m^2\bar{C}_2+2C_H+4C_\zeta^\#+4|h(0)|^2\right)\\
			&\quad+\lambda^3\gamma^4\left(\frac{2}{9}C_{1,v}+\frac{20d}{\beta}+18\bar{B}_2+36m^2\bar{C}_2+18C_H\right).
		\end{split}
	\end{align}
	Now, we establish an upper bound for \hyperref[eq.V_21]{$\mathsf{V}_2$}, which can be decomposed as follows:
	\begin{align}
		\mathsf{V}_2&=-2\lambda\e[\la V_t^\lambda-V_t^{\lambda,n},H_\gamma(Z_{\bt}^\lambda,X_{\ut})-H(Z_{\bt}^\lambda,X_{\ut})\ra]\tag{$\mathsf{V}_{2,1}$}\label{eq.V_2,11}\\
		&\quad-2\lambda\e[\la V_t^\lambda-V_t^{\lambda,n},H(Z_{\bt}^\lambda,X_{\ut})-h(Z_{\bt}^\lambda)\ra]\tag{$\mathsf{V}_{2,2}$}\label{eq.V_2,21}\\
		&\quad-2\lambda\e[\la V_t^\lambda-V_t^{\lambda,n},h(Z_{\bt}^\lambda)-h_{MY,\epsilon}(Z_{\bt}^\lambda)\ra]\tag{$\mathsf{V}_{2,3}$}\label{eq.V_2,31}\\
		&\quad-2\lambda\e[\la V_t^\lambda-V_t^{\lambda,n},h_{MY,\epsilon}(Z_{\bt}^\lambda)-h_{MY,\epsilon}(Z_t^\lambda)\ra]\tag{$\mathsf{V}_{2,4}$}\label{eq.V_2,41}\\
		&\quad-2\lambda\e[\la V_t^\lambda-V_t^{\lambda,n},h_{MY,\epsilon}(Z_t^\lambda)-h_{MY,\epsilon}(\zeta_t^{\lambda,n})\ra]\tag{$\mathsf{V}_{2,5}$}\label{eq.V_2,51}.
	\end{align}
	We proceed with establishing an upper bound for \hyperref[eq.V_2,11]{$\mathsf{V}_{2,1}$}. Applying Lemmas~\ref{lemma.2q moment of algorithm}, \ref{lemma.mse - H_lambda and H}, and Young's inequality, we obtain
	\begin{align}
		\mathsf{V}_{2,1}&\leq \frac{\lambda\gamma}{9}\e[|V_t^\lambda-V_t^{\lambda,n}|^2]+9\lambda\gamma^{-1}\e[|H_\gamma(Z_{\bt}^\lambda,X_{\ut})-H(Z_{\bt}^\lambda,X_{\ut})|^2]\notag\\
		&\leq \frac{\lambda\gamma}{9}\e[|V_t^\lambda-V_t^{\lambda,n}|^2]+18\lambda\gamma^{-3}\left(4K_H^2\e[(1+|X_0|)^{2\rho}]+m^2\right)(1+\e[|Z_{\bt}^\lambda|^{12r}])\notag\\
		&\leq \frac{\lambda\gamma}{9}\e[|V_t^\lambda-V_t^{\lambda,n}|^2]+18\lambda\gamma^{-3}\left(4K_H^2\e[(1+|X_0|)^{2\rho}]+m^2\right)(1+\bar{C}_{12r}).\label{eq.V_2,1}
	\end{align}
	To establish an upper bound for \hyperref[eq.V_2,21]{$\mathsf{V}_{2,2}$}, we use \eqref{Continuous-time interpolation} and \eqref{Auxiliary process with certain initial} to obtain
	\begin{align}
		\mathsf{V}_{2,2}&=-2\lambda\e[\la V_{\bt}^\lambda-V_{\bt}^{\lambda,n},H(Z_{\bt}^\lambda,X_{\ut})-h(Z_{\bt}^\lambda)\ra]\tag{$\mathsf{V}_{2,2,1}$}\\
		&\quad+2\lambda^2\gamma\e\left[\left\la\int_{\bt}^t(V_{\bs}^\lambda-V_s^{\lambda,n})\,\dd s,H(Z_{\bt}^\lambda,X_{\ut})-h(Z_{\bt}^\lambda)\right\ra\right]\tag{$\mathsf{V}_{2,2,2}$}\\
		&\quad+2\lambda^2\e\left[\left\la\int_{\bt}^t(H_\gamma(Z_{\bs}^\lambda,X_{\us})-h_{MY,\epsilon}(\zeta_s^{\lambda,n}))\,\dd s,H(Z_{\bt}^\lambda,X_{\ut})-h(Z_{\bt}^\lambda)\right\ra\right]\tag{$\mathsf{V}_{2,2,3}$}.
	\end{align}
	Since $\mathcal{H}_t^\lambda=\mathcal{F}_t^\lambda\vee\mathcal{G}_{\bt}\vee\sigma(\theta_0,\nu_0)$ and $X_{\ut}$ is independent of $\mathcal{H}_{\bt}$, we have that
	\begin{align}
		\mathsf{V}_{2,2,1}&=-2\lambda\e[\e[\la V_{\bt}^\lambda-V_{\bt}^{\lambda,n},H(Z_{\bt}^\lambda,X_{\ut})-h(Z_{\bt}^\lambda)\ra\mid\mathcal{H}_{\bt}]]\notag\\
		&=-2\lambda\e[\la V_{\bt}^\lambda-V_{\bt}^{\lambda,n},h(Z_{\bt}^\lambda)-h(Z_{\bt}^\lambda)\ra]=0.\label{eq.V_2,2,1}
	\end{align}
	Moreover, by \eqref{Continuous-time interpolation} and \eqref{Auxiliary process with certain initial}, we obtain
	\begin{align}
		\mathsf{V}_{2,2,2}&=2\lambda^2\gamma\e\left[\left\la \int_{\bt}^t (V_{\bs}^\lambda-V_{\bs}^{\lambda,n}) \,\dd s, H(Z_{\bt}^\lambda,X_{\ut})-h(Z_{\bt}^\lambda)\right\ra\right]\tag{$\mathsf{V}_{2,2,2}^a$}\\
		&\quad+2\lambda^3\gamma\e\left[\left\la \int_{\bt}^t\int_{\bs}^s (\gamma V_r^{\lambda,n}+h_{MY,\epsilon}(\zeta_r^{\lambda,n}))\,\dd r\, \dd s, H(Z_{\bt}^\lambda,X_{\ut})-h(Z_{\bt}^\lambda)\right\ra\right]\tag{$\mathsf{V}_{2,2,2}^b$}\\
		&\quad-2\lambda^2\gamma\e\left[\left\la \int_{\bt}^t\int_{\bs}^s \sqrt{2\lambda\gamma\beta^{-1}}\,\dd W_r^\lambda\, \dd s, H(Z_{\bt}^\lambda,X_{\ut})-h(Z_{\bt}^\lambda)\right\ra\right]\tag{$\mathsf{V}_{2,2,2}^c$}.
	\end{align}
	Since $\int_{\bt}^t(V_{\bs}^\lambda-V_{\bs}^{\lambda,n}) \,\dd s=(t-\bt)(V_{\bt}^\lambda-V_{\bt}^{\lambda,n})$, it follows that
	\begin{align}\label{eq.V_2,2,2,1}
		\mathsf{V}_{2,2,2}^a=2\lambda^2\gamma(t-\bt)\e[\la V_{\bt}^\lambda-V_{\bt}^{\lambda,n},H(Z_{\bt}^\lambda,X_{\ut})-h(Z_{\bt}^\lambda)\ra]=0,
	\end{align}
	where the second equality follows from the same argument as in the analysis of \eqref{eq.V_2,2,1}. We note, by Remark~\ref{remark.local Lip. of h} and Lemma~\ref{lemma.2q moment of algorithm}, for any $t\in(nT,(n+1)T]$, that
	\begin{align}
		&\e[|H(Z_{\bt}^\lambda,X_{\ut})|^2]+\e[|h(Z_{\bt}^\lambda)|^2]\notag\\
		&\leq 2K_H^2\e[(1+|X_{\ut}|)^{2\rho}(1+|Z_{\bt}^\lambda|^{4r})] +2L_h^2\e[(1+|Z_{\bt}^\lambda|)^{4r}]+2|h(0)|^2\notag\\
		&\leq 2K_H^2\e[(1+|X_0|)^{2\rho}](1+\bar{C}_{4r}) + 2^{4r}L_h^2(1+\bar{C}_{4r})+2|h(0)|^2\leq c_A,\label{eq.cA}
	\end{align}
	where $c_A\coloneq (2K_H^2\e[(1+|X_0|)^{2\rho}]+2^{4r}L_h^2)(1+\bar{C}_{4r})+2|h(0)|^2=\mathcal{O}((d/\beta)^{2r})$.
	By Young's inequality, Cauchy-Schwarz inequality, Lemma~\ref{lemma.2nd bound of auxiliary process}, \eqref{eq.Lip MY h} and \eqref{eq.cA}, together with $\gamma\geq\gamma_{\min}\geq\max\{1,K\}$, we have
	\begin{align}
		\mathsf{V}_{2,2,2}^b&\leq 2\lambda^3\gamma\int_{\bt}^t\int_{\bs}^s \big(\gamma^2\e[|V_r^{\lambda,n}|^2]+\e[|h_{MY,\epsilon}(\zeta_r^{\lambda,n})|^2]\big)\,\dd r\,\dd s\notag\\
		&\quad+2\lambda^3\gamma\big(\e[|H(Z_{\bt}^\lambda,X_{\ut})|^2]+\e[|h(Z_{\bt}^\lambda)|^2]\big)\notag\\
		&\leq \lambda^3\gamma^3\left(2C_V^\#+4C_\zeta^\#+4|h(0)|^2+2c_A\right).\label{eq.V_2,2,2,2}
	\end{align}
	Furthermore, we note that, an application of Itô's formula yields
	\[(t-\bt)\int_{\bt}^t\,\dd W_s^\lambda=\int_{\bt}^t\int_{\bs}^s\,\dd r\,\dd W_s^\lambda+\int_{\bt}^t\int_{\bs}^s\,\dd W_r^\lambda\,\dd s.\]
	Hence, by applying \eqref{eq.cA}, we obtain that
	\begin{align}
		\mathsf{V}_{2,2,2}^c&=-2\lambda^2\gamma\e\left[\left\la (t-\bt)\sqrt{2\lambda\gamma\beta^{-1}}\int_{\bt}^t\,\dd W_s^\lambda, H(Z_{\bt}^\lambda,X_{\ut})-h(Z_{\bt}^\lambda)\right\ra\right]\notag\\
		&\quad+2\lambda^2\gamma\e\left[\left\la \sqrt{2\lambda\gamma\beta^{-1}}\int_{\bt}^t\int_{\bs}^s\,\dd r\,\dd W_s^\lambda, H(Z_{\bt}^\lambda,X_{\ut})-h(Z_{\bt}^\lambda)\right\ra\right]\notag\\
		&\leq 2\lambda^3\gamma^2\beta^{-1}\e[|W_t^\lambda-W_{\bt}^\lambda|^2]+4\lambda^2\gamma\big(\e[|H(Z_{\bt}^\lambda,X_{\ut})|^2]+\e[|h(Z_{\bt}^\lambda)|^2]\big)\notag\\
		&\quad+2\lambda^3\gamma^2\beta^{-1}\e\left[\left|\int_{\bt}^t(s-\bs)\,\dd W_s^\lambda\right|^2\right]\notag\\
		&\leq 4\lambda^3\gamma^2 d/\beta+4\lambda^2\gamma c_A.
	\end{align}
	Moreover, using Young's inequality, Lemmas~\ref{lemma.2 moment of algorithm}
	and \ref{lemma.2nd bound of taming factor}, \eqref{eq.Lip MY h} and
	\eqref{eq.cA}, together with $\gamma\geq\gamma_{\min}\geq\max\{1,K\}$, we obtain
	\begin{align}
		\mathsf{V}_{2,2,3}&\leq 2\lambda^2\gamma^{-1}\left(\e[|H_\gamma(Z_{\bt}^\lambda,X_{\ut})|^2]+\sup_{t\in(nT,(n+1)T]}\e[|h_{MY,\epsilon}(\zeta_t^{\lambda,n})|^2]\right)\notag\\
		&\quad+2\lambda^2\gamma\left(\e[|H(Z_{\bt}^\lambda,X_{\ut})|^2]+\e[|h(Z_{\bt})|^2]\right)\notag\\
		&\leq 2\lambda^2\gamma^{-1}\big(2m^2\bar{C}_2+C_H\gamma+2K^2C_\zeta^\#+2|h(0)|^2\big)+2\lambda^2\gamma c_A\notag\\
		&\leq \lambda^2\gamma\left(4m^2\bar{C}_2+2C_H+4C_\zeta^\#+4|h(0)|^2+2c_A\right).\label{eq.V_2,2,3}
	\end{align}
	By using \eqref{eq.V_2,2,1}, \eqref{eq.V_2,2,2,1}, \eqref{eq.V_2,2,2,2}-\eqref{eq.V_2,2,3} and the condition $\gamma\geq\gamma_{\min}\geq 1$, we obtain that, 
	\begin{align}\label{eq.V_2,2}
		\begin{split}
			\mathsf{V}_{2,2}&\leq\lambda^2\gamma\left(4m^2\bar{C}_2+2C_H+4C_\zeta^\#+4|h(0)|^2+6c_A\right) \\
			&\quad+\lambda^3\gamma^4\left(2C_V^\#+4C_\zeta^\#+4|h(0)|^2+2c_A+\frac{4d}{\beta}\right).
		\end{split}
	\end{align}
	Next, we consider establishing an upper bound for \hyperref[eq.V_2,31]{$\mathsf{V}_{2,3}$}. By \eqref{Continuous-time interpolation} and \eqref{Auxiliary process with certain initial}, we have
	\begin{align}
		\mathsf{V}_{2,3}&=-2\lambda\e[\la V_{\bt}^\lambda-V_{\bt}^{\lambda,n},h(Z_{\bt}^\lambda)-h_{MY,\epsilon}(Z_{\bt}^\lambda)\ra]\tag{$\mathsf{V}_{2,3,1}$}\\
		&\quad+2\lambda^2\e\left[\left\la\int_{\bt}^t\gamma(V_{\bs}^\lambda-V_s^{\lambda,n})\,\dd s,h(Z_{\bt}^\lambda)-h_{MY,\epsilon}(Z_{\bt}^\lambda)\right\ra\right]\tag{$\mathsf{V}_{2,3,2}$}\\
		&\quad+2\lambda^2\e\left[\left\la\int_{\bt}^t(H_\gamma(Z_{\bs}^\lambda,X_{\us})-h_{MY,\epsilon}(\zeta_s^{\lambda,n}))\,\dd s,h(Z_{\bt}^\lambda)-h_{MY,\epsilon}(Z_{\bt}^\lambda)\right\ra\right]\tag{$\mathsf{V}_{2,3,3}$}.
	\end{align}
	Applying Young's inequality together with Lemmas~\ref{lemma.error of MY regularization} and \ref{lemma.2q moment of algorithm}, we obtain
	\begin{align}
		\mathsf{V}_{2,3,1}&\leq \frac{\lambda\gamma}{9}\sup_{t\in(nT,(n+1)T]}\e[|V_t^\lambda-V_t^{\lambda,n}|^2]+9\lambda\gamma^{-1}\e[|h(Z_{\bt}^\lambda)-h_{MY,\epsilon}(Z_{\bt}^\lambda)|^2]\notag\\
		&\leq \frac{\lambda\gamma}{9}\sup_{t\in(nT,(n+1)T]}\e[|V_t^\lambda-V_t^{\lambda,n}|^2]+9\cdot 2^{4r-2}\lambda\gamma^{-1}L_h^4\e[(1+|Z_{\bt}^\lambda|+R_0)^{8r-2}]\epsilon^2\notag\\
		&\leq \frac{\lambda\gamma}{9}\sup_{t\in(nT,(n+1)T]}\e[|V_t^\lambda-V_t^{\lambda,n}|^2]+ C_4\lambda\gamma^{-1}\epsilon^2,\label{eq.V_2,3,1}
	\end{align}
	where $C_4\coloneq 9\cdot 2^{12r-5}L_h^4\big((1+R_0)^{8r-2}+\bar{C}_{8r-2}\big)=\mathcal{O}((d/\beta)^{4r-1})$. Similarly, using Young's inequality, Remark~\ref{remark.local Lip. of h}, Lemmas~\ref{lemma.hMY leq h}, \ref{lemma.2 moment of algorithm} and \ref{lemma.mse of interpolation}, and the assumption that $\gamma\geq\gamma_{\min}\geq 1$ and $\lambda\leq\lambda_{\max,\gamma}\leq 1/(4\gamma)$, we have that
	\begin{align}
		\mathsf{V}_{2,3,2}&\leq \frac{2\lambda^2\gamma^2}{9}\int_{\bt}^t\e[|V_s^\lambda-V_s^{\lambda,n}|^2]\,\dd s+\frac{2\lambda^2\gamma^2}{9}\int_{\bt}^t\e[|V_{\bs}^\lambda-V_s^\lambda|^2]\,\dd s\notag\\
		&\quad+18\lambda^2\big(\e[|h(Z_{\bt}^\lambda)|^2]+\e[|h_{MY,\epsilon}(Z_{\bt}^\lambda)|^2]\big)\notag\\
		&\leq \frac{2\lambda\gamma}{9}\sup_{t\in(nT,(n+1)T]}\e[|V_t^\lambda-V_t^{\lambda,n}|^2]+\frac{2}{9}\lambda^3\gamma^3 C_{1,v}+36\lambda^2\e[|h(Z_{\bt}^\lambda)|^2]\notag\\
		&\leq \frac{2\lambda\gamma}{9}\sup_{t\in(nT,(n+1)T]}\e[|V_t^\lambda-V_t^{\lambda,n}|^2]+\frac{2}{9}\lambda^3\gamma^3 C_{1,v}+\lambda^2 c_D,
	\end{align}
	where $c_D\coloneq 36 (2^{4r}L_h^2(1+\bar{C}_{4r})+2|h(0)|^2)=\mathcal{O}((d/\beta)^{2r})$. Moreover, by using Young's inequality, Lemmas~\ref{lemma.hMY leq h},
	\ref{lemma.2 moment of algorithm}, \ref{lemma.2q moment of algorithm},
	\ref{lemma.2nd bound of auxiliary process}, and
	\ref{lemma.2nd bound of taming factor}, together with
	Remark~\ref{remark.local Lip. of h}, \eqref{eq.Lip MY h}, and the condition
	$\gamma\geq\gamma_{\min}\geq\max\{1,K\}$, we obtain that
	\begin{align}
		\mathsf{V}_{2,3,3}&\leq 2\lambda^2\gamma^{-1}\left(\e[|H_\gamma(Z_{\bt}^\lambda,X_{\ut})|^2]+\sup_{t\in(nT,(n+1)T]}\e[|h_{MY,\epsilon}(\zeta_t^{\lambda,n})|^2]\right)\notag\\
		&\quad+2\lambda^2\gamma\left(\e[|h(Z_{\bt}^\lambda)|^2]+\e[|h_{MY,\epsilon}(Z_{\bt}^\lambda)|^2]\right)\notag\\
		&\leq \lambda^2\gamma\left(4m^2\bar{C}_2+2C_H+4C_\zeta^\#+4|h(0)|^2\right)+4\lambda^2\gamma\e[|h(Z_{\bt}^\lambda)|^2]\notag\\
		&\leq\lambda^2\gamma\left(4m^2\bar{C}_2+2C_H+4C_\zeta^\#+12|h(0)|^2+2^{4r+2}L_h^2(1+\bar{C}_{4r})\right).\label{eq.V_2,3,3}
	\end{align}
	From \eqref{eq.V_2,3,1}-\eqref{eq.V_2,3,3} and $\gamma\geq\gamma_{\min}\geq 1$, it follows that
	\begin{align}\label{eq.V_2,3}
		\begin{split}
			\mathsf{V}_{2,3}&\leq\frac{\lambda\gamma}{3}\sup_{t\in(nT,(n+1)T]}\e[|V_t^\lambda-V_t^{\lambda,n}|^2]+C_4\lambda\gamma^{-1}\epsilon^2+\frac{2}{9}\lambda^3\gamma^4 C_{1,v}\\
			&\quad+\lambda^2\gamma\left(c_D+4m^2\bar{C}_2+2C_H+4C_\zeta^\#+12|h(0)|^2+2^{4r+2}L_h^2(1+\bar{C}_{4r})\right).
		\end{split}
	\end{align}
	To establish upper bounds for \hyperref[eq.V_2,41]{$\mathsf{V}_{2,4}$} and \hyperref[eq.V_2,51]{$\mathsf{V}_{2,5}$}, by applying Young's inequality, Lemma~\ref{lemma.mse of interpolation}, Remark~\ref{remark.MY}, and $\gamma\geq\gamma_{\min}\geq K$, we obtain
	\begin{align}
		\mathsf{V}_{2,4}&\leq \frac{\lambda\gamma}{9}\e[|V_t^\lambda-V_t^{\lambda,n}|^2]+9\lambda\gamma^{-1}\e[|h_{MY,\epsilon}(Z_{\bt}^\lambda)-h_{MY,\epsilon}(Z_t^\lambda)|^2]\notag\\
		&\leq \frac{\lambda\gamma}{9}\sup_{t\in(nT,(n+1)T]}\e[|V_t^\lambda-V_t^{\lambda,n}|^2]+9\lambda^2\gamma^{-1}K^2\bar{B}_2\notag\\
		&\leq \frac{\lambda\gamma}{9}\sup_{t\in(nT,(n+1)T]}\e[|V_t^\lambda-V_t^{\lambda,n}|^2]+9\lambda^2\gamma\bar{B}_2,\label{eq.V_2,4}
	\end{align}
	and
	\begin{align}
		\mathsf{V}_{2,5}&\leq \frac{\lambda\gamma}{9}\e[|V_t^\lambda-V_t^{\lambda,n}|^2]+9\lambda\gamma^{-1}K^2\e[|Z_t^\lambda-\zeta_t^{\lambda,n}|^2]\notag\\
		&\leq \frac{\lambda\gamma}{9}\sup_{t\in(nT,(n+1)T]}\e[|V_t^\lambda-V_t^{\lambda,n}|^2]+\frac{\lambda\gamma}{24}\e[|Z_t^\lambda-\zeta_t^{\lambda,n}|^2],\label{eq.V_2,5}
	\end{align}
	where the last inequality follows from the restriction that $\gamma\geq\gamma_{\min}\geq 6\sqrt{6}K$. 
	Combining \eqref{eq.V_1}, \eqref{eq.V_2,1}, \eqref{eq.V_2,2}, \eqref{eq.V_2,3}-\eqref{eq.V_2,5}, we obtain
	\begin{align}\label{eq.differential of V}
		\begin{split}
			\frac{\dd}{\dd t}\e[|V_t^\lambda-V_t^{\lambda,n}|^2]&\leq -2\lambda\gamma\e[|V_t^\lambda-V_t^{\lambda,n}|^2]+\lambda\gamma\sup_{t\in(nT,(n+1)T]}\e[|V_t^\lambda-V_t^{\lambda,n}|^2]\\
			&\quad+\lambda S_{\lambda,\gamma}+\frac{\lambda\gamma}{24}\e[|Z_t^\lambda-\zeta_t^{\lambda,n}|^2],
		\end{split}
	\end{align}
	where
	\begin{align*}
		S_{\lambda,\gamma}&\coloneq C_1\lambda^2\gamma^4+C_2\lambda\gamma+C_3\gamma^{-3}+C_4\gamma^{-1}\epsilon^2,\label{eq.S_lambda,gamma}\\
		C_1&\coloneq 4C_{1,v}/9+24d/\beta+18\bar{B}_2+36m^2\bar{C}_2+18C_H+2C_V^\#\\
		&\quad+4C_\zeta^\#+4|h(0)|^2+2c_A=\mathcal{O}((d/\beta)^{2r}),\\
		C_2&\coloneq 12m^2\bar{C}_2+6C_H+12C_\zeta^\#+20|h(0)|^2+6c_A+c_D\\
		&\quad\ +2^{4r+2}L_h^2(1+\bar{C}_{4r})+9\bar{B}_2=\mathcal{O}((d/\beta)^{2r}),\\
		C_3&\coloneq 18\big(4K_H^2\e[(1+|X_0|)^{2\rho}]+m^2\big)(1+\bar{C}_{12r})=\mathcal{O}((d/\beta)^{6r}),\\
		C_4&\coloneq 9\cdot 2^{12r-5}L_h^4\big((1+R_0)^{8r-2}+\bar{C}_{8r-2}\big)=\mathcal{O}((d/\beta)^{4r-1}).
	\end{align*}
	Combining Lemmas~\ref{lemma.2 moment of algorithm}, \ref{lemma.mse of interpolation}, and \ref{lemma.2nd bound of auxiliary process} with Cauchy-Schwarz inequality, under the assumptions that $\gamma\geq\gamma_{\min}\geq 1$ and $\lambda\leq\lambda_{\max,\gamma}\leq 1/(4\gamma)$, we obtain
	\begin{align*}
		\sup_{t\in(nT,(n+1)T]}\e[|Z_t^\lambda-\zeta_t^{\lambda,n}|^2]&\leq 3\sup_{t\in(nT,(n+1)T]}\left(\e[|\zeta_t^{\lambda,n}|^2]+\e[|Z_{\bt}^\lambda|^2]+\e[|Z_t^\lambda-Z_{\bt}^\lambda|^2]\right)\leq C_J\\
		\sup_{t\in(nT,(n+1)T]}\e[|V_t^\lambda-V_t^{\lambda,n}|^2]&\leq 3\sup_{t\in(nT,(n+1)T]}\left(\e[|V_t^{\lambda,n}|^2]+\e[|V_{\bt}^\lambda|^2]+\e[|V_t^\lambda-V_{\bt}^\lambda|^2]\right)\leq C_M,
	\end{align*}
	where $C_J\coloneq 3C_\zeta^\#+3\bar{C}_2+3\bar{B}_2=\mathcal{O}(d/\beta)$ and $C_M\coloneq 3C_V^\#+3\bar{B}_2+3C_{1,v}=\mathcal{O}(d/\beta)$. 
	Consequently, \eqref{eq.differential of V} can be rewritten as
	\begin{align*}
		\frac{\dd}{\dd t}\e[|V_t^\lambda-V_t^{\lambda,n}|^2]&\leq -2\lambda\gamma\e[|V_t^\lambda-V_t^{\lambda,n}|^2]\\
		&\quad+\lambda\gamma\left(\frac{S_{\lambda,\gamma}}{\gamma}+\frac{1}{24}\sup_{t\in(nT,(n+1)T]}\e[|Z_t^\lambda-\zeta_t^{\lambda,n}|^2]\right.\\
		&\quad\quad\quad\quad\ \left.+\sup_{t\in(nT,(n+1)T]}\e[|V_t^\lambda-V_t^{\lambda,n}|^2]\right),
	\end{align*}
	which implies, by applying integrating factor, that
	\begin{align*}
		&\e[|V_t^\lambda-V_t^{\lambda,n}|^2]\\
		&\leq \frac{1}{2}\left(\frac{S_{\lambda,\gamma}}{\gamma}+\sup_{t\in(nT,(n+1)T]}\left(\frac{1}{24}\e[|Z_t^\lambda-\zeta_t^{\lambda,n}|^2]+\e[|V_t^\lambda-V_t^{\lambda,n}|^2]\right)\right)\left(1-e^{-2\lambda\gamma (t-nT)}\right)\\
		&\leq \frac{1}{2}\sup_{t\in(nT,(n+1)T]}\e[|V_t^\lambda-V_t^{\lambda,n}|^2]+\frac{1}{2}\left(\frac{S_{\lambda,\gamma}}{\gamma}+\frac{1}{24}\sup_{t\in(nT,(n+1)T]}\e[|Z_t^\lambda-\zeta_t^{\lambda,n}|^2]\right).
	\end{align*}
	Taking the supremum over $t\in(nT,(n+1)T]$ gives
	\begin{equation}\label{eq.mse of momentum}
		\sup_{t\in(nT,(n+1)T]}\e[|V_t^\lambda-V_t^{\lambda,n}|^2]\leq \frac{S_{\lambda,\gamma}}{\gamma}+\frac{1}{24}\sup_{t\in(nT,(n+1)T]}\e[|Z_t^\lambda-\zeta_t^{\lambda,n}|^2].
	\end{equation}
	
	To establish an upper bound for the second term on the right-hand side of \eqref{eq.auxiliary-decomposition}, we apply Itô's formula to the expressions given in \eqref{Continuous-time interpolation} and \eqref{Auxiliary process with certain initial} to obtain that
	\[\e[|Z_t^\lambda-\zeta_t^{\lambda,n}|^2]=2\lambda\int_{nT}^t\e[\la Z_s^\lambda-\zeta_s^{\lambda,n},V_{\bs}^\lambda-V_s^{\lambda,n}\ra]\,\dd s,\]
	which can be further rewritten as
	\begin{align*}
		\e[|Z_t^\lambda-\zeta_t^{\lambda,n}|^2]&=2\lambda\int_{nT}^t\e[\la Z_s^\lambda-\zeta_s^{\lambda,n},V_{\bs}^\lambda-V_{\bs}^{\lambda,n}\ra]\,\dd s\tag{$\mathsf{Z}_{1}$}\label{eq.Z_11}\\
		&\quad+2\lambda^2\int_{nT}^t\e\left[\left\la Z_s^\lambda-\zeta_s^{\lambda,n},\int_{\bs}^s\left(\gamma V_r^{\lambda,n}+h_{MY,\epsilon}(\zeta_r^{\lambda,n})\right)\,\dd r\right\ra\right]\,\dd s\tag{$\mathsf{Z}_{2}$}\label{eq.Z_21}\\
		&\quad-2\lambda\int_{nT}^t\e\left[\left\la Z_s^\lambda-\zeta_s^{\lambda,n},\sqrt{2\lambda\gamma\beta^{-1}}\int_{\bs}^s\,\dd W_r^\lambda \right\ra\right]\,\dd s.\tag{$\mathsf{Z}_{3}$}
	\end{align*}
	Applying Young's inequality to \hyperref[eq.Z_11]{$\mathsf{Z}_1$} yields 
	\begin{align}\label{eq.Z_1}
		\mathsf{Z}_{1}\leq \frac{\lambda}{4}\int_{nT}^t\e[|Z_s^\lambda-\zeta_s^{\lambda,n}|^2]\,\dd s+4\lambda\int_{nT}^t\e[|V_{\bs}^\lambda-V_{\bs}^{\lambda,n}|^2]\,\dd s.
	\end{align}
	Moreover, applying Young's inequality and Cauchy-Schwarz inequality to \hyperref[eq.Z_21]{$\mathsf{Z}_2$} yields
	\begin{equation}\label{eq.Z_2}
		\mathsf{Z}_{2}\leq \frac{\lambda}{4}\int_{nT}^t\e[|Z_s^\lambda-\zeta_s^{\lambda,n}|^2]\,\dd s+4\lambda^3\int_{nT}^t\e\left[\int_{\bs}^s|\gamma V_r^{\lambda,n}+h_{MY,\epsilon}(\zeta_r^{\lambda,n})|^2\,\dd r\right]\,\dd s.
	\end{equation}
	By Lemma~\ref{lemma.2nd bound of auxiliary process}, \eqref{eq.Lip MY h}, and $\gamma\geq\gamma_{\min}\geq K$, we obtain
	\begin{align}
		&4\lambda^3\int_{nT}^t\e\left[\int_{\bs}^s|\gamma V_r^{\lambda,n}+h_{MY,\epsilon}(\zeta_r^{\lambda,n})|^2\,\dd r\right]\,\dd s\notag\\
		&\leq 8 \lambda^3\gamma^2\int_{nT}^t\sup_{nT\leq r\leq s}\e[|V_r^{\lambda,n}|^2]\,\dd s+16\lambda^3\int_{nT}^t\e\left[\int_{\bs}^s\left(K^2|\zeta_r^{\lambda,n}|^2+|h(0)|^2\right)\,\dd r\right]\,\dd s\notag\\
		&\leq 8\lambda^2\gamma^2 C_V^\#+16\lambda^2 \gamma^2 C_\zeta^\#+16\lambda^2|h(0)|^2.\label{eq.Z_2-1}
	\end{align}
	Thus, by substituting \eqref{eq.Z_2-1} into \eqref{eq.Z_2}, it holds that
	\begin{equation}\label{eq.Z_1,2}
		\mathsf{Z}_{2}\leq \frac{\lambda}{4}\int_{nT}^t\e[|Z_s^\lambda-\zeta_s^{\lambda,n}|^2]\,\dd s+8\lambda^2\gamma^2 C_V^\#+16\lambda^2 \gamma^2 C_\zeta^\#+16\lambda^2|h(0)|^2.
	\end{equation}
	Next, recall that $\mathcal{H}_t^\lambda= \mathcal{F}_t^\lambda\vee\mathcal{G}_{\bt}\vee\sigma(\theta_0,\nu_0)$. We have, for any $s\in(nT,t]$, that
	\begin{align}
		&\e\left[\left\la Z_{\bs}^\lambda-\zeta_{\bs}^{\lambda,n},\sqrt{2\lambda\gamma\beta^{-1}}(W_s^\lambda-W_{\bs}^\lambda) \right\ra\right]\notag\\
		&=\e\left[\e\left[\left\la Z_{\bs}^\lambda-\zeta_{\bs}^{\lambda,n},\sqrt{2\lambda\gamma\beta^{-1}}(W_s^\lambda-W_{\bs}^\lambda)\right\ra\mid\mathcal{H}_{\bs}^\lambda\right] \right]\notag\\
		&=\e\left[\left\la Z_{\bs}^\lambda-\zeta_{\bs}^{\lambda,n},\sqrt{2\lambda\gamma\beta^{-1}}\e[W_s^\lambda-W_{\bs}^\lambda\mid\mathcal{H}_{\bs}^\lambda]\right\ra\right]=0.\label{eq.Z_3-1}
	\end{align}
	Moreover, notice that
	\begin{equation}\label{eq.Z_31}
		-2\lambda\int_{nT}^t\e\left[\left\la\lambda\int_{\bs}^s V_{\br}^\lambda\,\dd r,\sqrt{2\lambda\gamma\beta^{-1}}(W_s^\lambda-W_{\bs}^\lambda)\right\ra\right]\,\dd s=0.
	\end{equation}
	Thus, by Young's inequality, Cauchy-Schwarz inequality, \eqref{eq.Z_3-1}, \eqref{eq.Z_31} and Lemma~\ref{lemma.mse of interpolation}, we have 
	\begin{align}
		\mathsf{Z}_{3}&=-2\lambda\int_{nT}^t\e\left[\left\la Z_{\bs}^\lambda-\zeta_{\bs}^{\lambda,n},\sqrt{2\lambda\gamma\beta^{-1}}(W_s^\lambda-W_{\bs}^\lambda) \right\ra\right]\,\dd s\notag\\
		&\quad-2\lambda\int_{nT}^t\e\left[\left\la \lambda\int_{\bs}^s(V_{\br}^\lambda-V_r^{\lambda,n})\,\dd r,\sqrt{2\lambda\gamma\beta^{-1}}(W_s^\lambda-W_{\bs}^\lambda) \right\ra\right]\,\dd s\notag\\
		&\leq \lambda\int_{nT}^t\left(\e\left[\lambda\int_{\bs}^s|V_{\br}^\lambda-V_r^{\lambda,n}|^2\,\dd r\right]\,\dd s+2\lambda^2\gamma\beta^{-1}\e[|W_s^\lambda-W_{\bs}^\lambda|^2]\right)\,\dd s\notag\\
		&\leq \lambda^2\int_{nT}^t\e\left[\int_{\bs}^s|V_{\br}^\lambda-V_r^{\lambda,n}|^2\,\dd r\right]\,\dd s+2\lambda^2\gamma d/\beta\notag\\
		&\leq 2\lambda^2\int_{nT}^t\e\left[\int_{\bs}^s\big(|V_{\br}^\lambda-V_r^{\lambda}|^2+|V_r^\lambda-V_r^{\lambda,n}|^2\big)\,\dd r\right]\,\dd s+2\lambda^2\gamma d/\beta\notag\\
		&\leq 2\lambda^2\int_{nT}^t\int_{\bs}^s\e[|V_r^\lambda-V_r^{\lambda,n}|^2]\,\dd r\,\dd s+2\lambda^2\gamma C_{1,v}+2\lambda^2\gamma d/\beta.\label{eq.Z_3}
	\end{align}
	Combining \eqref{eq.Z_1}, \eqref{eq.Z_1,2} and \eqref{eq.Z_3}, together with $\gamma\geq\gamma_{\min}\geq 1$ and $\lambda\leq\lambda_{\max,\gamma}\leq 1/(4\gamma)$, yields
	\begin{align*}
		\e[|Z_t^\lambda-\zeta_t^{\lambda,n}|^2]&\leq \frac{\lambda}{2}\int_{nT}^t\e[|Z_s^\lambda-\zeta_s^{\lambda,n}|^2]\,\dd s+4\lambda\int_{nT}^t\e[|V_{\bs}^\lambda-V_{\bs}^{\lambda,n}|^2]\,\dd s\\
		&\quad+2\lambda^2\int_{nT}^t\int_{\bs}^s\e[|V_r^\lambda-V_r^{\lambda,n}|^2]\,\dd r\,\dd s\\
		&\quad+8\lambda^2\gamma^2 C_V^\#+16\lambda^2 \gamma^2 C_\zeta^\#+16\lambda^2|h(0)|^2+2\lambda^2\gamma C_{1,v}+2\lambda^2\gamma d/\beta\\
		&\leq \frac{\lambda}{2}\int_{nT}^t\e[|Z_s^\lambda-\zeta_s^{\lambda,n}|^2]\,\dd s+6\lambda\int_{nT}^t\sup_{nT\leq r\leq s}\e[|V_r^\lambda-V_r^{\lambda,n}|^2]\,\dd s\\
		&\quad+8\lambda^2\gamma^2 C_V^\#+16\lambda^2 \gamma^2 C_\zeta^\#+16\lambda^2|h(0)|^2+2\lambda^2\gamma C_{1,v}+2\lambda^2\gamma d/\beta\\
		&\leq\frac{1}{2}\sup_{t\in(nT,(n+1)T]}\e[|Z_t^\lambda-\zeta_t^{\lambda,n}|^2]+6\sup_{t\in(nT,(n+1)T]}\e[|V_t^\lambda-V_t^{\lambda,n}|^2]+C_5\lambda^2\gamma^2,
	\end{align*}
	where $C_5\coloneq 8C_V^\#+16C_\zeta^\#+16|h(0)|^2+2C_{1,v}+2d/\beta=\mathcal{O}(d/\beta)$, which implies, by using \eqref{eq.mse of momentum}, that
	\begin{align*}
		\e[|Z_t^\lambda-\zeta_t^{\lambda,n}|^2]&\leq \frac{1}{2}\sup_{t\in(nT,(n+1)T]}\e[|Z_t^\lambda-\zeta_t^{\lambda,n}|^2]+C_5\lambda^2\gamma^2\\
		&\quad+6\left(\frac{S_{\lambda,\gamma}}{\gamma}+\frac{1}{24}\sup_{t\in(nT,(n+1)T]}\e[|Z_t^\lambda-\zeta_t^{\lambda,n}|^2]\right)\\
		&\leq \frac{3}{4}\sup_{t\in(nT,(n+1)T]}\e[|Z_t^\lambda-\zeta_t^{\lambda,n}|^2]+\frac{6S_{\lambda,\gamma}}{\gamma}+C_5\lambda^2\gamma^2.
	\end{align*}
	Taking the supremum over $t\in(nT,(n+1)T]$ gives
	\begin{equation}\label{eq.sup Z-xi}
		\sup_{t\in(nT,(n+1)T]}\e[|Z_t^\lambda-\zeta_t^{\lambda,n}|^2]\leq\frac{24S_{\lambda,\gamma}}{\gamma}+4C_5\lambda^2\gamma^2.
	\end{equation}
	Substituting \eqref{eq.sup Z-xi} back into \eqref{eq.mse of momentum} yields
	\begin{equation}\label{eq.mse of velocity}
		\sup_{t\in(nT,(n+1)T]}\e[|V_t^\lambda-V_t^{\lambda,n}|^2]\leq \frac{2S_{\lambda,\gamma}}{\gamma}+\frac{C_5}{6}\lambda^2\gamma^2.
	\end{equation}
	Finally, in view of \eqref{eq.auxiliary-decomposition}, by using \eqref{eq.sup Z-xi} and \eqref{eq.mse of velocity}, we conclude that
	\[W_2^2(\mathcal{L}(Z_t^\lambda,V_t^\lambda),\mathcal{L}(\zeta_t^{\lambda,n},V_t^{\lambda,n}))\leq C_0^2\left(\frac{S_{\lambda,\gamma}}{\gamma}+\lambda^2\gamma^2\right),\]
	where $C_0\coloneq\sqrt{26(C_5+1)}=\mathcal{O}((d/\beta)^{1/2})$. This completes the proof.
\end{proof}

\phantomsection
\begin{proof}[\textbf{Proof of Theorem~\ref{theorem:contraction result of dalalyan}}]
	\label{proof.contraction result of dalalyan}
	The proof is adapted from \cite[Theorem 1]{dalalyan2020sampling}, we provide details here for explicit constants.
	Let $(r_t,R_t)$ and $(r_t',R_t')$ be two solutions of the underdamped Langevin SDE \eqref{eq.MY kinetic Langevin} with initial conditions $(r_0,R_0)$ and $(r_0',R_0')$, and define
	\begin{equation}\label{eq.y,z}
		y_t\coloneq (R_t-R_t')+\gamma(r_t-r_t'),\quad z_t\coloneq -(R_t-R_t')
	\end{equation}
	for all $t\geq 0$. By Lemma~\ref{lemma.MY}, $u_{MY,\epsilon}$ is twice continuously differentiable. Hence, by the fundamental theorem of calculus,
	\begin{equation}\label{eq.taylor-MY}
		h_{MY,\epsilon}(r_t)-h_{MY,\epsilon}(r_t')=S_t(r_t-r_t'),
	\end{equation}
	where 
	\begin{equation}\label{eq.S_t}
		S_t\coloneq\int_0^1 \operatorname{Hess}(u_{MY,\epsilon})(r_t'+s(r_t-r_t'))\,\dd s.
	\end{equation}
	By \eqref{eq.y,z}, \eqref{eq.taylor-MY} and the expression of \eqref{eq.MY kinetic Langevin}, we have
	\begin{align}
		\frac{\dd y_t}{\dd t}&=-\gamma(R_t-R_t')-(h_{MY,\epsilon}(r_t)-h_{MY,\epsilon}(r_t'))+\gamma(R_t-R_t')\notag\\
		&=-(h_{MY,\epsilon}(r_t)-h_{MY,\epsilon}(r_t'))=-S_t(r_t-r_t')=-\frac{S_t(y_t+z_t)}{\gamma},\label{eq.dy}
	\end{align}
	and
	\begin{align}\label{eq.dz}
		\begin{split}
			\frac{\dd z_t}{\dd t}&=\gamma(R_t-R_t')+(h_{MY,\epsilon}(r_t)-h_{MY,\epsilon}(r_t'))=-\gamma z_t+\frac{S_t(y_t+z_t)}{\gamma}.
		\end{split}
	\end{align}
	Combining \eqref{eq.dy} and \eqref{eq.dz}, we have
	\begin{align}
		\frac{\dd}{\dd t}|(y_t,z_t)|^2&=2\left\la y_t,\frac{\dd y_t}{\dd t}\right\ra+2\left\la z_t,\frac{\dd z_t}{\dd t}\right\ra\notag\\
		&=-\frac{2}{\gamma}\left\la y_t, S_t(y_t+z_t)\right\ra-2\gamma \left\la z_t, z_t\right\ra+\frac{2}{\gamma}\left\la z_t, S_t(y_t+z_t)\right\ra\notag\\
		&=-\frac{2}{\gamma}\la y_t,S_t y_t\ra+\frac{2}{\gamma} \la z_t,S_t z_t\ra -2\gamma|z_t|^2.\label{eq.derivative of y,z}
	\end{align}
	Using Remark~\ref{remark.MY}, we have, for all $t\geq 0$, that
	\begin{equation}\label{eq.quadratic form of S_t}
		\frac{m}{2}I_d\preceq S_t\preceq KI_d.
	\end{equation}
	Substituting \eqref{eq.quadratic form of S_t} into the right-hand side of \eqref{eq.derivative of y,z} yields
	\begin{align}
		\begin{split}
			\frac{\dd}{\dd t}|(y_t,z_t)|^2&\leq -\frac{m}{\gamma}|y_t|^2+\frac{2K}{\gamma}|z_t|^2-2\gamma|z_t|^2 \leq \frac{\max\{-m,2(K-\gamma^2)\}}{\gamma}\,|(y_t,z_t)|^2.
		\end{split}
	\end{align}
	By integrating factor and the assumption that $\gamma\geq\gamma_{\min}\geq \sqrt{K+m/2}$, we obtain, for all $t\geq 0$, that
	\[|(y_t,z_t)|^2\leq \exp\left(\frac{\max\{-m,2(K-\gamma^2)\}}{\gamma} t\right)|(y_0,z_0)|^2\leq \exp\left(-\frac{m}{\gamma} t\right)|(y_0,z_0)|^2,\]
	and thus
	\begin{align}\label{eq.contraction y,z}
		|(y_t,z_t)|\leq \exp\left(-\frac{m}{2\gamma} t\right)|(y_0,z_0)|.
	\end{align}
	Next, by the definition of $y_t$ and $z_t$ in \eqref{eq.y,z}, we have
	\[|y_0|^2\leq 2|R_0-R_0'|^2+2\gamma^2|r_0-r_0'|^2,\quad |z_0|^2=|R_0-R_0'|^2,\]
	so that
	\begin{equation}\label{eq.initial y,z}
		|(y_0,z_0)|\leq\sqrt{3+2\gamma^2}\,|(r_0-r_0',R_0-R_0')|.
	\end{equation}
	Moreover, using the identity $r_t-r_t'=(y_t+z_t)/\gamma$, along with \eqref{eq.contraction y,z} and \eqref{eq.initial y,z}, and the condition $\gamma\geq\gamma_{\min}\geq 1$, we have, for all $t\geq 0$, that
	\begin{align*}
		|r_t-r_t'|\leq\frac{\sqrt{2}}{\gamma}|(y_t,z_t)|&\leq \frac{\sqrt{2(3+2\gamma^2)}}{\gamma}\exp\left(-\frac{m}{2\gamma}t\right)|(r_0-r_0',R_0-R_0')|\\
		&\leq 4\exp\left(-\frac{m}{2\gamma}t\right)|(r_0-r_0',R_0-R_0')|.
	\end{align*}
	Taking expectations and then the infimum, we conclude that
	\[W_2(\mathcal{L}(r_t),\mathcal{L}(r_t'))\leq 4\exp\left(-\frac{m}{2\gamma}t\right)W_2(\mathcal{L}(r_0,R_0),\mathcal{L}(r_0',R_0')),\]
	which completes the proof.
\end{proof}

\phantomsection
\begin{proof}[\textbf{Proof of Proposition~\ref{theorem:W_2 - scaled and auxiliary}}]\label{Proof.W_2 - scaled and auxiliary}
	For all $n\in\mathbb{N}_0$, $t\in(nT,(n+1)T]$, and $\gamma\geq\gamma_{\min}$, by Theorem~\ref{theorem:contraction result of dalalyan}, Corollary~\ref{coro.r_t pi}, and Proposition~\ref{proposition: W_2 - interpolation and auxiliary}, we have
	\begin{align*}
		W_2\big(\mathcal{L}(\zeta_t^{\lambda,n}),\mathcal{L}(r_t^\lambda)\big)&\leq \sum_{k=1}^n W_2\big(\mathcal{L}(\zeta_t^{\lambda,k}),\mathcal{L}(\zeta_t^{\lambda,k-1})\big)\\
		&=\sum_{k=1}^n W_2(\mathcal{L}\big(\widehat{\zeta}_t^{kT,Z_{kT}^\lambda,V_{kT}^\lambda,\lambda}),\mathcal{L}(\widehat{\zeta}_t^{kT,\zeta_{kT}^{\lambda,k-1},V_{kT}^{\lambda,k-1},\lambda})\big)\\
		&\leq \sum_{k=1}^n 4 \exp\left(-\frac{m\lambda}{2\gamma}(t-kT)\right)W_2\big(\mathcal{L}(Z_{kT}^\lambda,V_{kT}^\lambda),\mathcal{L}(\zeta_{kT}^{\lambda,k-1},V_{kT}^{\lambda,k-1})\big)\\
		&\leq 4C_0\sum_{k=1}^n\exp\left(-\frac{m}{4\gamma}(n-k)\right)\left(\lambda\gamma+\sqrt{\frac{S_{\lambda,\gamma}}{\gamma}}\right)\\
		&\leq 4C_0\left(\lambda\gamma+\sqrt{\frac{S_{\lambda,\gamma}}{\gamma}}\right)\frac{1}{1-e^{-\frac{m}{4\gamma}}}\\
		&\leq \frac{16C_0\gamma}{m}\exp\left(\frac{m}{4\gamma}\right)\left(\lambda\gamma+\sqrt{\frac{S_{\lambda,\gamma}}{\gamma}}\right)\\
		&\leq \frac{16C_0e}{m}\left(\lambda\gamma^2+\sqrt{\gamma S_{\lambda,\gamma}}\right),
	\end{align*}
	where $C_0$ is given in Proposition~\ref{proposition: W_2 - interpolation and auxiliary}, with explicitly expression given in Table~\ref{tab.constants}, and where the third inequality holds due to $T\geq 1/(2\lambda)$ for $\lambda\leq\lambda_{\max,\gamma}\leq 1$, the fifth inequality holds due to $1-e^{-x}\geq xe^{-x}$ for all $x\geq 0$, and the last inequality holds due to $\gamma\geq\gamma_{\min}\geq 14m$.
\end{proof}

\phantomsection
\begin{proof}[\textbf{Proof of Corollary~\ref{coro.r_t pi}}]
	\label{proof.r_t pi}
	Let $(r_t,R_t)_{t\ge 0}$ be the solution of
	\eqref{eq.MY kinetic Langevin} with initial condition
	$(r_0,R_0)=(\theta_0,\nu_0)$. By definition of the scaled process \eqref{scaled process}, we have $(r_t^\lambda,R_t^\lambda)=(r_{\lambda t},R_{\lambda t})$. Let $(r_t',R_t')_{t\ge 0}$ be another solution of
	\eqref{eq.MY kinetic Langevin} initialized according to the invariant
	measure $\Pi_\beta^\epsilon$. Applying Theorem~\ref{theorem:contraction result of dalalyan} at time $\lambda t$
	yields
	\[
	W_2\bigl(\mathcal L(r_{\lambda t}),\mathcal L(r_{\lambda t}')\bigr)
	\le
	4\exp\left(-\frac{\lambda m}{2\gamma}t\right)
	W_2\bigl(\mathcal L(\theta_0,\nu_0),\Pi_\beta^\epsilon\bigr).
	\]
	Since $\Pi_\beta^\epsilon$ is invariant for
	\eqref{eq.MY kinetic Langevin}, its position marginal is
	$\pi_\beta^\epsilon$, and hence
	$\mathcal L(r_{\lambda t}')=\pi_\beta^\epsilon$. Therefore,
	\[
	W_2\bigl(\mathcal L(r_t^\lambda),\pi_\beta^\epsilon\bigr)
	\le
	4\exp\left(-\frac{\lambda m}{2\gamma}t\right)
	W_2\bigl(\mathcal L(\theta_0,\nu_0),\Pi_\beta^\epsilon\bigr),
	\]
	which proves the claim.
\end{proof}

\phantomsection
\begin{proof}[\textbf{Proof of Lemma~\ref{Lemma.sampling error}}]\label{Proof.sampling error}
	By Remark~\ref{remark.local Lip. of h}, we have, for all $\theta\in\mathbb{R}^d$, that
	\begin{align}\label{eq.upper bound of htheta}
		|h(\theta)|-|h(0)|\leq |h(\theta)-h(0)|\leq L_h(1+|\theta|)^{2r-1}|\theta|\leq L_h(1+|\theta|)^{2r},
	\end{align}
	which implies,
	\begin{align}\label{eq.upper bound of h}
		|h(\theta)|\leq |h(0)|+L_h(1+|\theta|)^{2r}\leq |h(0)|+L_h 2^{2r-1}(1+|\theta|^{2r}).
	\end{align}
	Let $(\Theta_n, \Theta)$ follow an optimal coupling of $\mathcal{L}(\theta_n^\lambda)$ and $\pi_\beta$, that is, $\mathcal{L}(\Theta_n)=\mathcal{L}(\theta_n^\lambda)$, $\mathcal{L}(\Theta)=\pi_\beta$ and
	\begin{equation}\label{eq.coupling}
		\left(\e[|\Theta_n-\Theta|^2]\right)^{1/2}=W_2(\mathcal{L}(\theta_n^\lambda),\pi_\beta).
	\end{equation}
	Then, we have, for all $\theta,\theta'\in\mathbb{R}^d$, that 
	\begin{align*}
		u(\theta)-u(\theta')=\int_0^1\la h(\theta'+t(\theta-\theta')),\theta-\theta'\ra\,\dd t.
	\end{align*}
	Therefore, using Cauchy-Schwarz inequality and \eqref{eq.upper bound of h}, we obtain
	\begin{align*}
		|u(\theta)-u(\theta')|&\leq \left|\int_0^1\la h(\theta'+t(\theta-\theta')),\theta-\theta'\ra\,\dd t\right|\\
		&\leq \int_0^1|h(\theta'+t(\theta-\theta'))|\,|\theta-\theta'|\,\dd t\\
		&\leq \int_0^1\left(|h(0)|+L_h2^{2r-1}(1+|t\theta+(1-t)\theta'|^{2r})\right)|\theta-\theta'|\,\dd t\\
		&\leq \left(|h(0)|+L_h 2^{2r-1}+L_h 2^{4r-2}(|\theta|^{2r}+|\theta'|^{2r})\right)|\theta-\theta'|.
	\end{align*}
	Taking expectations and applying the Cauchy-Schwarz inequality, the
	elementary bound $|a+b+c|^2\le 3(|a|^2+|b|^2+|c|^2)$ for $a,b,c\in\mathbb{R}$, together with \eqref{eq.coupling} and Lemmas
	\ref{lemma.moment bounds of pi_beta} and
	\ref{lemma.2q moment of algorithm}, we obtain
	\begin{align*}
		&\e[u(\theta_n^\lambda)]-\e[u(\theta_\infty)]\\
		&=\e[u(\Theta_n)]-\e[u(\Theta)]\\
		&\leq \e\left[\left(|h(0)|+L_h2^{2r-1}+L_h 2^{4r-2}(|\Theta_n|^{2r}+|\Theta|^{2r})\right)|\Theta_n-\Theta|\right]\\
		&\leq \left(\e\left[\left||h(0)|+L_h2^{2r-1}+L_h 2^{4r-2}(|\Theta_n|^{2r}+|\Theta|^{2r})\right|^2\right]\right)^{1/2}\left(\e[|\Theta_n-\Theta|^2]\right)^{1/2}\\
		&\leq \left(3(|h(0)|+L_h2^{2r-1})^2+3L_h^2 2^{8r-4}\e[|\theta_n^\lambda|^{4r}]+3L_h^2 2^{8r-4}\e[|\theta_\infty|^{4r}]\right)^{1/2}W_2(\mathcal{L}(\theta_n^\lambda),\pi_\beta)\\
		&\leq \left(3(|h(0)|+L_h2^{2r-1})^2+3L_h^2 2^{8r-4}\big(\bar C_{4r}+C_{\pi_\beta ,4r}\big)\right)^{1/2}W_2(\mathcal{L}(\theta_n^\lambda),\pi_\beta).
	\end{align*}
	Finally, applying Theorem~\ref{Main Theorem} yields
	\[\e[u(\theta_n^\lambda)]-\e[u(\theta_\infty)]\leq C'\left(\sqrt{\lambda}\gamma+\lambda\gamma^{5/2}+\gamma^{-1}+\exp\left(-\frac{\lambda m}{2\gamma}n\right)W_2(\mathcal{L}(\theta_0,\nu_0),\Pi_\beta)+\epsilon\right),\]
	where $C'\coloneq\big(3(|h(0)|+L_h2^{2r-1})^2+3L_h^2 2^{8r-4}(\bar C_{4r}+C_{\pi_\beta ,4r})\big)^{1/2} \dot{C}=\mathcal{O}((d/\beta)^{4r+1/2})$ with $\dot{C}$ given in Theorem~\ref{Main Theorem}.
\end{proof}

\phantomsection
\begin{proof}[\textbf{Proof of Lemma~\ref{Lemma.concentration property}}]\label{Proof.concentration property}
	By Remark~\ref{remark.dissipativity of h}, we have $\theta^*\in\bar{B}(0,R_0)$ with $R_0= \sqrt{2u(0)/m}$. We note that, for all $\theta,\theta'\in\bar{B}(0,R_0+1)$,
	\[|h(\theta)-h(\theta')|\leq L_h(1+2R_0)^{2r-1}|\theta-\theta'|.\] 
	Hence, by the convexity of $u$ as shown in Remark~\ref{remark.dissipativity of h} and the fact that $h(\theta^*)=0$, we have
	\begin{align}
		\e[u(\theta_\infty)\II_{\bar{B}(0,R_0+1)}(\theta_\infty)]-u(\theta^*)&\leq \int_{\bar{B}(0,R_0+1)} (u(\theta)-u(\theta^*))\,\pi_\beta(\dd \theta)\notag\\
		&\leq \int_{\bar{B}(0,R_0+1)} \la h(\theta)-h(\theta^*),\theta-\theta^*\ra\,\pi_\beta(\dd \theta)\notag\\
		&\leq L_h(1+2R_0)^{2r-1}\int_{\bar{B}(0,R_0+1)}|\theta-\theta^*|^2\,\pi_\beta(\dd\theta)\notag\\
		&\leq L_h(1+2R_0)^{2r-1}\e[|\theta_\infty-\theta^*|^2\II_{\bar{B}(0,R_0+1)}]\notag\\
		&\leq L_h(1+2R_0)^{2r-1}\e[|\theta_\infty-\theta^*|^2].\label{eq.first part of u(invariant)}
	\end{align}
	Moreover, since $\theta^*\in\bar{B}(0,R_0)$, we have
	\[\p(\theta_\infty\notin\bar{B}(0,R_0+1))\leq \p(|\theta_\infty-\theta^*|>1)\leq \e[|\theta_\infty-\theta^*|^2],\]
	where the last inequality follows from the Chebyshev's inequality. Applying Cauchy-Schwarz inequality yields
	\begin{align}
		\e[u(\theta_\infty)\II_{{\bar{B}(0,R_0+1)}^c}(\theta_\infty)]&\leq \left(\e[u(\theta_\infty)^2]\right)^{1/2}\left(\p(\theta_\infty\in {\bar{B}(0,R_0+1)}^c)\right)^{1/2}\notag\\
		&\leq \left(\e[u(\theta_\infty)^2]\right)^{1/2}\left(\e[|\theta_\infty-\theta^*|^2]\right)^{1/2}.\label{eq.second part of u(invariant)}
	\end{align}
	Next, for any $\theta\in\mathbb{R}^d$, by the fundamental theorem of calculus, Cauchy-Schwarz inequality, and \eqref{eq.upper bound of htheta}, we have that
	\begin{align*}
		|u(\theta)-u(0)|=\left|\int_0^1\la h(t\theta),\theta\ra\,\dd t\right|&\leq \int_0^1|h(t\theta)|\, |\theta|\,\dd t\\
		&\leq |\theta|\int_0^1 \left(|h(0)|+L_h(1+|t\theta|)^{2r}\right)\,\dd t\\
		&\leq (|h(0)|+L_h)(1+|\theta|)^{2r+1}.
	\end{align*}
	Therefore, we obtain that
	\begin{align*}
		|u(\theta)|^2&\leq 2|u(\theta)-u(0)|^2+2|u(0)|^2\\
		&\leq 2(|h(0)|+L_h)^2(1+|\theta|)^{4r+2}+2|u(0)|^2\\
		&\leq 2^{4r+3}(|h(0)|^2+L_h^2)(1+|\theta|^{4r+2})+2|u(0)|^2
	\end{align*}
	Therefore, by Lemma~\ref{lemma.moment bounds of pi_beta}, the second moment of $u(\theta_\infty)$ satisfies
	\begin{equation}\label{eq.moment bounds of u(invariant)}
		\e[|u(\theta_\infty)|^2]\leq 2^{4r+3}(|h(0)|^2+L_h^2)(1+C_{\pi_\beta ,4r+2})+2|u(0)|^2.
	\end{equation}
	Combining \eqref{eq.first part of u(invariant)}, \eqref{eq.second part of u(invariant)} and \eqref{eq.moment bounds of u(invariant)}, we conclude that
	\[\e[u(\theta_\infty)]-u(\theta^*)\leq C''\e[|\theta_\infty-\theta^*|^2],\]
	where $C''\coloneq L_h(1+2R_0)^{2r-1}+(2^{4r+3}(|h(0)|^2+L_h^2)(1+C_{\pi_\beta ,4r+2})+2u(0)^2)^{1/2}=\mathcal{O}((d/\beta)^{r+1/2})$.
\end{proof}

\newpage

\section{Analytic Expression of Constants}

{\small
	\begin{xltabular}{\textwidth}{
			>{\raggedright\arraybackslash}p{1.8cm} 
			>{\centering\arraybackslash}p{0.8cm}
			>{\raggedright\arraybackslash}X
			>{\raggedright\arraybackslash}p{2.5cm}}
		
		\caption{Explicit expressions of constants.} \label{tab.constants} \\
		\toprule
		\multicolumn{2}{c}{{Constant}} & \multicolumn{1}{c}{{Explicit expression}} & \multicolumn{1}{c}{{Dependence}} \\
		\midrule
		\endfirsthead
		
		\multicolumn{4}{c}{{Table \thetable{} (Continued): Explicit expressions of constants.}} \\[0.5em]
		\toprule
		\multicolumn{2}{c}{{Constant}} & \multicolumn{1}{c}{{Explicit expression}} & \multicolumn{1}{c}{{Dependence}} \\
		\midrule
		\endhead
		
		\bottomrule
		\multicolumn{4}{r}{\small\itshape Continued on next page...} \\
		\endfoot
		
		\bottomrule
		\endlastfoot
		
		Remark~\ref{remark.local Lip. of h} & $K_H$ & $2^{2r-1}K_G+K_F$ & - \\
		\cmidrule{2-4}
		& $L_h$ & $L_G+L_F\e[(1+2|X_0|)^{\rho-1}]$ & -\\
		\midrule
		
		Remark~\ref{remark.dissipativity of h} & $R_0$ & $\sqrt{2u(0)/m}$ & - \\
		\midrule
		
		Theorem~\ref{Main Theorem}
		& $\dot{C}$ & $2\max\{C_0,16C_0e/m\}(1+\sqrt{C_1}+\sqrt{C_2}+\sqrt{C_3}+\sqrt{C_4})+4+5\bar{c}$ & $\mathcal{O}((d/\beta)^{3r+1/2})$ \\
		\midrule
		
		Theorem~\ref{Theorem.Optimization}
		& $C'$ & $(3(|h(0)|+L_h2^{2r-1})^2+3L_h^2 2^{8r-4}(\bar C_{4r}+C_{\pi_\beta ,4r}))^{1/2} \dot{C}$ & $\mathcal{O}((d/\beta)^{4r+1/2})$ \\
		\cmidrule{2-4}
		& $C''$ & $L_h(1+2R_0)^{2r-1}+(2^{4r+3}(|h(0)|^2+L_h^2)(1+C_{\pi_\beta ,4r+2})+2u(0)^2)^{1/2}$ & $\mathcal{O}((d/\beta)^{r+1/2})$ \\
		\midrule
		Remark~\ref{remark.precision-expected} & $\bar{C}''$ & $L_h(1+2R_0)^{2r-1}+(2^{4r+3}(|h(0)|^2+L_h^2)(1+(4u(0)/m+4(d+4r)/m)^{2r+1})+2u(0)^2)^{1/2}$ & $\mathcal{O}(d^{r+1/2})$ \\
		\midrule
		Lemma~\ref{lemma.moment bounds of pi_beta} 
		& $C_{\pi_\beta,2q}$ & $(4u(0)/m+4(d+2(q-1))/(\beta m))^q$ & $\mathcal{O}((d/\beta)^{q})$\\
		\midrule
		
		Lemma~\ref{lemma.contraction of pi_beta}
		& $\bar c$ & $\sqrt{2c_\pi/m}$ & $\mathcal{O}((d/\beta)^{2r-1/2})$ \\
		\cmidrule{2-4}
		& $c_\pi$ & $2^{12r-4}L_h^4\left((1+R_0)^{8r-2}+C_{\pi_\beta,8r-2}\right)/m$ & $\mathcal{O}((d/\beta)^{4r-1})$ \\
		\midrule
		
		Lemma~\ref{lemma.2 moment of algorithm} & $\bar{C}_2$ & $6\e[|\theta_0|^2]+8\e[|\nu_0|^2]+86\big(7C_H/8+u(0)/2+d/\beta\big)/m$ & $\mathcal{O}(d/\beta)$ \\
		\cmidrule{2-4}
		& $\bar{B}_2$ & $\e[|\nu_0|^2]+9m^2\bar{C}_2+9C_H/2+4d/\beta$ & $\mathcal{O}(d/\beta)$ \\
		\midrule
		
		Lemma~\ref{lemma.2q moment of algorithm}
		& $\bar{C}_K$ & $7C_H/8+u(0)/2+d/\beta$ & $\mathcal{O}(d/\beta)$ \\
		\cmidrule{2-4}
		($q\geq 2$ and & $C_m$ & $5m/8+7m^2/4$ & - \\
		\cmidrule{2-4}
		$q\in\mathbb{N}$)& $C_f$ & $59(1+m^2)$ & - \\
		\cmidrule{2-4}
		& $c_f$ & $6d(d+2)/\beta^2+4C_Hd/\beta+12(1/(8m)+7)^2(m^4+K_H^4\e[(1+|X_0|)^{4\rho}])$ & $\mathcal{O}((d/\beta)^2)$ \\
		\cmidrule{2-4}
		& $K_f$ & $2^{3q-2}((d+q)/\beta)^{q/2}(C_f^{q/2}+m^q+K_H^q\e[(1+|X_0|)^{\rho q}])$ & $\mathcal{O}((d/\beta)^{q/2})$ \\
		\cmidrule{2-4}
		& $K_g$ & $2^{q-1}((d+2q)/\beta)^q+2^{2q-2}(1/(8m)+7)^q(m^{2q}+K_H^{2q}\e[(1+|X_0|)^{2\rho q}])$ & $\mathcal{O}((d/\beta)^q)$ \\
		\cmidrule{2-4}
		& $C_K$ & $3^{q-1}(K_f+K_g)$ & $\mathcal{O}((d/\beta)^q)$ \\
		\cmidrule{2-4}
		& $N_1$ & $(64/m)^{q-1}(q \bar{C}_K+3q(q-1)2^{q-3}C_f\beta^{-1}d)^{q}$ & $\mathcal{O}((d/\beta)^q)$ \\
		\cmidrule{2-4}
		& $N_2$ & $(8/\sqrt{m})^{q-2}(q(q-1)2^{q-3}c_f)^{q/2}$ & $\mathcal{O}((d/\beta)^q)$ \\
		\cmidrule{2-4}
		& $N_3$ & $64(q(q-1)2^{q-3}3^{q-1}K_f)^2/m$ & $\mathcal{O}((d/\beta)^q)$ \\
		\cmidrule{2-4}
		& $N_0$ & $N_1+N_2+N_3+q(q-1)2^{q-3}C_K$ & $\mathcal{O}((d/\beta)^q)$ \\
		\cmidrule{2-4}
		& $\bar{C}_{2q}$ & $2^{5q-2}\e[|\theta_0|^{2q}]+2^{5q-1}\e[|\nu_0|^{2q}]+2^{4q+6}N_0/m$ & $\mathcal{O}((d/\beta)^q)$ \\
		\midrule
		
		Lemma~\ref{lemma.mse of interpolation} 
		& $C_{1,v}$ & $3\bar{B}_2/4+3m^2\bar{C}_2/2+3C_H/4+6d/\beta$ & $\mathcal{O}(d/\beta)$\\
		\midrule
		
		Lemma~\ref{lemma.2nd bound MY kinetic}
		& $A_c$ & $(5u(0)+|h(0)|)\beta/8$ & $\mathcal{O}(\beta)$ \\
		\cmidrule{2-4}
		& $C_r$ & $11\e[u(\theta_0)]+6\e[|\theta_0|^2]+8\e[|\nu_0|^2]+86(d+A_c)/(\beta m)$ & $\mathcal{O}(d/\beta)$ \\
		\cmidrule{2-4}
		& $C_R$ & $\e[|\nu_0|^2]+2|h(0)|^2+2C_r+2d/\beta$ & $\mathcal{O}(d/\beta)$\\
		\midrule
		
		Lemma~\ref{lemma.2nd bound of auxiliary process}
		& $C_\zeta^\#$ & $11\e[u(\theta_0)]+6\e[|\theta_0|^2]+8\e[|\nu_0|^2]+430(d+A_c)/(\beta m)$ & $\mathcal{O}(d/\beta)$ \\
		\cmidrule{2-4}
		& $C_V^\#$ & $\bar{B}_2+2|h(0)|^2+2C_\zeta^\#+2d/\beta$ & $\mathcal{O}(d/\beta)$\\
		\midrule
		Lemma~\ref{lemma.2nd bound of taming factor} & $C_H$ & $8\left(m^2+K_H^2\e[(1+|X_0|)^{2\rho}]\right)$ & - \\
		\midrule
		
		Proposition~\ref{proposition: W_2 - interpolation and auxiliary}
		& $c_A$ & $(2K_H^2\e[(1+|X_0|)^{2\rho}]+2^{4r}L_h^2)(1+\bar{C}_{4r})+2|h(0)|^2$ & $\mathcal{O}((d/\beta)^{2r})$ \\
		\cmidrule{2-4}
		& $c_D$ & $36(2^{4r}L_h^2(1+\bar{C}_{4r})+2|h(0)|^2)$ & $\mathcal{O}((d/\beta)^{2r})$ \\
		& $C_1$ & $4C_{1,v}/9+24d/\beta+18\bar{B}_2+36m^2\bar{C}_2+18C_H+2C_V^\#+4C_\zeta^\#+4|h(0)|^2+2c_A$ & $\mathcal{O}((d/\beta)^{2r})$ \\
		\cmidrule{2-4}
		& $C_2$ & $12m^2\bar{C}_2+6C_H+12C_\zeta^\#+20|h(0)|^2+6c_A+c_D+2^{4r+2}L_h^2(1+\bar{C}_{4r})+9\bar{B}_2$ & $\mathcal{O}((d/\beta)^{2r})$ \\
		\cmidrule{2-4}
		& $C_3$ & $18(4K_H^2\e[(1+|X_0|)^{2\rho}]+m^2)(1+\bar{C}_{12r})$ & $\mathcal{O}((d/\beta)^{6r})$ \\
		\cmidrule{2-4}
		& $C_4$ & $9\cdot 2^{12r-5}L_h^4((1+R_0)^{8r-2}+\bar{C}_{8r-2})$ & $\mathcal{O}((d/\beta)^{4r-1})$ \\
		\cmidrule{2-4}
		& $C_5$ & $8C_V^\#+16C_\zeta^\#+16|h(0)|^2+2C_{1,v}+2d/\beta$ & $\mathcal{O}(d/\beta)$ \\
		\cmidrule{2-4}
		& $C_0$ & $\sqrt{26(C_5+1)}$ & $\mathcal{O}((d/\beta)^{1/2})$ \\
		\cmidrule{2-4}
		& $C_J$ & $3C_\zeta^\#+3\bar{C}_2+3\bar{B}_2$ & $\mathcal{O}(d/\beta)$ \\
		\cmidrule{2-4}
		& $C_M$ & $3C_V^\#+3\bar{B}_2+3C_{1,v}$ & $\mathcal{O}(d/\beta)$ \\
\end{xltabular}}


\end{document}